\documentclass[a4paper,11pt, reqno]{amsart}
\usepackage{color}
\usepackage[normalem]{ulem}

\usepackage{enumerate}
\usepackage{amssymb, amsmath}
\usepackage{mathrsfs}%,dsfont}
\usepackage{amscd}
\usepackage[active]{srcltx}
\usepackage{verbatim}
\usepackage[colorlinks,linkcolor={black},citecolor={blue},urlcolor={black}]{hyperref}
\theoremstyle{plain}
\newtheorem{theorem}{Theorem}[section]
\newtheorem{corollary}[theorem]{Corollary}
\newtheorem{lemma}[theorem]{Lemma}
\newtheorem{proposition}[theorem]{Proposition}

\newtheorem{definition}[theorem]{Definition}
\newtheorem{assumption}[theorem]{Assumption}

\newtheorem*{definition*}{Definition}
\newtheorem{introtheorem}{Theorem}

\theoremstyle{remark}
\newtheorem{remark}[theorem]{Remark}
\newtheorem{example}[theorem]{Example}

\newtheorem*{claim*}{Claim}
\newtheorem*{remark*}{Remark}
\newtheorem*{example*}{Example}
\newtheorem*{notation*}{Notation}

\numberwithin{equation}{section}

\def\N{{\mathbb N}}

\def\R{{\mathbb R}}

\def\D{{\mathbb D}}
\def\Q{{\mathbb Q}}

\newcommand{\one}{{{\bf 1}}}

\renewcommand{\a}{\alpha}
\renewcommand{\b}{\beta}

\renewcommand{\d}{\delta}
\newcommand{\eps}{\varepsilon}
\renewcommand{\phi}{\varphi}

\newcommand{\dd}{\mathrm{d}}

\DeclareMathOperator{\supp}{supp}

\newcommand{\ip}[1]{\langle {#1}\rangle}

\newcommand{\norm}[1]{\| {#1}\|}
\newcommand{\abs}[1]{\vert {#1}\vert}

\DeclareMathOperator{\Ric}{Ric}
\DeclareMathOperator{\Lip}{Lip}
\DeclareMathOperator{\Hess}{Hess}

\DeclareMathOperator{\cd}{CD}
\DeclareMathOperator{\rcd}{RCD}
\DeclareMathOperator{\evi}{EVI}
\DeclareMathOperator{\geo}{Geo}
\DeclareMathOperator{\ch}{Ch}
\DeclareMathOperator{\bH}{H}
\DeclareMathOperator{\ent}{Ent}
\DeclareMathOperator{\bl}{BL}
\DeclareMathOperator{\be}{BE}

\newcommand{\ddt}{\frac{\mathrm{d}}{\mathrm{d}t}}
\newcommand{\ddtr}{\frac{\mathrm{d}^+}{\mathrm{d}t}}

\newcommand{\cH}{\mathscr{H}}

\newcommand{\cL}{\mathcal{L}}

\newcommand{\cE}{\mathcal{E}}

\newcommand{\cP}{\mathscr{P}}

\renewcommand{\tilde}{\widetilde}

\newcommand{\e}{\mathrm{e}}

\newcommand{\ecdkn}{{\cd^e(K,N)}}
\newcommand{\cdskn}{{\cd^*(K,N)}}
\newcommand{\cdkn}{{\cd(K,N)}}
\newcommand{\rcdkn}{{\rcd^*(K,N)}}
\newcommand{\evikn}{{\evi_{K,N}}}

\newcommand{\fs}{{\frak{s}}}
\newcommand{\fc}{{\frak{c}}}
\newcommand{\ckn}[1]{\fc_{K/N}\big(#1\big)}

\newcommand{\skn}[1]{\fs_{K/N}\left(#1\right)}

\newcommand{\sigkn}[2]{\sigma^{(#1)}_{K/N}\big(#2\big)}
\newcommand{\sigknp}[2]{\sigma^{(#1)}_{K/N'}\big(#2\big)}
\newcommand{\wug}[1]{\abs{\nabla #1}_w}

\begin{document}

\title[Curvature-Dimension Condition and Bochner's Inequality]{On the
  Equivalence of the Entropic Curvature-Dimension Condition and
  Bochner's Inequality on Metric Measure Spaces}

 \author{Matthias Erbar\quad*\quad
 Kazumasa Kuwada\quad*\quad
Karl-Theodor Sturm}
 \address{
 University of Bonn\\
 Institute for Applied Mathematics\\
 Endenicher Allee 60\\
 53115 Bonn\\
 Germany}
 \email{erbar@iam.uni-bonn.de}
 \email{kuwada.kazumasa@ocha.ac.jp}
 \email{sturm@uni-bonn.de}

\thanks{}
%\keywords{}

%\subjclass[2000]{Primary 60H07; Secondary: 35J15, 35K90, 35R15,
%  47A60, 47B44, 47D05, 47F05, 60H30, 60G15}

 %\begin{abstract}
 %\end{abstract}

%\dedicatory{}

\date\today

\maketitle

\begin{abstract}
We prove the equivalence of the
curvature-dimension bounds of Lott-Sturm-Villani (via entropy and optimal transport) and of
Bakry--\'Emery (via energy and $\Gamma_2$-calculus)  in complete generality for infinitesimally Hilbertian metric measure spaces. In particular, we establish the full Bochner
inequality on such metric measure spaces.
Moreover, we deduce new contraction bounds for the heat flow on Riemannian manifolds and on mms in terms of the $L^2$-Wasserstein distance. 
\end{abstract}

\tableofcontents

\section{Introduction}\label{sec:intro}

Bochner's inequality is one of the most fundamental estimates in
geometric analysis. It states that
\begin{equation}\label{eq:intro-bochner}
\frac12\Delta|\nabla u|^2-\langle\nabla u, \nabla\Delta u\rangle\ge K \cdot |\nabla u|^2+\frac1N \cdot |\Delta u|^2
\end{equation}
for each smooth function $u$ on a Riemannian manifold $(M,g)$ provided
$K\in\R$ is a lower bound for the Ricci curvature on $M$ and $N\in
(0,\infty]$ is an upper bound for the dimension of $M$.  The main
results of this paper is an analogous Bochner inequality on metric
measure spaces $(X,d,m)$ with linear heat flow and satisfying the
(reduced) curvature-dimension condition.  Indeed, we will also prove
the converse: if the heat flow on a mms $(X,d,m)$ is linear then an
appropriate version of \eqref{eq:intro-bochner} (for the canonical
gradient and Laplacian on $X$) will imply the reduced
curvature-dimension condition.  Besides that, we also derive new,
sharp $W_2$-contraction results for the heat flow as well as pointwise
gradient estimates and prove that each of them is equivalent to the
curvature-dimension condition. That way, we obtain a complete
one-to-one correspondence between the Eulerian picture captured in the
Bochner inequality and the Lagrangian interpretation captured in the
curvature-dimension inequality.

\bigskip

The {\bf curvature-dimension condition} $\cdkn$ was introduced by
Sturm in \cite{S06}. It was later adopted and slightly modified by
Lott \& Villani, see also the elaborate presentation in the monograph
\cite{Vil09}. The $\cdkn$-condition for finite $N$ is a sophisticated
tightening up of the much simpler $\cd(K,\infty)$-condition introduced
as a synthetic Ricci bound for metric measure spaces independently by
Sturm \cite{S06} and Lott \& Villani \cite{LV09}. From the very
beginning, a disadvantage of the $\cdkn$-condition for finite $N$ was
the lack of a local-to-global result. To overcome this drawback,
Bacher \& Sturm \cite{BS10} introduced the \emph{reduced
  curvature-dimension condition} $\cdskn$ which has a local-to-global
property and which is equivalent to the local version of $\cdkn$. The
curvature-dimension condition $\cdkn$ has been verified for Riemannian
manifolds \cite{S06}, Finsler spaces \cite{Oht09}, Alexandrov spaces
\cite{Pet10}, \cite{ZZ10}, cones \cite{BS11} and warped products of
Riemannian manifolds \cite{Ket12}.  Actually, in all these cases the
conditions $\cdkn$ and $\cdskn$ turned out to be equivalent.

\bigskip

A completely different approach to generalized curvature-dimension
bounds was set forth in the pioneering work of Bakry and \'Emery
\cite{BE85}. It applies to the general setting of Dirichlet forms and
the associated Markov semigroups and is formulated using the
(iterated) \emph{carr\'e du champ} operators built from the generator
of the semigroup. This \textbf{energetic curvature-dimension
  condition} $\be(K,N)$ has proven a powerful tool in particular in
infinite dimensional situations. It yields hypercontractivity of the
semigroup and has successfully been used to derive functional
inequalities like the logarithmic Sobolev inequalities in a variety of
examples. Among the remarkable analytic consequences of the
Bakry--\'Emery condition $\be(K,\infty)$ we single out the point-wise
gradient estimates for the semigroup $H_t$. It implies that for any
$f$ in a large class of functions
\begin{align*}
\Gamma(H_t f) ~\le~ \e^{-2Kt}\, H_t\Gamma(f)\;,
\end{align*}
where $\Gamma$ is the carr\'e du champ operator.

\bigskip

The relation between the two notions of curvature bounds based on
optimal transport and Dirichlet forms has been studied in large
generality by Ambrosio, Gigli and Savar\'e in a series of recent works
\cite{AGS11b,AGS12}, see also \cite{AGMR12}. The key tool of their
analysis is a powerful calculus on metric measure spaces which allows
them to match the two settings. Starting from a metric measure
structure they introduce the so called Cheeger energy which takes over
the role of the 'standard' Dirichlet energy and is obtained by
relaxing the $L^2$-norm of the slope of Lipschitz functions. A key
result is the identification of the $L^2$-gradient flow of the Cheeger
energy with the Wasserstein gradient flow of the entropy. This is the
mms equivalent of the famous result by Jordan--Kinderlehrer--Otto
\cite{JKO98} and allows one to define unambiguously a heat flow in
metric measure spaces.

We say that a metric measure space is \emph{infinitesimally
  Hilbertian} if the heat flow is linear. This is equivalent to the
Cheeger energy being the associated Dirichlet form. We denote its
domain by $W^{1,2}$. Under the assumption of linearity of the heat
flow, Ambrosio--Gigli--Savar\'e prove that $\cd(K,\infty)$ implies
$\be(K,\infty)$ and the converse also holds under an additional
regularity assumption. Combining linearity of the heat flow with the
$\cd(K,\infty)$ condition leads to the \textbf{Riemannian curvature
  condition} $\rcd(K,\infty)$ introduced in \cite{AGS11b}. This
concept again turns out to be stable under Gromov--Hausdorff
convergence and tensorization.

\bigskip

Recently, also {\bf Bochner's inequality} has been extended to
singular spaces. Ohta \& Sturm \cite{OS11} proved it for Finsler
spaces and Gigli, Kuwada \& Ohta \cite{GKO10} and Zhang \& Zhu
\cite{ZZ12} for Alexandrov spaces. Finally, Ambrosio, Gigli \&
Savar\'e established the Bochner inequality without the dimension term
(i.e. with $N=\infty$) in $\rcd(K,\infty)$ spaces. However, in the
classical setting, the full strength of Bochner's inequality only
comes to play if also the dimension effect is taken into account,
i.e. with finite $N$. This can be seen for example from the famous
results of Li--Yau \cite{LY86} who derive from it a differential
Harnack inequality, eigenvalue estimates for the Laplacian and
Gaussian heat kernel bounds.

\bigskip

We prove the equivalence of
curvature-dimension bounds via optimal transport and via the
Bakry--\'Emery approach in full generality for infinitesimally Hilbertian metric measure spaces. In particular, we establish the full Bochner
inequality on such metric measure spaces.

\medskip

Our approach strongly relies on properties and consequences of a new
curvature-dimension condition, the so-called {\bf entropic curvature
  dimension condition} $\ecdkn$. It simply states that the Boltzmann
entropy $\ent$ is $(K,N)$-convex on the Wasserstein space
$\cP_2(X,d)$.  Here a function $s$ on an interval $I\subset\R$ is
called $(K,N)$-convex if
\begin{equation}
s''~\ge~ K+\frac1N\cdot (s')^2\;.
\end{equation}
holds in distribution sense. A function $S$ on a geodesic space is
called $(K,N)$-convex if it is $(K,N)$-convex along each unit speed
geodesic -- or at least along each curve within a class of unit speed geodesics which connect each pair of points in $X$. This way, $(K,N)$-convexity is a weak formulation of
\begin{equation}
  \Hess S~ \ge~ K+\frac{1}{N}\big(\nabla S \otimes\nabla S \big)\;.
\end{equation}

Our first result is the following

\begin{introtheorem}[Theorem~\ref{thm:equiv_as}]
  For a essentially non-branching mms
  (see Definition~\ref{def:ess-non-branch})
  the entropic curvature-dimension condition $\ecdkn$ is equivalent
  to the reduced curvature-dimension condition $\cdskn$.
\end{introtheorem}

We say that a metric measure space satisfies the \textbf{Riemannian
  curvature-dimension condition} $\rcdkn$ if it is infinitesimally
Hilbertian and satisfies $\ecdkn$ or $\cdskn$. This notion turns out
to have the natural stability properties. Namely, we prove (see
Theorems~\ref{thm:rcd-stable}, \ref{thm:rcd-tensor},
\ref{thm:rcd-locglob}) that the $\rcdkn$ condition is preserved under
measured Gromov--Hausdorff convergence as well as under tensorization of metric measure
spaces and that it has a local--to--global property.

\bigskip

The geometric intuition coming from the analysis of $(K,N)$-convex
functions and their gradient flows leads to a new form of the {\bf
  Evolution Variation Inequality} $\evikn$ on the Wasserstein space
taking into account also the effect of the dimension bound.
Until now, the notion of $\evikn$ gradient flow was known only without
dimension term (i.e. with $N=\infty$). These Evolution Variational
Inequalities first appeared in the setting of Hilbert spaces where
they characterize uniquely the gradient flows of $K$-convex
functionals.
In a general metric setting and in connection with optimal transport
these inequalities have been extensively studied in \cite{OW06, DS08,
  AGS11b}. In particular, it turned out that $\rcd(K,\infty)$ spaces
can be characterized by the fact that the heat flow is an
$\evi_{K,\infty}$ gradient flow of the entropy.
Here we obtain a reinforcement of this result. Namely, the new
Riemannian curvature-dimension condition $\rcdkn$ is equivalent to the
existence of an $\evikn$ gradient flow of the entropy in the following
sense.

\begin{introtheorem}[Definition~\ref{def:KNflow}, Theorem~\ref{thm:RCDKN-equiv}]
  A mms $(X,d,m)$ satisfies $\rcdkn$ if and only if $(X,d)$ is a
  length space, $m$ satisfies an integrability condition
  \eqref{eq:exp-int} and every $\mu_0\in\cP_2(X,d)$ is the starting
  point of a curve $(\mu_t)_{t\ge0}$ in $\cP_2(X,d)$ such that for any
  other $\nu\in\cP_2(X,d)$ and a.e. $t>0$:
\begin{align}
    \ddt\skn{\frac12 W_2(\mu_t,\nu)}^2 + K\cdot\skn{\frac12 W_2(\mu_t,\nu)}^2~\leq~\frac{N}{2}\left(1-\frac{U_{N}(\nu)}{U_{N}(\mu_t)}\right).
    %\ddt\left[e^{Kt}\sknp{\frac12 d(x_t,z)}^2\right] ~\leq~e^{Kt}\frac{N}{2}\left(1-\frac{U_{N'}(z)}{U_{N'}(x_t)}\right)
  \end{align}
  Here $U_N(\mu)=\exp\Big(-\frac1N \ent(\mu)\Big)$ and
  $\fs_\kappa(r)=\sqrt{1/\kappa}\sin\big(\sqrt{\kappa}r\big)$ provided
  $\kappa>0$ and $\fs_{\kappa}(r)=\sqrt{1/(-\kappa)}\sinh
  \big(\sqrt{-\kappa}r\big),\ \fs_0(r)=r$ for $\kappa<0$
  resp. $\kappa=0$.
\end{introtheorem}

This curve is unique, in fact, it is the heat flow which we denote in
the following by $\mu_t=H_t\mu_0$.
\medskip

The Evolution Variation Inequality $\evikn$ as stated above
immediately implies new, sharp contraction estimates (or, more
precisely, expansion bounds) in Wasserstein metric for the heat flow.

\begin{introtheorem}[Theorem~\ref{thm:contraction}, Theorem~\ref{thm:W2-contraction} and Proposition 2.12]
  Let $(X,d,m)$ be a $\rcdkn$
  space. Then for any $\mu,\nu\in\cP_2(X,d)$ and $s,t>0$:
  \begin{align}\label{eq:intro-contraction}
    \skn{\frac12
      W_2(H_t\mu,H_s\nu)}^2~\leq~&e^{-K(s+t)}\cdot\skn{\frac12
      W_2(\mu,\nu)}^2%\\\nonumber &
   + \frac{N}{K}\Big(1-e^{-K(s+t)}\Big)\frac{\big(\sqrt{t}-\sqrt{s}\big)^2}{2(s+t)}\;.
  \end{align}
The latter implies the slightly weaker bound
\begin{align*}
W_2( H_t\mu,H_s\nu )^2
&~\le~
\e^{-K \tau(s,t)} \cdot W_2(\mu,\nu)^2 + 2N
\frac{ 1 - \e^{-K \tau(s,t)}}{ K \tau(s,t) }
\big( \sqrt{t} - \sqrt{s} \big)^2 \; ,
\end{align*}
where $\tau (s,t) = 2( t + \sqrt{ts} + s )/3$.
In the particular case $t=s$ this reduces to the well-known estimate 
  $W_2( H_t\mu,H_t\nu )  \le
    \e^{-Kt}\cdot W_2(\mu,\nu)$.
\end{introtheorem}

\bigskip

Due to the work of Kuwada \cite{Kuw10}, it is well known that
$W_2$-expansion bounds are intimately related to pointwise gradient
estimates. The next result is a particular case of a more general
equivalence that will be the subject of a forthcoming publication
\cite{Kuw13}.

\begin{introtheorem}[Theorem~\ref{thm:grad-est}]
  Assume that the mms $(X,d,m)$ is infinitesimally Hilbertian and
  satisfies a regularity assumption (Assumption~\ref{ass:Ch-reg}).  If
  the $W_2$-expansion bound \eqref{eq:intro-contraction} holds then
  for any $f$ of finite Cheeger energy:
  \begin{align}\label{eq:intro-grad-est}
    \wug{\bH_t f}^2 + \frac{4Kt^2}{N\big(e^{2Kt}-1\big)}\abs{\Delta \bH_t f}^2~\leq~e^{-2Kt}\bH_t\big(\wug{f}^2\big) \quad \mbox{$m$-a.e.}
  \end{align}
\end{introtheorem}

Note that Assumption~\ref{ass:Ch-reg} is the same as what is assumed
in \cite{AGS12} and it is always satisfied if $(X,d,m)$ is
$\rcd(K',\infty)$ for any $K'\in \R$. Hence, Theorem 3 and Theorem 4
imply in particular that \eqref{eq:intro-grad-est} holds on a $\rcdkn$
space. Here $\wug{f}$ denotes the weak upper gradient of $f$
introduced in \cite{AGS11a}. This kind of gradient estimate has first
been established by Bakry and Ledoux \cite{BL06} in the setting of
$\Gamma$-calculus. It is new in the framework of metric measure spaces
and allows us to establish the Bochner formula for the canonical
gradients and Laplacians on mms.

\begin{introtheorem}[Theorem~\ref{thm:Bochner}]
  Assume that the mms $(X,d,m)$ is infinitesimally Hilbertian and
  satisfies the gradient estimate \eqref{eq:intro-grad-est}. Then for
  all $f\in D(\Delta)$ with $\Delta f\in W^{1,2}(X,d,m)$ and all $g\in
  D(\Delta)$ bounded and non-negative with $\Delta g\in L^\infty(X,m)$
  we have
  \begin{align}\label{eq:intro-rough-bochner}
    \frac12\int\Delta g |\nabla f|_w^2 \dd m - \int
    g\langle\nabla(\Delta f),\nabla f\rangle \dd m \geq K\int g
    |\nabla f|_w^2\dd m + \frac{1}{N}\int g\big(\Delta f\big)^2\dd
    m\;.
  \end{align}
\end{introtheorem}

\bigskip

\begin{introtheorem}[Proposition~\ref{prop:Bochner2BEW}, Theorem~\ref{thm:BEW2CDE}]
  Assume that the mms $(X,d,m)$ is infinitesimally Hilbertian and
  satisfies Assumption~\ref{ass:Ch-reg}.
  Then the Bochner inequality
  $\be(K,N)$ \eqref{eq:intro-rough-bochner} implies the entropic
  curvature-dimension condition $\ecdkn$.
\end{introtheorem}

Thus we have closed the circle. All the previous key properties are
equivalent to each other, at least if we require the heat flow to be
linear.

\begin{introtheorem}[Summary]
  Let $(X,d,m)$ be an infinitesimally Hilbertian metric measure space.
  Then the following properties are equivalent:
\begin{itemize}
\item[(i)] $\cdskn$,
\item[(ii)] $\ecdkn$,
\item[(iii)]$(X,d)$ is a length space, \eqref{eq:exp-int} and the
  existence of the $\evikn$ gradient flow of the entropy starting from
  every $\mu\in\cP_2(X,d)$.
\end{itemize}
If one of them is satisfied, we obtain the following:
\begin{itemize}
\item[(iv)] The $W_2$-expansion bound \eqref{eq:intro-contraction},
\item[(v)] The Bakry--Ledoux pointwise gradient estimate $\bl(K,N)$
  \eqref{eq:intro-grad-est},
\item[(vi)] The Bochner inequality $\be(K,N)$ \eqref{eq:intro-rough-bochner}.
\end{itemize}
Moreover, under Assumption~\ref{ass:Ch-reg},
all of properties {\rm (i)--(vi)} are equivalent.
\end{introtheorem}

\bigskip

\begin{remark*}
Finally, let us point out -- on a more heuristic level -- two remarkable links between $(K,N)$-convexity and the Bakry-\'Emery condition $\be(K,N)$:
\begin{itemize}
\item[(I)] The $(K,N)$-convexity of a function $V$ on a Riemannian manifold $(M,g)$  can be interpreted as the $BE(K,N)$-condition for the re-scaled drift diffusion
\begin{equation}\label{SDE}
dX_t=\sqrt{2\alpha}\, dB_t-\nabla V(X_t)\,dt
\end{equation}
in the limit of vanishing diffusion.
\item[(II)]
The $\be(K,N)$-condition for the Brownian motion or heat flow on $M$ is equivalent to the $(K,N)$-convexity of the function $S=\ent(.)$ on the
Wasserstein space $\cP_2(M)$.
\end{itemize}
Both links are related to each other since the heat flow is the solution to the ODE (''without diffusion'')
\begin{equation*}\label{ODE}
d\mu_t=-\nabla S(\mu_t)\,dt
\end{equation*}
on $\cP_2(M)$ (regarded as infinite dimensional Riemannian manifold).
The link (II) is the main result of this paper.

To see (I), note that in the case $\alpha>0$, equilibration and regularization effects of the stochastic dynamic \eqref{SDE} can be formulated in terms of the Bakry-\'Emery estimate for the generator
$L=\alpha\Delta-\nabla V\cdot \nabla$
of the associated transition semigroup $\bH_tu(x)={\mathbb E}_x[u(X_t)]$.
The law of $X_t$ evolves according to the dual semigroup $(\bH_t^*)_{t>0}$ with generator $L^*u=\alpha\Delta u+\mathrm{div}(u\cdot\nabla V)$.
Assume that the manifold $M$ has dimension $\le n$ and Ricci curvature $\ge k$.
Then the time-changed operator $\tilde L:=\frac1\alpha L$ satisfies the Bakry-\'Emery condition $\be(\frac1\alpha K,\frac1\alpha N)$ provided
\begin{equation}\label{alpha}
\mathrm{Hess}\, V-\frac1{N-\alpha n}(\nabla V\otimes\nabla V)\ge K-\alpha k
\end{equation}
[Prop. 4.21].
In the Wasserstein picture, the $\be(\frac1\alpha K,\frac1\alpha N)$-condition for $\tilde L$ translates into the $(\frac1\alpha K,\frac1\alpha N)$-convexity of the functional $\tilde S(\mu)=\ent(\mu)+\frac1\alpha \int V\,d\mu$ [Thm. 7]. The latter in turn is equivalent to the $(K,N)$-convexity of $S(\mu)=\alpha\ent(\mu)+ \int V\,d\mu$
on $\cP_2(M)$ [Lemma 2.9].

Note that this also makes perfectly sense for $\alpha=0$  in which case  the associated gradient flow equation on the Wasserstein space $\cP_2(M)$ reads
\begin{equation*}
d\mu_t=-\nabla V\,dt.
\end{equation*}
Obviously, this  precisely describes the evolution on $M$ determined by the semigroup $(\bH_t^*)_{t>0}$ with generator $L^*u=\mathrm{div}(u\cdot\nabla V)$.
Equilibration and regularization for this evolution are characterized by  the parameters $K$ and $N$ in the bound \eqref{alpha}
for $\alpha=0$, i.e.
\begin{equation*}
\mathrm{Hess}\, V-\frac1{N}(\nabla V\otimes\nabla V)\ge K.
\end{equation*}
This is the $(K,N)$-convexity of $V$ on $M$.
\end{remark*}

\textbf{Organization of the article.}  First we illustrate the new
concept of $(K,N)$-convexity in a smooth and finite dimensional
setting. Since many of the arguments which relate geodesic convexity,
the Evolution Variational Inequality and space-time expansion bounds
for the gradient flow are of a purely metric nature we study
$(K,N)$-convexity, $\evi_{K,N}$ and its consequences in the general
setting of metric spaces in Section~\ref{sec:evi}. In
Section~\ref{sec:cds} we turn to the study of $(K,N)$-convexity of the
entropy on the Wasserstein space. The entropic curvature-dimension
condition is introduced in Section~\ref{sec:cdkn} and its basic
properties are established. In particular we prove equivalence with
the reduced curvature-dimension condition for essentially
non-branching spaces. In Section~\ref{sec:riem-cdkn} we prove that the
entropic curvature-dimension condition plus linearity of the heat flow
is equivalent to the existence of an $\evi_{K,N}$ gradient flow of the
entropy which leads to the Riemannian curvature-dimension
condition. Here we also prove the stability results for $\rcdkn$.
Finally, in Section~\ref{sec:cde-bochner} we prove the equivalence of
the entropic curvature-dimension condition, space-time Wasserstein
expansion bounds, pointwise gradient estimates and the Bochner
inequality for infinitesimally Hilbertian metric measure spaces. As
applications, new functional inequalities deduced from $\ecdkn$ are
studied in Section~\ref{sec:FI} and the sharp Lichnerowicz bound for
$\rcdkn$ spaces is established in Section~\ref{sec:Lichnerowicz}.

\section{$(K,N)$-convex functions and their EVI gradient flows}\label{sec:evi}
\subsection{Gradient flows and $(K,N)$-convexity in a smooth setting}
\label{sec:smooth}

In order to illustrate the concept of $(K,N)$-convexity of the entropy
and the consequences for its gradient flow, we consider in this
section a smooth and finite-dimensional setting.  \medskip

Let $M$ be a smooth connected and geodesically complete Riemannian
manifold with metric tensor $\ip{\cdot,\cdot}$ and Riemannian distance
$d$. Let $S:M\to\R$ be a smooth function. Given two real numbers
$K\in\R$ and $N>0$, we say that $S$ is $(K,N)$-convex, if
  \begin{align}\label{eq:knconvex1}
    \Hess S - \frac{1}{N}\big(\nabla S \otimes\nabla S \big)~\geq~K\;,
  \end{align}
in the sense that for all $x\in M$ and $v\in T_xM$ we have
\begin{align*}
    \Hess S(x)[v] - \frac{1}{N}\ip{\nabla S(x),v}_x^2~\geq~K\abs{v}_x^2\;.
\end{align*}
Obviously, this condition becomes weaker as $N$ increases and in the
limit $N\to\infty$ we recover the notion of $K$-convexity, i.e. $\Hess
S\geq K$. It turns out to be useful to introduce the function
$U_N:M\to\R_+$ given by
\begin{align*}
  U_N(x)~=~\exp\left(-\frac{1}{N}S(x)\right)\;.
\end{align*}
A direct calculation shows that \eqref{eq:knconvex1} can equivalently
be written as:
\begin{align}\label{eq:knconvex2}
  \Hess U_N ~\leq~ -\frac{K}{N}\cdot U_N\;.
\end{align}
This condition can be thought of as a ``concavity'' property of
$U_N$. As with concavity, it can be expressed in an integrated form. To
this end we introduce the following functions.
\begin{definition}\label{def:distortioncoefficients}
  For $\kappa\in\R$ and $\theta\geq0$ we define the functions
  \begin{align*}
    \fs_\kappa(\theta)~&=~
    \begin{cases}
      \frac{1}{\sqrt{\kappa}}\sin\left(\sqrt\kappa\theta\right)\;, & \kappa>0\;,\\
      \theta\;, & \kappa=0\;,\\
      \frac{1}{\sqrt{-\kappa}}\sinh\left(\sqrt{-\kappa}\theta\right)\;, & \kappa<0\;,
    \end{cases}\\
%  \end{align*}
%\begin{align*}
    \fc_\kappa(\theta)~&=~
    \begin{cases}
      \cos\left(\sqrt\kappa\theta\right)\;, & \kappa\geq0\;,\\
      %1\;, & K=0\;,\\
      \cosh\left(\sqrt{-\kappa}\theta\right)\;, & \kappa<0\;.
    \end{cases}
  \end{align*}
Moreover, for $t\in[0,1]$ we set
 \begin{align*}
   \sigma_\kappa^{(t)}(\theta)~=~
   \begin{cases}
     \frac{\fs_\kappa(t\theta)}{\fs_\kappa(\theta)}\;, & \kappa\theta^2\neq0 \text{ and } \kappa\theta^2<\pi^2\;,\\
     t\;, & \kappa\theta^2=0\;,\\
     +\infty\;, & \kappa\theta^2\geq \pi^2\;.
   \end{cases}
 \end{align*}
\end{definition}
\begin{lemma}\label{lem:knconvexint}
  The following statements are equivalent:
  \begin{itemize}
  \item[(i)] The function $S$ is $(K,N)$-convex.
  \item[(ii)] For each constant speed geodesic
    $(\gamma_t)_{t\in[0,1]}$ in $M$  and all $t\in[0,1]$ we have with $d:=d(\gamma_0,\gamma_1)$:
    \begin{align}\label{eq:knconvex3}
      U_N(\gamma_t)~\geq~\sigkn{1-t}{d}\cdot U_N(\gamma_0) + \sigkn{t}{d}\cdot U_N(\gamma_1)\;.
    \end{align}
  \item[(iii)] For each constant speed geodesic
    $(\gamma_t)_{t\in[0,1]}$ in $M$ we have that
    \begin{align}\label{eq:knconvex4}
       U_N(\gamma_1)~\leq~\ckn{d}\cdot U_N(\gamma_0) + \frac{\skn{d}}{d}\cdot\left.\ddt\right\vert_{t=0}U_N(\gamma_t)\;.
    \end{align}
  \end{itemize}
  %  \begin{align}\label{eq:knconvex3}
  %   \skn{d(\gamma_0,\gamma_1)}\cdot U_N(\gamma_t)~\geq~\skn{(1-t)d(\gamma_0,\gamma_1)}\cdot U_N(\gamma_0) + \skn{t d(\gamma_0,\gamma_1)}\cdot U_N(\gamma_1)\;.
  % \end{align}
\end{lemma}

\begin{proof}
  (i)$\Rightarrow$(ii): Let $(\gamma_t)_{t\in[0,1]}$ be a constant
  speed geodesic. Then in particular
  $\abs{\dot{\gamma_t}}_{\gamma_t}=d$ and \eqref{eq:knconvex2}
  immediately yields that the function $u:t\mapsto U_N(\gamma_t)$
  satisfies
  \begin{align}\label{eq:geoconvex}
    u''(t)~\leq~-\frac{K}{N}d^2\cdot u(t)\;.
  \end{align}
  The function $v:[0,1]\to\R$ given by the right-hand side of
  \eqref{eq:knconvex3} has the same boundary values as
  $u$ and satisfies $v''=-(K/N) d^2\cdot v$. A comparison principle thus
  yields $u\geq v$.

  (ii)$\Rightarrow$(iii): This follows immediately by subtracting
  $U_N(\gamma_0)$ on both sides of \eqref{eq:knconvex3}, dividing by
  $t$ and letting $t\searrow0$.

  (iii)$\Rightarrow$(i): Let $\gamma:[-1,1]\to M$ be a constant speed
  geodesic with $\gamma_0=x$ and $\dot\gamma_0=v$,
  i.e. $d=d(\gamma_0,\gamma_1)=\abs{v}$. Using \eqref{eq:knconvex4}
  for the rescaled geodesics $\gamma':[0,1]\to M,\; t\mapsto\gamma_{\eps t}$ and
  $\gamma'':[0,1]\to M,\;t\mapsto\gamma_{-\eps t}$ and adding up we obtain
  \begin{align*}
    U_N(\gamma_{\eps}) + U_N(\gamma_{-\eps}) - 2\ckn{\eps d}\cdot U_N(\gamma_{0})~\leq~ 0\;.
  \end{align*}
  Dividing by $\eps^2$ and using the fact that $\ckn{\eps
    d}=1-\frac{K}{N}\eps^2d^2+o(\eps^2)$ finally yields
  \begin{align*}
    \Hess U_N(x)[v]~\leq~-\frac{K}{N}\abs{v}^2\;.
  \end{align*}
\end{proof}

\begin{remark}\label{rem:finite-diam}
  We note that the existence of a $(K,N)$-convex function $S:M\to\R$
  with $K>0$ poses strong constraints on the manifold $M$. In
  particular, it implies that the diameter of $M$ is bounded by
  $\sqrt{\frac{N}{K}}\pi$. This is immediate from the characterization
  \eqref{eq:knconvex3} and the singularity of the coefficient
  $\sigma_\kappa^{(t)}(\cdot)$ at $\pi/\sqrt{\kappa}$.
\end{remark}

\begin{lemma}\label{lem:gradflowevi}
  Assume that $S$ is $(K,N)$-convex and differentiable. A smooth curve
  $x:[0,\infty)\to M$ is a solution to the gradient flow equation
  \begin{align}\label{eq:smoothgradflow}
   \dot x_t~=~-\nabla S(x_t) \quad \forall t>0\;,
  \end{align}
  if and only if the following Evolution Variation Inequality
  ($EVI_{K,N}$) holds: for all $z\in M$ and all $t>0$:
  % \begin{align}\label{eq:smoothevikn}
  %      \frac{U_N(z)}{U_N(x(t))}~\leq~\ckn{d(x(t),z)}
  %   + \frac{1}{K}\ddt\ckn{d(x(t),z)}
  % \end{align}
  % if $K\neq0$ and otherwise if $K=0$:
  % \begin{align}\label{eq:smoothevi0}
  %   \frac{1}{N}\ddt\frac12 d(x(t),z)^2~\leq~1 -
  %   \frac{U_N(z)}{U_N(x(t))}\;.
  % \end{align}
 \begin{align}\label{eq:smoothevikn}
      \ddt\skn{\frac12 d(x_t,z)}^2 + K\cdot\skn{\frac12 d(x_t,z)}^2 ~\leq~\frac{N}{2}\left(1-\frac{U_N(z)}{U_N(x_t)}\right)\;.
  \end{align}
\end{lemma}

\begin{proof}
  To prove the only if part, fix $t\geq0$, $z\in M$ and a constant
  speed geodesic $\gamma:[0,1]\to M$ connecting $x_t$ to $z$. Observe
  that by \eqref{eq:smoothgradflow} and the first variation formula we
  have
  \begin{align*}
    \left.\frac{\dd}{\dd s}\right\vert_{s=0}U_N(\gamma_s)~&=~-\frac{1}{N}U_N(x_t)\ip{\nabla S(x_t),\dot\gamma_0}~=~\frac{1}{N}U_N(x_t)\ip{\dot x_t,\dot\gamma_0}\\
              &=~-\frac{1}{N}U_N(x_t)\ddt \frac12 d(x_t,z)^2\;.
  \end{align*}
  Combining this with the $(K,N)$-convexity condition in the form
  \eqref{eq:knconvex4} we obtain with $d=d(x_t,z)$:
  \begin{align}
    \label{eq:smoothevialt}
    U_N(z)~\leq~\ckn{d} U_N(x_t) - \frac{\skn{d}}{N d}U_N(x_t)\ddt\frac12 d(x_t,z)^2\;.
  \end{align}
  Using the identity
  \begin{align}\label{eq:sc-trick}
    \frac{2}{N}\skn{\frac12\theta}^2~=~\frac{1}{K}\big(1-\ckn{\theta})\;,
  \end{align}
  it is immediate to see that the last inequality is equivalent to
  \eqref{eq:smoothevikn}.

  For the if part, fix $t\geq0$ and a constant speed geodesic
  $\gamma:[0,1]\to M$ with $\gamma_0=x_t$. Using the Evolution
  Variational inequality in the form \eqref{eq:smoothevialt} with
  $z=\gamma_\eps$ for some $\eps>0$ we obtain
  \begin{align*}
    U_N(\gamma_\eps)-\ckn{\eps\abs{v}}U_N(\gamma_0)~\leq~\frac{\skn{\eps\abs{v}}}{N\eps\abs{v}}U_N(\gamma_0)\ip{\dot x_t,\eps v}\;,
  \end{align*}
  where $v=\dot\gamma_0$. Dividing by $\eps$ and letting
  $\eps\searrow0$, taking into account that $\ckn{\eps d}=1+o(\eps)$
  and $\skn{\eps d}=\eps d +o(\eps^2)$, we obtain
  \begin{align*}
    \ip{-\nabla S(x_t),v}~\leq~\ip{\dot x_t,v}\;.
  \end{align*}
  Since the direction of $v\in T_{x_t}M$ was arbitrary we obtain
  \eqref{eq:smoothgradflow}.
\end{proof}

We conclude this section by exhibiting some 1-dimensional models of
$(K,N)$-convex functions.

\begin{example}\label{ex:1D-model}
  Each of the following are $(K,N)$-convex functions. Note that the
  domain of definition is maximal in each case.
  \begin{itemize}
  \item[(i)] For $N>0$ and $K>0$ let $S:(-\frac{\pi}{2}\sqrt{\frac{N}{K}},\frac{\pi}{2}\sqrt{\frac{N}{K}})\to\R$ defined by
    \begin{align*}
      S(x)~=~-N \log\cos\left(x\sqrt{K/N}\right)\;.
    \end{align*}
  \item[(ii)] For $N > 0$ and $K = 0$ let $S: (0,\infty) \to \R$ defined by
    \begin{align*}
        S(x) = - N \log x \;.
    \end{align*}
  \item[(iii)] For $N>0$ and $K < 0$ let $S:(0,\infty)\to\R$ defined by
    \begin{align*}
      S(x)~=~-N \log\sinh\left(x\sqrt{-K/N}\right)\;.
    \end{align*}
  \item[(iv)] For $N>0$ and $K<0$ let $S:(-\infty,\infty)\to\R$ defined by
    \begin{align*}
      S(x)~=~-N \log\cosh\left(x\sqrt{-K/N}\right)\,.
    \end{align*}
  \end{itemize}
\end{example}
The cases (i) and (iv) of the previous example canonically extend to multidimensional spaces.
 \begin{example}
 Let $(M,g)$ be a $n$-dimensional Riemannian manifold, $z\in M$ be any point and  $N>0$ be any real number. 
 \begin{itemize}
  \item[(i)] Then for each $K>0$ the function 
   \begin{align*}
      S(x)~=~-N \log\cos\left(d(x,z)\sqrt{K/N}\right)
    \end{align*}
  defined on the open ball $\big\{x\in M: \ d(x,z)<\frac{\pi}{2}\sqrt{N/K}\big\}$ is $(K,N)$-convex provided the sectional curvature of the underlying space  is $\le K/N$. (This in particular applies to the Euclidean space $\R^n$.)
 \item[(ii)]
  For each $K<0$ the function
   \begin{align*}
      S(x)~=~-N \log\cosh\left(d(x,z)\sqrt{-K/N}\right)
    \end{align*}
  defined on all of $M$ is $(K,N)$-convex provided the sectional curvature of the underlying space  is $\ge K/N$.
(This in particular applies to the Euclidean space $\R^n$.)
\end{itemize}
Indeed, analogous statements hold true on geodesic spaces with generalized bounds for the sectional curvature in the sense of Alexandrov \cite{BBI01}.
\end{example}

\subsection{$(K,N)$-convexity in metric spaces}

We proceed our study of $(K,N)$-convexity in a purely metric setting.
Let $(X,d)$ be a complete and separable metric space and let $S : X
\to [ - \infty , \infty ]$ be a functional on $X$.
% $U:X\to(-\infty,\infty]$ be a lower semi-continuous functional.
We denote by $D(S):=\{x\in X~:~S(x) < \infty \}$ the proper domain of
$S$. Given a number $N\in(0,\infty)$ we define the functional
$U_N:X\to[0,\infty)$ by setting
\begin{align}\label{eq:def-U}
  U_N(x)~:=~\exp\left(-\frac{1}{N}S(x)\right)\;.
\end{align}

\begin{definition}\label{def:knconvex}
  Let $K\in\R$, $N\in(0,\infty)$. We say that the functional $S$ is
  $(K,N)$-convex if and only if for each pair $x_0,x_1\in D(S)$ there
  exists a constant speed geodesic $\gamma:[0,1]\to X$ connecting
  $x_0$ to $x_1$ such that for all $t\in[0,1]$:
  \begin{align}\label{eq:defknconvex}
     U_N(\gamma_t)~\geq~\sigkn{1-t}{d(\gamma_0,\gamma_1)}\cdot U_N(\gamma_0) + \sigkn{t}{d(\gamma_0,\gamma_1)}\cdot U_N(\gamma_1)\;.
  \end{align}
  If \eqref{eq:defknconvex} holds for every geodesic $\gamma:[0,1]\to
  D(S)$ we say that $S$ is \emph{strongly} $(K,N)$-convex.
\end{definition}

For investigating $(K,N)$-convexity (especially for the strong form),
the following equivalent conditions will be helpful in the sequel.

\begin{lemma}\label{lem:knconvexint2}
  Let $u : X \to [ 0, \infty )$ be a upper semi-continuous function
  and $\kappa \in \R$.
%   $\gamma$ a constant speed geodesic on $M$ and $d : = d ( \gamma_0 , \gamma_1 )$.
  Then the following statements are equivalent:
  \begin{itemize}
  \item[(i)]
    For each constant speed geodesic $\gamma : [ 0 , 1 ] \to X$ and $t \in [0 , 1]$,
    $\displaystyle
% \frac{\mathrm{d}^2}{\mathrm{d}t^2}
      u'' (\gamma_t)
      ~\leq~
      - \kappa d ( \gamma_0 , \gamma_1 )^2 u (\gamma_t)
    $
    in the distributional sense, i.e.
%     for each constant speed geodesinc $(\gamma_t)_{t\in[0,1]}$ in $M$
%     we have with $d:=d(\gamma_0,\gamma_1)$:
    \begin{align*}
      \int_0^1 \varphi'' (t) u (\gamma_t) \,\dd t
      ~\leq~
      - \kappa d ( \gamma_0 , \gamma_1 )^2 \int_0^1 \varphi (t) u (\gamma_t) \,\dd t
    \end{align*}
    for any $\varphi \in C_0^\infty ( ( 0 , 1 ) )$ with $\varphi \ge 0$.
  \item[(ii)] For each constant speed geodesic $\gamma$ on $X$ and $t \in [0,1]$,
%    For each constant speed geodesic
%    $(\gamma_t)_{t\in[0,1]}$ in $M$  and all $t\in[0,1]$ we have with $d:=d(\gamma_0,\gamma_1)$:
    \begin{align}\label{eq:knconvex2-3}
       u (\gamma_t)~\geq~
       \sigma_{\kappa}^{(1-t)} ( d ( \gamma_0 , \gamma_1 ) ) \cdot u(\gamma_0)
         +
       \sigma_{\kappa}^{(t)} ( d ( \gamma_0 , \gamma_1 ) ) \cdot u(\gamma_1).
%      U_N(\gamma_t)~\geq~\sigkn{1-t}{d}\cdot U_N(\gamma_0) + \sigkn{t}{d}\cdot U_N(\gamma_1)\;.
    \end{align}
  \item[(iii)] For each constant speed geodesic $\gamma$ on $X$,
    there is $\delta = \delta_\gamma > 0$ such that
    for all $0 \le s \le t \le 1$ with $t-s \le \delta$ and $\a \in [ 0 , 1 ]$,
%    For each constant speed geodesic
%    $(\gamma_t)_{t\in[0,1]}$ in $M$  and all $t\in[0,1]$ we have with $d:=d(\gamma_0,\gamma_1)$:
    \begin{align}\label{eq:knconvex2-3'}
       u (\gamma_{(1-\a)s + \a t})~\geq~
       \sigma_{\kappa}^{(1-\a)} ( d ( \gamma_s , \gamma_t ) ) \cdot u(\gamma_s)
         +
       \sigma_{\kappa}^{(\a)} ( d ( \gamma_s , \gamma_t ) ) \cdot u(\gamma_t).
%      U_N(\gamma_t)~\geq~\sigkn{1-t}{d}\cdot U_N(\gamma_0) + \sigkn{t}{d}\cdot U_N(\gamma_1)\;.
    \end{align}
%  \item[(iv)]
%     For each constant speed geodesic
%     $(\gamma_t)_{t\in[0,1]}$ in $M$ we have that
%     \begin{align}\label{eq:knconvex2-4}
%        U_N(\gamma_1)~\leq~\ckn{d}\cdot U_N(\gamma_0) + \frac{\skn{d}}{d}\cdot\left.\ddt\right\vert_{t=0}U_N(\gamma_t)\;.
%     \end{align}
  \item[(iv)] For each constant speed geodesic $\gamma$ on $X$ and $t \in [ 0, 1 ]$,
    \begin{equation*}
        u(\gamma_t)~\geq~(1-t)\cdot u(\gamma_0)+ t \cdot u(\gamma_1)
        + \kappa d ( \gamma_0 , \gamma_1 )^2 \int_0^1 g (t,r) u(\gamma_r) \dd r
    \end{equation*}
with $g (t,r)= \min\{(1-t)r,(1-r)t \}$ being the Green function on the interval $[0,1]$.
  \end{itemize}
In particular, when $ - \infty \notin S (X)$ and $S$ is lower semi-continuous,
$S$ is strongly $(K,N)$-convex if and only if
$u = U_N$ and $\kappa = K/N$ satisfies one of these conditions.
  %  \begin{align}\label{eq:knconvex3}
  %   \skn{d(\gamma_0,\gamma_1)}\cdot U_N(\gamma_t)~\geq~\skn{(1-t)d(\gamma_0,\gamma_1)}\cdot U_N(\gamma_0) + \skn{t d(\gamma_0,\gamma_1)}\cdot U_N(\gamma_1)\;.
  % \end{align}
\end{lemma}

\begin{proof}
For simplicity of presentation,
we denote $\theta = \theta_\gamma = d ( \gamma_0 , \gamma_1 )$
in this proof whenever a fixed geodesic is under consideration.
we also denote the restriction of $\gamma$ on $[s,t]$
for $0 \le s < t \le 1$ by $\gamma^{[s,t]} : [ 0 ,1 ] \to X$,
that is, $\gamma^{[s,t]}_r := \gamma_{(1-r)s + rt}$.

(i) $\Rightarrow$ (iv):
Let us denote $u_* (s) := \int_0^1 g (s,r) u(\gamma_r) \,\dd r$.
Since we have
\begin{equation*}
  \int_0^1 \varphi'' (r) u_* (r) \,\dd r = - \int_0^1 \varphi (r) u(\gamma_r) \,\dd r
\end{equation*}
for any $\varphi \in C^\infty_0 ((s,t))$ with $\varphi \ge 0$,
(i) implies $( u ( \gamma_\cdot ) - \kappa \theta^2 u_* )'' \le 0$ on $[ 0 , 1 ]$
in the distributional sense.
Thus the distributional characterization of convex functions
(see \cite[Theorem~1.29]{S11}, for instance) yields that
$u ( \gamma_\cdot ) - \kappa \theta^2 u_*$ coincides
with a concave function a.e.~and hence concave
because $u$ is upper semi-continuous.
It immediately implies (iv) since $u_* (0) = u_* (1) = 1$.

(iv) $\Rightarrow$ (i):
Note first that $u ( \gamma_t )$ is continuous.
Indeed, the condition (iv) together with the upper semi-continuity of $u$
implies that
$u ( \gamma_t )$ is continuous at $t = 0 , 1$.
Thus the continuity follows
by applying the same for $\gamma^{[0,s]}$ and $\gamma^{[s,1]}$.
For $s \in ( 0 , 1 )$ and $h > 0$ with $s+h , s-h \in [ 0 , 1 ]$,
we apply (iv) to $\gamma^{[s-h,s+h]}$ and $t = 1/2$ to obtain
\begin{equation*}
\frac{u(\gamma_{s+h}) + u (\gamma_{s-h}) - 2 u ( \gamma_s )}{2}
\le
4 \kappa h^2 \theta^2
\int_0^1
  g \left( \frac12 , r \right)
  u ( \gamma_{ s + (2 r - 1 ) h } )
\,\dd r .
\end{equation*}
Then (i) follows by multiplying $\varphi \in C_0^\infty ( ( 0 , 1 ) )$,
integrating w.r.t.~$t$ (for sufficiently small $h$),
dividing by $h^2$ and $h \to 0$ with a change of variable.

(i) $\Rightarrow$ (ii):
Take $\eps > 0$ and $\varphi \in C_0^\infty ( ( 0 , 1 ) )$
with $\int_0^1 \varphi (x) \, \dd x = 1$,
and let
\begin{equation*}
\tilde{u}_\eps (t)
: =
\int_0^1
\eps^{-1} \phi ( \eps^{-1} ( t - r ) ) u ( \gamma_r )
\,\dd r\,.
\end{equation*}
Then (i) implies $\tilde{u}_\eps'' (t) \le -\kappa \theta^2 \tilde{u}_\eps (t)$
for each $t \in [ a_\eps , 1 ]$ for some $a_\eps > 0$.
Note that $a_\eps$ can be chosen so that $\lim_{\eps \to 0} a_\eps = 0$.
Thus, in the same way as in Lemma~2.2,
we obtain
\begin{equation*}
\tilde{u}_\eps ( (1-t) a_\eps + t )
\ge
\sigma_{\kappa}^{(1-t)} ( \theta ) \tilde{u_\eps} ( a_\eps )
+
\sigma_{\kappa}^{(t)} ( \theta ) \tilde{u_\eps} (1).
\end{equation*}
By virtue of the equivalence (i) $\Leftrightarrow$ (iv),
$u \circ \gamma$ is continuous and
hence $\tilde{u}_\eps \to u \circ \gamma$ as $\eps \to 0$
uniformly on $[ 0 , 1 ]$.
Thus the conclusion follows by letting $\eps \to 0$.

(ii) $\Rightarrow$ (iii):
It follows by considering (ii) for $\gamma^{[s,t]}$.

(iii) $\Rightarrow$ (i):
We imitate the proof of the implication (iv) $\Rightarrow$ (i)
by using the following:
\begin{equation*}
\lim_{h \to 0} \frac{1}{h^2}
\left( \frac12 - \sigma_\kappa^{(1/2)} (2h\theta) \right)
=
- \frac14 \kappa \theta^2 .
\end{equation*}
% Finally, note that the condition (iii) makes a connection
% with the strong $(K,N)$-convexity of $U$. Indeed,
% The condition (iii) for $u(t) = U_N (\gamma_t)$,
% $\kappa = K/N$ and $\theta = d ( \gamma_0 , \gamma_1 )$ is
% nothing but \eqref{eq:defknconvex} for the geodesic $\gamma|_{[s,r]}$.
% It proves the last assertion.
\end{proof}

We conclude this section with some remarks about $(K,N)$-convexity.
The first property is immediate from the definition.

\begin{lemma}\label{lem:scale-convex}
  If $S$
% $U:X\to(-\infty,\infty]$
  is $(K,N)$-convex, then for $\lambda>0$
  the functional $\lambda\cdot S$ is $(\lambda K,\lambda N)$-convex.
\end{lemma}

\begin{lemma}\label{lem:convex-sum}
  Let $S^1:X\to(-\infty,\infty]$ be a $(K_1,N_1)$-convex functional
  and $S^2:X\to(-\infty,\infty]$ a strongly $(K_2,N_2)$-convex
  functional.
%   Let $U^i:X\to(-\infty,\infty]$ be strongly $(K_i,N_i)$-convex functionals
%   for $i=1,2$ and real numbers $K_i,N_i\in\R$ with $N_i>0$.
  Then the functional $S:=S^1 + S^2$ is $(K_1+K_2,N_1+N_2)$-convex.
  In particular, $S$ is strongly $(K_1+K_2,N_1+N_2)$-convex if $S^1$
  is strongly $(K_1 , N_1 )$-convex.
\end{lemma}

\begin{proof}
  Let us set $K=K_1+K_2$ and $N=N_1+N_2$ and given $x_0 , x_1 \in
  D(S)=D(S^1)\cap D(S^2)$ take a constant speed geodesic
  $\gamma:[0,1]\to X$ from $x_0$ to $x_1$ according to the convexity
  assumption of $S^1$. By the convexity assumption on $S^1$ and $S^2$
  we have
  \begin{align*}
    \log U_N(\gamma_t)~&=~\frac{N_1}{N}\frac{(-1)}{N_1}S^1(\gamma_t) + \frac{N_2}{N}\frac{(-1)}{N_2}S^2(\gamma_t)\\
%                       &\geq~\frac{N_1}{N}G \left(\frac{(-1)}{N_1}U^1(\gamma_0),\frac{(-1)}{N_1}U^1(\gamma_1),\frac{K_1}{N_1} \right)\\
%                       &\quad +\frac{N_2}{N}G \left(\frac{(-1)}{N_2}U^2(\gamma_0),\frac{(-1)}{N_2}U^2(\gamma_1),\frac{K_2}{N_2} \right)\;,
                      &\geq~\frac{N_1}{N}G_t \left(\frac{(-1)}{N_1}S^1(\gamma_0),\frac{(-1)}{N_1}S^1(\gamma_1),\frac{K_1}{N_1} d ( \gamma_0 , \gamma_1 )^2 \right)\\
                      &\quad +\frac{N_2}{N}G_t \left(\frac{(-1)}{N_2}S^2(\gamma_0),\frac{(-1)}{N_2}S^2(\gamma_1),\frac{K_2}{N_2} d ( \gamma_0 , \gamma_1 )^2 \right)\;,
  \end{align*}
  % where the function $G_t$ is given by \eqref{eq:convexG}
  % with $\theta=d(\gamma_0,\gamma_1)$.
  where the function $G_t$ is given by \eqref{eq:convexG}.
  By Lemma~\ref{lem:convexG} below,
  % $G$
  $G_t$ is convex. Hence we obtain
  \begin{align*}
  %  \log U_N(\gamma_t)~\geq~G\left(\frac{(-1)}{N}U(\gamma_0),\frac{(-1)}{N}U(\gamma_1),\frac{K}{N}\right)\;.
    \log U_N(\gamma_t)~\geq~G_t\left(\frac{(-1)}{N}S(\gamma_0),\frac{(-1)}{N}S(\gamma_1),\frac{K}{N} d( \gamma_0 , \gamma_1 )^2\right)\;.
  \end{align*}
  Taking the exponential on both sides yields the claim. The last
  assertion is obvious from the proof.
\end{proof}

\begin{lemma}\label{lem:convexG}
%   For any fixed $\theta\geq0$ and $t\in[0,1]$ the function
%   $G:\R\times\R\times(-\infty,\pi^2/\theta^2)\to\R$ given by
%   \begin{align}\label{eq:convexG}
%     G(x,y,\kappa)~=~~\log\left[\sigma^{(1-t)}_\kappa(\theta)e^x + \sigma^{(t)}_\kappa(\theta)e^y\right]
%   \end{align}
%   is convex.
  For any fixed $t\in[0,1]$ the function
  $G_t :\R\times\R\times(-\infty,\pi^2 )\to\R$ given by
  \begin{align}\label{eq:convexG}
    G_t (x,y,\kappa)~=~~\log\left[\sigma^{(1-t)}_\kappa(1)e^x + \sigma^{(t)}_\kappa(1)e^y\right]
  \end{align}
  is convex.
\end{lemma}
Note that we have $\sigma^{(s)}_\kappa (\theta) = \sigma^{(s)}_{\kappa \theta^2 } (1)$
for $s \in [ 0 , 1 ]$, $\theta \geq 0$ and $\kappa \in ( - \infty , \pi^2 / \theta^2 )$.
It is useful to apply this lemma.
\begin{proof}
  We define the function
  % $g^{(t)}: \kappa\mapsto\log\sigma^{(t)}_\kappa(\theta)$
  $g^{(t)}: \kappa\mapsto\log \sigma^{(t)}_\kappa(1)$
  % on $(-\infty,\pi^2/\theta^2)$
  on $(-\infty, \pi^2 )$
  and write
  \begin{align*}
    G_t (x,y,\kappa)~=~F\Big(x+g^{(1-t)}(\kappa),y+g^{(t)}(\kappa)\Big)\;,
  \end{align*}
  where $F(u,v)=\log\big(e^u+ e^v\big)$. The claim then follows by
  noting that the function $F$ is convex, $a\mapsto F(u+a,v+a)$ is
  increasing and that the functions $g^{(t)}$ are convex.
\end{proof}

Finally we remark that the notion of $(K,N)$-convexity is consistent
in the parameters $K$ and $N$.

\begin{lemma}\label{lem:convexconsistent}
  If $S$ is $(K,N)$-convex then it is also $(K',N')$-convex for all
  $K'\leq K$ and $N'\geq N$. Moreover, it is $K$-convex in the sense
  that for each pair $x_0,x_1\in D(S)$ there exist a constant speed
  geodesic $\gamma:[0,1]\to X$ connecting $x_0$ to $x_1$ such that for
  all $t\in[0,1]$:
  \begin{align}\label{eq:kconvex}
    S(\gamma_t)~\leq~(1-t)S(\gamma_0) + t S(\gamma_1) -\frac{K}{2}t(1-t)d(\gamma_0,\gamma_1)^2\;.
  \end{align}
\end{lemma}

\begin{proof}
  Consistency in $K$ is immediate from the fact that for any fixed $t$
  and $\theta$ the coefficient $\sigkn{t}{\theta}$ is increasing in
  $K$.
%   As for the consistency in $N$, note that \eqref{eq:defknconvex}
%   is equivalent to the following:
%   \begin{equation*}
%     - S ( \gamma_t )
%     ~\geq~
%     \frac{1}{N^{-1}}
%     G_t \left(
%       - N^{-1} S (\gamma_0),
%       - N^{-1} S (\gamma_0),
%       N^{-1} K d ( \gamma_0 , \gamma_1 )^2
%     \right)
%   \end{equation*}
%   Since $G_t ( 0 , 0 , 0 ) = 0$, the right hand side of the last
%   inequality is nonincreasing in $N$ by Lemma~\ref{lem:convexG}. This
%   is nothing but the claim.
  Consistency in $N$ is
  a consequence e.g. of Lemma~\ref{lem:convex-sum}
  and the trivial observation that for any $N'>N$
  the constant functional $S^0 \equiv 0$ is $(0,N'-N)$-convex.

  Using the consistency in $N$ we can derive \eqref{eq:kconvex} by
  subtracting $1$ on both sides of \eqref{eq:defknconvex}, multiplying
  with $N$ and passing to the limit $N\nearrow\infty$. Here we use the
  fact that $\sigkn{t}{\theta}=t+-K(t^3-t)\theta^2/(6N) + o(1/N)$ and
  $U_N(x)=1-S(x)/N+o(1/N)$.
\end{proof}

\subsection{Evolution Variational Inequalities in metric spaces}
\label{sec:metric}

In this section we study the Evolution Variational Inequality with
parameters $K$ and $N$ and the associated notion of gradient flow in a
purely metric setting. In particular, we investigate the relation with
geodesic convexity. Our approach extends the results obtained in
\cite{DS08,AGS11b} where the case $N=\infty$ has been considered.

\medskip

Let $(X, d)$ be a complete separable geodesic metric space and $S : X
\to ( - \infty , \infty ]$ a lower semi-continuous functional.  Note
that our framework is slightly more restrictive than that in the last
section. We define the \emph{descending slope} of $S$ at $x\in D(S)$
as
\begin{align*}
  \abs{\nabla^-S}(x)~:=~\limsup\limits_{y\to x}\frac{\big(S(x)-S(y)\big)_+}{d(x,y)}\;.
\end{align*}
For $x\notin D(S)$ we set $\abs{\nabla^-S}=+\infty$. A curve
$\gamma:I\to X$ defined on an interval $I\subset\R$ is called
\emph{absolutely continuous} if
\begin{align}\label{eq:abscont}
  d(\gamma_s,\gamma_t)~\leq~\int_s^tg(r)\dd r\quad\forall s,t\in I\;,s\leq t\;,
\end{align}
for some $g\in L^1(I)$. For an absolutely continuous curve $\gamma$ the
\emph{metric speed}, defined by
\begin{align*}
  \abs{\dot\gamma}(t)~:=~\lim\limits_{h\to0}\frac{d(\gamma_{t+h},\gamma_t)}{\abs{h}}\;,
\end{align*}
exists for a.e. $t\in I$ and is the minimal $g$ in \eqref{eq:abscont}
(see e.g. \cite[Thm.~1.1.2]{AGS08}). The following is a classical
notion of gradient flow in a metric space, see e.g. \cite{AGS08}.

\begin{definition}[Gradient flow]\label{def:metric-gf}
  We say that a locally absolutely continuous curve $x:[0,\infty)\to
  X$ with $x_0\in D(S)$ is a (downward) gradient flow of $S$ starting in $x_0$ if
  the \emph{Energy Dissipation Equality} holds:
  \begin{align}\label{eq:EDE}
    S(x_{s})~=~S(x_{t})
    + \frac12\int_s^t\abs{\dot x_r}^2+\abs{\nabla^-S}(x_r)\dd r\quad\forall 0\leq s\leq t\;.
  \end{align}
\end{definition}
We introduce here a more restrictive notion of gradient flow based on
the Evolution Variational Inequality.
\begin{definition}[$\evikn$ gradient flow]\label{def:KNflow}
  Let $K\in \R$, $
N\in(0,\infty)$ and let $x:(0,\infty)\to D(S)$ be a
  locally absolutely continuous curve. We say that $(x_t)$ is an
  $\evikn$ \emph{gradient flow} of $S$ starting in $x_0$ if
  $\lim_{t\to0}x_t=x_0$ and if for all $z\in D(S)$ the \emph{Evolution
  Variational Inequality}
  \begin{align}\label{eq:knflow}
    \ddt\skn{\frac12 d(x_t,z)}^2 + K\cdot\skn{\frac12 d(x_t,z)}^2~\leq~\frac{N}{2}\left(1-\frac{U_{N}(z)}{U_{N}(x_t)}\right)
    %\ddt\left[e^{Kt}\sknp{\frac12 d(x_t,z)}^2\right] ~\leq~e^{Kt}\frac{N}{2}\left(1-\frac{U_{N'}(z)}{U_{N'}(x_t)}\right)
  \end{align}
holds for a.e. $t>0$.
\end{definition}

% \begin{remark}\label{rem:evialt}
%   Using the identity
%    \begin{align}\label{eq:sc-trick}
%     \frac{2}{N}\skn{\frac12\theta}^2~=~\frac{1}{K}\big(1-\ckn{\theta})\;,
%   \end{align}
%   we can rewrite \eqref{eq:knflow} in the form
  % \begin{align}\label{eq:evialt1}
  %   \frac{U_N(z)}{U_N(P_tx)}~\leq~\ckn{d(P_tx,z)}
  %   -\frac{2}{N}\ddtr\skn{\frac12d(P_tx,z)}^2\;.
  % \end{align}
  % Note that if $K\neq0$ the Evolution Variational Inequality
  % \eqref{eq:knflow3} can also be writen in the following form
  % \begin{align}\label{eq:evialt2}
  %  \frac{U_N(z)}{U_N(P_tx)}~\leq~\ckn{d(P_tx,z)}
  %   + \frac{1}{K}\ddtr\ckn{d(P_tx,z)}\;.
  % \end{align}
  % or even
%   \begin{align}
%     \label{eq:evialt}
%     \ddt\frac12
%     d(x_t,z)^2~\leq~\frac{N d(x_t,z)}{\skn{d(x_t,z)}}\left[\ckn{d(x_t,z)}
%       - \frac{U_{N}(z)}{U_{N}(x_t)}\right]\;.
%   \end{align}
% \end{remark}

\begin{lemma}\label{lem:consistentKN}
  If $(x_t)_t$ is an $\evikn$ flow for $S$, then it is also an $\evi_{K',N'}$ flow
  for $S$ for any $K'\leq K$ and $N'\geq N$. Moreover, $(x_t)$ is an
  $\evi_K$ flow for $S$, i.e. for all $z\in D(S)$ and a.e. $t>0$:
  \begin{align}\label{eq:kflow}
    \frac12\ddt d(x_t,z)^2 + \frac{K}{2}d(x_t,z)^2~\leq~S(z) - S(x_t)\;.
  \end{align}
\end{lemma}

 \begin{proof}
   Using the \eqref{eq:sc-trick}
   % \begin{align}\label{eq:sc-trick}
  %   \frac{2}{N}\skn{\frac12\theta}^2~=~\frac{1}{K}\big(1-\ckn{\theta})\;,
  % \end{align}
   one checks that \eqref{eq:knflow} is equivalent to either of the
   following inequalities:
   \begin{align}\label{eq:evialt1}
     \frac12 \ddt d(x_t,z)^2 ~&\leq~\frac{Nd}{\skn{d}}\left[\ckn{d} -
       \frac{U_{N}(z)}{U_{N}(x_t)}\right]\\\label{eq:evialt2}
     \frac12 \ddt d(x_t,z)^2~&\leq~\frac{d}{\skn{d}}N\left[1-\frac{U_N(z)}{U_N(x_t)}\right] -
     2Kd\frac{\skn{\frac12 d}^2}{\skn{d}}\;,
   \end{align}
   where we set $d=d(x_t,z)$. Consistency in $K$ can be seen from
   \eqref{eq:evialt1} by noting that for any $\theta\geq0$ we have
   that $\skn{\theta}$ and $\ckn{\theta}/\skn{\theta}$ is decreasing
   in $K$. Consistency in $N$ follows from \eqref{eq:evialt2} and the
   fact that for any $v\in\R$ and $\theta\geq0$ both
   \begin{align*}
     N\left[1-\exp\Big(-\frac{1}{N}v\Big)\right]\frac{1}{\skn{\theta}}\quad\text{ and }\quad -K\cdot\frac{\skn{\frac12\theta}^2}{\skn{\theta}}
   \end{align*}
   are increasing in $N$. \eqref{eq:kflow} follows immediately from
   \eqref{eq:evialt2} by passing to the limit as $N\to\infty$. For
   this we note that
   \begin{align*}
     \lim\limits_{N\to\infty}\skn{d}=d\;,\quad\lim\limits_{N\to\infty}N\left[1-\frac{U_N(z)}{U_N(x_t)}\right]=S(z)-S(x_t)\;.
   \end{align*}
\end{proof}

\begin{remark}
  This shows consistency with the theory of $\evi_K$ gradient flows of
  geodesically $K$-convex functions. It can be thought of as the
  limiting case $N=\infty$. By taking the limit $N\to\infty$ in the
  estimates obtained in this section we recover the corresponding
  results for $\evi_K$ flows established in \cite{DS08,AGS11b}.
\end{remark}

We summarize here some properties of $\evikn$ gradient flows. To this
end we set for $\kappa\in\R$ and $t\geq0$:
\begin{align*}
  e_\kappa(t)~=~\int_0^t\e^{\kappa s}\dd s\;.
\end{align*}

\begin{proposition}\label{prop:evi-properties}
  Let $(x_t)$ be an $\evikn$ gradient flow of $S$ starting in
  $x_0$. Then the following statements hold:
  \begin{itemize}
  \item[(i)] If $x_0\in D(S)$ then $(x_t)$ is also a metric gradient
    flow in the sense of Definition~\ref{def:metric-gf}. In
    particular, the map $t\mapsto S(x_t)$ is non-increasing.
  \item[(ii)] We have the uniform regularization bound
    \begin{align}\label{eq:regularization}
      \frac{U_N(z)}{U_N(x_t)}~\leq~1 + \frac{2}{N e_K(t)}\skn{\frac12
        d(x_0,z)}^2
    \end{align}
  \item[(iii)] If $S$ is bounded below we have the uniform continuity
    estimate
  \begin{align}\label{eq:uniform-continuity}
    \skn{\frac12d(x_{t_1},x_{t_0})}^2~\leq~\frac{N}{2e_{-K}(t_1-t_0)}\left[1-\frac{U_N(x_{t_0})}{U_{N}^{\max}}\right]\;.
  \end{align}
  \end{itemize}
\end{proposition}

\begin{proof}
  By Lemma~\ref{lem:consistentKN} $(x_t)$ is an $\evi_K$ flow of $S$
  and hence a metric gradient flow by \cite[Prop.~3.9]{AG12}.
  \eqref{eq:regularization} follows immediately from
  \eqref{eq:evi-int} in Proposition~\ref{prop:evi-equiv} below by
  taking $t_0=0$. The uniform continuity estimate
  \eqref{eq:uniform-continuity} is obtained similarly by taking
  $z=x_{t_0}$.
\end{proof}

The following result collects several equivalent reformulations of the
definition of $\evikn$ gradient flows which will be useful in the
sequel. For this we say that a subset $D\subset D(S)$ is \emph{dense
  in energy}, if for any $z\in D(S)$ there exists a sequence
$(z_n)\subset D$ such that $d(z_n,z)\to0$ and $S(z_n)\to S(z)$ as
$n\to\infty$. For a function $f:I\to\R$ on some interval $I$ we use
the notation
\begin{align*}
  \ddtr f(t)~=~\limsup\limits_{h\searrow0}\frac{f(t+h)-f(t)}{h}
\end{align*}
to denote the right derivative.

\begin{proposition}\label{prop:evi-equiv}
  Let $D\subset D(S)$ be dense in energy and let $x:(0,\infty)\to
  D(S)$ be a locally absolutely continuous curve with
  $\lim_{t\to0}x_t=x_0$. Then $(x_t)$ is an $\evikn$ gradient flow of
  $S$ if and only if one of the following statements holds:
  \begin{itemize}
  \item[(i)] The differential inequality \eqref{eq:knflow} holds for
    all $z\in D$ and a.e. $t>0$.
  \item[(ii)] For all $z\in D$ and all $0\leq t_0\leq t_1$:
    \begin{align}\label{eq:evi-int}
      e_K(t_1-t_0)\frac{N}{2}\left(1-\frac{U_N(z)}{U_N(x_{t_1})}\right)~\geq~\e^{K(t_1-t_0)}\skn{\frac12 d(x_{t_1},z)}-\skn{\frac12 d(x_{t_0},z)}^2\;.
    \end{align}
  \item[(iii)] For all $z\in D$ and \emph{all} $t>0$:
  \begin{align}\label{eq:knflowallt}
    \ddtr\skn{\frac12 d(x_t,z)}^2 + K\cdot \skn{\frac12
      d(x_t,z)}^2~\leq~\frac{N}{2}\left(1-\frac{U_{N}(z)}{U_{N}(x_t)}\right)
  \end{align}
 \end{itemize}
\end{proposition}

\begin{proof}
  We prove the equivalence of Definition~\ref{def:KNflow} and
  (ii). Assume that $(x_t)$ is an $\evikn$ flow and note that the
  right hand side of \eqref{eq:knflow} can be rewritten as
  \begin{align*}
    \e^{-Kt}\ddt\left[\e^{Kt}\skn{\frac12 d(x_t,z)}^2\right]\;.
  \end{align*}
  Integrating from $t_0$ to $t_1$ and using that the map $t\mapsto
  U_N(x_t)$ is non-decreasing by (i) of
  Proposition~\ref{prop:evi-properties} then yields \eqref{eq:evi-int}
  for all $z\in D(S)$. Conversely, differentiating \eqref{eq:evi-int}
  yields \eqref{eq:knflow}. The fact that \eqref{eq:evi-int} holds for
  all $z\in D(S)$ if and only if it holds for all $z\in D$ is
  obvious. Similar arguments show the equivalence of
  Definition~\ref{def:KNflow} with (i) and (iii).
\end{proof}

An important property of $\evikn$ flows is the following expansion
bound.

\begin{theorem}\label{thm:contraction}
  Let $(x_t),(y_t)$ be two $\evikn$ gradient flows of $S$ starting
  from $x_0$ resp. $y_0$. Then for all $s,t\geq0$:
  \begin{align}\label{eq:contraction}
    %\skn{\frac12 d(x_t,y_s)}^2~\leq~e^{-K(s+t)}\left[\skn{\frac12 d(x_0,y_0)}^2 + N E_K(s+t)\frac{\big(\sqrt{t}-\sqrt{s}\big)^2}{2(s+t)}\right]\;.
    \skn{\frac12 d(x_t,y_s)}^2~\leq~\e^{-K(s+t)}\skn{\frac12 d(x_0,y_0)}^2%\\\nonumber
                                    + \frac{N}{K}\Big(1-\e^{-K(s+t)}\Big)\frac{\big(\sqrt{t}-\sqrt{s}\big)^2}{2(s+t)}\;.
  \end{align}
\end{theorem}

\begin{proof}
  Let us fix $s,t>0$. Choose $\lambda,r>0$ such that $\lambda r=t$ and
  $\lambda^{-1}r=s$, i.e. $\lambda=\sqrt{\frac{t}{s}}$ and
  $r=\sqrt{ts}$. From \eqref{eq:evi-int} applied to $(x_t)$ with
  $z=y_{\lambda^{-1}r}$ and $t_0=\lambda r,t_1=\lambda(r+\eps)$ for
  some $\eps>0$ we obtain
  \begin{align}\label{eq:contraction1}
   \frac{N}{2}\frac{U_N(y_{\lambda^{-1}r})}{U_N(x_{\lambda(r+\eps)})}~\leq~\frac{N}{2}&-\frac{1}{e_{-K}(\lambda\eps)}\skn{\frac12 d(x_{\lambda(r+\eps)},y_{\lambda^{-1}r})}^2\\\nonumber
                                                                         &+\frac{1}{e_{K}(\lambda\eps)}\skn{\frac12 d(x_{\lambda r},y_{\lambda^{-1}r})}^2\;.
  \end{align}
  Similarly, choosing $z=x_{\lambda(r+\eps)}$ and
  $t_0=\lambda^{-1}r,t_1=\lambda^{-1}(r+\eps)$ and applying
  \eqref{eq:evi-int} to $(y_s)$ we obtain
  \begin{align}\label{eq:contraction2}
   \frac{N}{2}\frac{U_N(x_{\lambda(r+\eps)})}{U_N(y_{\lambda^{-1}(r+\eps)})}~\leq~\frac{N}{2}&-\frac{1}{e_{-K}(\lambda^{-1}\eps)}\skn{\frac12 d(y_{\lambda^{-1}(r+\eps)},x_{\lambda(r+\eps)})}^2\\\nonumber
                                                                         &+\frac{1}{e_{K}(\lambda^{-1}\eps)}\skn{\frac12 d(y_{\lambda^{-1}r},x_{\lambda(r+\eps)})}^2\;.
  \end{align}
  Multiplying \eqref{eq:contraction1} and \eqref{eq:contraction2}
  after taking square roots and using Young's inequality,
  $2\sqrt{ab}\leq\lambda a + \lambda^{-1} b$, we deduce the estimate
  \begin{align}\label{eq:contraction3}
    &N\sqrt{\frac{U_N(y_{\lambda^{-1}r})}{U_N(y_{\lambda^{-1}(r+\eps)})}}~\leq~\frac{N}{2}(\lambda^{-1}+\lambda)\\\nonumber
    &+ \skn{\frac12 d(y_{\lambda^{-1}r},x_{\lambda(r+\eps)})}^2\left[\frac{\lambda^{-1}}{e_{K}(\lambda^{-1}\eps)}-\frac{\lambda}{e_{-K}(\lambda\eps)}\right]\\\nonumber
    & +\skn{\frac12 d(x_{\lambda r},y_{\lambda^{-1}r})}^2\left[\frac{\lambda}{e_{K}(\lambda\eps)}-\frac{\lambda^{-1}}{e_{-K}(\lambda^{-1}\eps)}\right]\\\nonumber
&-\frac{\lambda^{-1}\eps}{e_{-K}(\lambda^{-1}\eps)}\frac{1}{\eps}\left[\skn{\frac12 d(y_{\lambda^{-1}(r+\eps)},x_{\lambda(r+\eps)})}^2-\skn{\frac12 d(x_{\lambda r},y_{\lambda^{-1}r})}^2\right]\;.
  \end{align}
  Note that as $\eps\to0$ we have
  \begin{align*}
    \frac{e_{-K}(\lambda^{-1}\eps)}{\lambda^{-1}\eps}\to1\quad \text{and}\quad \left[\frac{\lambda^{-1}}{e_{K}(\lambda^{-1}\eps)}-\frac{\lambda}{e_{-K}(\lambda\eps)}\right]\to-\frac{K}{2}(\lambda+\lambda^{-1})\;.
  \end{align*}
  Hence, if we consider the function $g:\R_+\to\R$ given by
  \begin{align*}
    g(\tau)~=~\frac{2}{N}\skn{\frac12 d(x_{\lambda \tau},y_{\lambda^{-1}\tau})}^2
  \end{align*}
  and take the limit as $\eps\searrow0$ in \eqref{eq:contraction3} we
  obtain
  \begin{align*}
   \left.\frac{\mathrm{d}^+}{\mathrm{d}\tau}\right\vert_{\tau=r} g(\tau)~\leq~-K(\lambda+\lambda^{-1})g(r) + (\lambda+\lambda^{-1}-2)\;.
  \end{align*}
  By an application of Gronwall's lemma we deduce that
  \begin{align*}
   g(r)~\leq~e^{-K(\lambda+\lambda^{-1})r}\Big[g(0) + \frac{\lambda+\lambda^{-1}-2}{(\lambda+\lambda^{-1})}e_K\big((\lambda+\lambda^{-1})r\big)\Big]\;.
  \end{align*}
  Rewriting $r,\lambda$ in terms of $s,t$ finally yields
  \eqref{eq:contraction}.
\end{proof}

\begin{remark}\label{rem:contraction-infinitesimal}
  In the limit $d(x_0,y_0)\to0$ and $s\to t$ the contraction estimate
  \eqref{eq:contraction} reads asymptotically as follows:
  \begin{align}\label{eq:contraction-infinitesimal}
    d(x_t,y_s)^2~\leq~e^{-2Kt}d(x_0,y_0)^2 + \frac{N}{K}\frac{1-e^{-2Kt}}{4t^2}\cdot\abs{s-t}^2%\\\nonumber
     + o\big(d(x_0,y_0)^2+\abs{t-s}^2\big)\;.
  \end{align}
\end{remark}

\begin{corollary}\label{cor:contraction}
  For each $x_0\in\overline{D(S)}$ there exist at most one $\evikn$
  gradient flow of $S$ starting from $x_0$. The maps $P_t:x_0\mapsto x_t$,
  where $(x_t)$ is the unique gradient flow starting from $x_0$
  constitute a continuous semigroup defined on a closed (possibly
  empty) subset of $\overline{D(S)}$.
\end{corollary}
The previous expansion estimate in Theorem~\ref{eq:contraction}
implies a slightly weaker estimate directly for the distance $d$ not
involving the functions $\mathfrak{s}_{K/N}$. More precisely, we have
the following:
\begin{proposition}\label{prop:simp-control}
  The expansion bound \eqref{eq:contraction} implies the following
  bound: For each $x_0 , x_1 \in X$ and $s , t \ge 0$, $x_t := P_t
  x_0$ and $y_s : = P_s y_0$ satisfies
\begin{align*}
d ( x_t , y_s )^2
&~\le~
\e^{-K \tau(s,t)} d ( x_0 , y_0 )^2 + 2N
\frac{ 1 - \e^{-K \tau(s,t)}}{ K \tau(s,t) }
\big( \sqrt{t} - \sqrt{s} \big)^2 \; ,
\end{align*}
where $\tau (s,t) = 2( t + \sqrt{ts} + s )/3$.
In particular, setting
$t=s$ yields the following estimate:
  \begin{align}\label{eq:WC0} d ( x_t , y_t ) & \le
    \e^{-Kt} d ( x_0 , y_0 ).
  \end{align}
\end{proposition}

\begin{proof}
  For $0 < s' < t'$, let $\Phi : [ 0, 1 ] \to [ s' , t' ]$ be given by
  $ \Phi (r) := ( \sqrt{s'} + ( \sqrt{t'} - \sqrt{s'} ) r )^2 $. Let
  $( \gamma_u )_{u \in [0,1]}$ be a constant speed geodesic.  By
  \eqref{eq:contraction}, there exists $C_1 > 0$ such that
\begin{equation} \label{eq:small}
d ( P_{r} \gamma_{u} , P_{r'} \gamma_{u'} )
\le
C_1 \left( | u - u' | + | \sqrt{r} - \sqrt{r'} | \right)
\end{equation}
when $| u - u' |$ and
$| \sqrt{r} - \sqrt{r'} |$ is sufficiently small.
By the convexity of $z \mapsto z^2$ on $\R$,
for $k \in \N$,
\begin{align*}
d ( P_{t'}\gamma_1 , P_{s'} \gamma_0 )^2
\le
\sum_{j=1}^{k}
d \big(
  P_{\Phi ((j-1)/k)} \gamma_{(j-1)/k} ,
  P_{\Phi (j/k)} \gamma_{j/k}
\big)^2 k.
\end{align*}
By virtue of \eqref{eq:small}, we have
% By \eqref{eq:small} and \eqref{eq:Wasserstein-contraction},
% there exist constants $C_2 , C_3 > 0$ such that
% for sufficiently large $k$,
\begin{align*}
\lim_{k \to \infty} &
\sum_{j=1}^{k}
d \big(
  P_{\Phi ((j-1)/k)} \gamma_{(j-1)/k} ,
  P_{\Phi (j/k)} \gamma_{j/k}
\big)^2 k
\\
& \le
4 \lim_{k \to \infty}
\sum_{j=1}^{k}
\skn{ \frac12
d \big(
  P_{\Phi ((j-1)/k)} \gamma_{(j-1)/k} ,
  P_{\Phi (j/k)} \gamma_{j/k}
\big)
}^2 k
\\
& \le
4 \lim_{k \to \infty} \Bigg[
\sum_{j=1}^k \e^{-K ( \Phi (j/k) + \Phi ((j-1)/k) )}
\skn{
  \frac12
  d ( \gamma_{(j-1)/k} , \gamma_{j/k} )
}^2 k
\\
& \hspace{6em} +
\frac{N}{2} \sum_{j=1}^k
\frac{ 1 - \e^{-K( \Phi (j/k) + \Phi ( (j-1)/k ) )} }
{
  K ( \Phi (j/k) + \Phi ( ( j-1 )/k ) )
}
\frac{ \big( \sqrt{ t' } - \sqrt{ s' } \big)^2 }{k}
\Bigg]
\\
& =
\int_0^1 \e^{-2 K \Phi (r)} dr d ( \gamma_0 , \gamma_1 )^2
+ 2N
\int_0^1 \frac{ 1 - \e^{-K \Phi (r)}}{ K \Phi (r)} dr
( \sqrt{t'} - \sqrt{s'} )^2 .
\end{align*}
Let $\lambda \ge 1$, $\tau , h > 0$,
$s' = \lambda^{-1} ( \tau + h )$,
$t' = \lambda ( \tau + h )$,
$\gamma_0 := P_{\lambda^{-1} r} y_0$ and $\gamma_1 : = P_{\lambda r} x_0$.
Then the last inequality implies
\begin{align*}
\frac{\mathrm{d}^+}{\mathrm{d}\tau}
d ( x_{\lambda \tau} , y_{\lambda^{-1} \tau} )^2
\le
- \frac{2K}{3} ( \lambda + \lambda^{-1} + 1 )
d ( x_{\lambda \tau} , y_{\lambda^{-1} \tau} )^2
+
2N
( \sqrt{\lambda} - \sqrt{ \lambda^{-1} } )^2 .
\end{align*}
Thus the conclusion follows from this estimate as in the proof of
Theorem~\ref{thm:contraction}.
\end{proof}

We now investigate the relation between the Evolution Variational
Inequality and geodesic convexity of the functional $S$.

\begin{theorem}\label{thm:eviconvex}
  Assume that for every starting point $x_0\in \overline{D(S)}$ the
  $\evikn$ flow for $S$ exists. Then $S$ is strongly $(K,N)$-convex.
\end{theorem}

\begin{proof}
  Let $P$ denote the $\evikn$ gradient flow semigroup of $S$. We treat
  the case $K\neq0$ first. So let $(\gamma_s)_{s\in[0,1]}$ be a
  constant speed geodesic. Let us fix $s\in[0,1], t>0$ and set
  $\gamma_s^t:=P_t\gamma_s$. We can assume that
  $d:=d(\gamma_0,\gamma_1)\neq0$. Using the identity \eqref{eq:sc-trick}
  % \begin{align*}
  %   \frac{2}{N}\skn{\frac12\theta}^2~=~\frac{1}{K}\big(1-\ckn{\theta})\;,
  % \end{align*}
  we see that \eqref{eq:evi-int} can be rewritten as
  \begin{align}\label{eq:evi-int-alt}
    \frac{U_N(z)}{U_N(P_{t_1}x)}e_K(t_1-t_0)~\leq~\frac{1}{K}\Big[\e^{K(t_1-t_0)}\ckn{d(P_{t_1}x,z)}%\\\nonumber
                                                                 -\ckn{d(P_{t_0}x,z)}\Big]\;.
  \end{align}
  Using \eqref{eq:evi-int-alt} with $t_0=0,t_1=t,x=\gamma_s$ and
  $z=\gamma_0$ respectively $z=\gamma_1$ we immediately obtain
  \begin{align*}
    \sigkn{1-s}{d}\cdot U_N&(\gamma_0) + \sigkn{s}{d}\cdot U_N(\gamma_1)\\
      ~\leq~\frac{U_N(\gamma^t_s)}{K\cdot e_K(t)}\Big[&\sigkn{1-s}{d}\cdot\Big(\e^{Kt}\ckn{d(\gamma_s^t,\gamma_0)}-\ckn{d(\gamma_s,\gamma_0)}\Big)\\
      +&\sigkn{s}{d}\cdot\Big(\e^{Kt}\ckn{d(\gamma_s^t,\gamma_1)}-\ckn{d(\gamma_s,\gamma_1)}\Big)\Big]\;.
  \end{align*}
  Let $A$ denote the term in square brackets in the last
  inequality. The claim follows if we show that for $t$ small enough
  we have $A\leq K\cdot e_K(t)=\e^{Kt}-1$ if $K>0$ and $A\geq \e^{Kt}-1$
  if $K<0$. Using the fact that $d(\gamma_s,\gamma_{s'})=\abs{s-s'}d$,
  we first find
  \begin{align*}
    A~=&~\frac{\e^{Kt}}{\skn{d}}\Big[\skn{(1-s)d}\cdot\ckn{d(\gamma_s^t,\gamma_0)}+\skn{sd}\cdot\ckn{d(\gamma_s^t,\gamma_1)}\Big]\\
       &-\frac{1}{\skn{d}}\Big[\skn{(1-s)d}\cdot\ckn{sd}+\skn{sd}\cdot\ckn{(1-s)d}\Big]\\
      :=&~ A_1+A_2    \;.
\end{align*}
By the angle sum identity for $\sin$ (resp. $\sinh$) we have
$A_2=-1$. To see that $A_1\leq \e^{Kt}$ (resp. $A_1\geq \e^{Kt}$), we
observe the following fact, which is easily verified using the
angle sum identities for trigonometric or hyperbolic functions: If
$\alpha,\alpha'\geq 0$ and
$\eps,\eps'\in[-\frac{\pi}{2},\frac{\pi}{2}]$ such that
$\eps+\eps'\geq0$, then, putting $\beta=\alpha+\eps,
\beta'=\alpha'+\eps'$, we have that
\begin{align*}
  \sin(\a)\cos(\beta') + \cos(\beta)\sin(\a')~&\leq~\sin(\a+\a')\;,\\
  \sinh(\a)\cosh(\beta') +
  \cosh(\beta)\sinh(\a')~&\geq~\sinh(\a+\a')\;.
\end{align*}
To conclude, we apply this with $\alpha=(1-s)d$, $\alpha'=sd$ and
$\eps=d(\gamma_s^t,\gamma_1)-(1-s)d$,
$\eps'=d(\gamma_s^t,\gamma_0)-sd$ and note that $\eps+\eps'\geq0$ by
the triangle inequality.

Finally, we treat the case $K=0$. By Lemma~\ref{lem:consistentKN} $P$
is a $\evi_{K',N}$ flow for every $K'<0$. Thus by the previous argument
\eqref{eq:defknconvex} holds with $K'$ instead of $K$ and we can pass
to the limit as $K'\nearrow0$.
\end{proof}

\section{Entropic and Riemannian curvature-dimension conditions}\label{sec:cds}
\subsection{The entropic curvature-dimension condition}
\label{sec:cdkn}

In this section we introduce a new curvature-dimension condition for
metric measure spaces based on $(K,N)$-convexity of the entropy on the
Wasserstein space.

Let $(X,d,m)$ be a metric measure space, i.e. $(X,d)$ is a complete
and separable metric space and $m$ is a locally finite,
$\sigma$-finite Borel measure on $X$.  We denote by $\cP_2(M,d)$ the
$L^2$-Wasserstein space over $(X,d)$, i.e. the set of all Borel
probability measures $\mu$ satisfying
\begin{align*}
\int d(x_0,x)^2\mu(\dd x)~<\infty
\end{align*}
for some, hence any, $x_0\in X$. The subspace of all measures
absolutely continuous w.r.t. $m$ is denoted by $\cP_2(X,d,m)$. The
$L^2$-Wasserstein distance between $\mu_0,\mu_1\in\cP_2(X,d)$ is
defined by
\begin{align*}
  W_2(\mu_0,\mu_1)^2~=~\inf \int d(x,y)^2\dd q(x,y)\;,
\end{align*}
where the infimum is taken over all Borel probability measures $q$ on
$X\times X$ with marginals $\mu_0$ and $\mu_1$.  Let us denote by
$\geo(X)=\{\gamma:[0,1]\to X\ |\ \gamma \text{ const. speed
  geodesic}\}$ the space of constant speed geodesics in $X$ equipped
with the topology of uniform convergence. For any $t\in[0,1]$ we
denote by $e_t:\geo(X)\to X$ the evaluation map $\gamma\mapsto
\gamma_t$. Recall that a \emph{dynamic optimal coupling} between
$\mu_0,\mu_1\in\cP_2(X,d)$ is a probability measure
$\pi\in\cP(\geo(X))$ such that $(e_0,\e_1)_\#\pi$ is an optimal
coupling of $\mu_0,\mu_1$. The curve $(\mu_t)_{t\in[0,1]}$ with
$\mu_t=(e_t)_\#\pi$ is then a geodesic in $\cP_2(X,d)$ connecting
$\mu_0$ to $\mu_1$. Moreover, by \cite[Lem.~I.2.11]{S06}, for each
geodesic $\Gamma:[0,1] \to \cP_2(X,d)$, there exists a probability
measure $\pi$ on $\geo(X)$ such that $\Gamma_t = (e_t)_\#\pi$ for all
$t\in[0,1]$.
%   Indeed, let
%   $\Gamma:[0,1]\to\cP_2(X,d)$ be a geodesic. By \cite[Lem.~I.2.11]{S06}
%   there exists a probability measure $\pi$ on $\geo(X)$ such that
%   $\Gamma_t=(e_t)_\#\pi$ for all $t\in[0,1]$. Here $\geo(X)$ denotes
%   the set of all constant speed geodesics $\gamma:[0,1]\to X$ endowed
%   with the topology of uniform convergence and $e_t:\geo(X)\to X$
%   denotes the evaluation at time $t$.

Given a measure $\mu\in\cP_2(X,d)$ we define its relative entropy by
\begin{align*}
  \ent(\mu)~:=~\int \rho\log\rho\dd m\;,
\end{align*}
if $\mu=\rho m$ is absolutely continuous w.r.t. $m$ and
$(\rho\log\rho)_+$ is integrable. Otherwise we set
$\ent(\mu)=+\infty$. The subset of probability measures with finite
entropy will be denoted by $\cP_2^*(X,d,m)$. Moreover, for a number
$N\in(0,\infty)$ we introduce the functional
$U_N:\cP_2(X,d)\to[0,\infty]$ by
\begin{align*}
  U_N(\mu)~:=~\exp\left(-\frac{1}{N}\ent(\mu)\right)\;.
\end{align*}

\begin{definition}\label{def:ecdkn}
  Given two numbers $K\in\R$, $N\in(0,\infty)$ we say that a metric
  measure space $(X,d,m)$ satisfies the \emph{entropic
    curvature-dimension condition} $\ecdkn$ if and only if for each
  pair $\mu_0,\mu_1\in\cP^*_2(X,d,m)$ there exists a constant speed
  geodesic $(\mu_t)_{t\in[0,1]}$ in $\cP_2^*(X,d,m)$ connecting
  $\mu_0$ to $\mu_1$ such that for all $t\in[0,1]$:
  \begin{align}\label{eq:ecdkn}
    U_{N}(\mu_t)~\geq~\sigkn{1-t}{W_2(\mu_0,\mu_1)} U_{N}(\mu_0) +
    \sigkn{t}{W_2(\mu_0,\mu_1)} U_{N}(\mu_1)\;.
  \end{align}
  If \eqref{eq:ecdkn} holds for any constant speed geodesic
  $(\mu_t)_{t\in[0,1]}$ in $\cP^*_2(X,d,m)$ we say that $(X,d,m)$ is a
  \emph{strong} $\ecdkn$ space.
\end{definition}

In other words, the $\ecdkn$-condition means that the entropy is
$(K,N)$-convex along Wasserstein geodesic. As an immediate consequence
of Lemma~\ref{lem:convexconsistent} we obtain the following
consistency result.

\begin{lemma}\label{lem:consistent}
  If $(X,d,m)$ satisfies the $\ecdkn$ condition, then it also
  satisfies $\cd^e(K',N')$ for any $K'\leq K$ and $N'\geq
  N$. Moreover, it satisfies the $\cd(K,\infty)$ condition.
\end{lemma}

Using similar arguments as in the case of the $\cd(K,\infty)$
condition introduced in \cite{S06} it is immediate to check that
$\ecdkn$ is invariant under isomorphisms of metric measure
spaces. Moreover, adapting \cite[Thm.~I.4.20]{S06}, one can check that
it is stable under convergence of metric measure spaces in the
transportation distance $\D$, also introduced in \cite{S06}.
% \begin{theorem}[Stability]\label{thm:cde-stable}
%     $\ecdkn$ is stable under Gromov--Hausdorff convergence
% \end{theorem}

% \begin{proof}

% \end{proof}

As an application of the additivity of $(K,N)$-convexity we note
the following

\begin{proposition}[Weighted spaces]\label{prop:weighted-cde}
  Let $(X,d,m)$ be a metric measure space satisfying $\ecdkn$ and let
  $V:X\to\R$ be a measurable function bounded from below that is
  strongly $(K',N')$-convex in the sense of
  Definition~\ref{def:knconvex}. Then $(X,d,\e^{-V}m)$ satisfies
  $\cd^e(K+K',N+N')$. In particular, if $(X,d,m)$ satisfies strong
  $\ecdkn$, then $(X,d,\e^{-V}m)$ also satisfies strong
  $\cd^e(K+K',N+N')$.
\end{proposition}

\begin{proof}
  We will first show that the functional $\overline
  V:\cP_2(X,d)\to(-\infty,\infty]$ defined by $\overline V (\mu)=\int
  V\dd \mu$ is strongly $(K',N')$-convex on $\cP_2(X,d)$. Let $\pi \in
  \cP ( \geo (X) )$ be an dynamic optimal coupling.  and set $\mu_t
  = (e_t)_\# \pi$.  From the $(K',N')$-convexity of $V$ we have for
  any $\gamma\in\geo(X)$ and $t\in[0,1]$:
  \begin{align}\label{eq:weighted-cde}
    \e^{-V(\gamma_t)/N'}~\geq~ \sigma^{(1-t)}_{K'/N'}\big(d(\gamma_0,\gamma_1)\big)\cdot\e^{-V(\gamma_t)/N'} + \sigma^{(t)}_{K'/N'}\big(d(\gamma_0,\gamma_1)\big)\cdot\e^{-V(\gamma_t)/N'}\;.
  \end{align}
  Take the logarithm on both sides of \eqref{eq:weighted-cde}. By
  virtue of Lemma \ref{lem:convexG}, we can use Jensen's inequality
  when integrating it w.r.t. $\pi$ to obtain
  \begin{align*}
%     -\frac{1}{N'}\overline V(\mu_t)~=~-\frac{1}{N'}\int V(\gamma_t)\dd\pi(\gamma)~&\geq~\int G\Big(-\frac{1}{N'} V(\gamma_0),-\frac{1}{N'} V(\gamma_1),d(\gamma_0,\gamma_1)^2\Big)\dd \pi(\gamma)\\
%                                   &\geq G\Big(-\frac{1}{N'} \overline V(\Gamma_0),-\frac{1}{N'} \overline V(\Gamma_0),W_2(\Gamma_0,\Gamma_1)^2\Big)\;.
    -\frac{1}{N'}\overline V(\Gamma_t)~=~-\frac{1}{N'}\int V(\gamma_t)\dd\pi(\gamma)~&\geq~\int G_t \Big(-\frac{1}{N'} V(\gamma_0),-\frac{1}{N'} V(\gamma_1), \frac{K'}{N'}d(\gamma_0,\gamma_1)^2\Big)\dd \pi(\gamma)\\
                                  &\geq G_t \Big(-\frac{1}{N'} \overline V(\mu_0),-\frac{1}{N'} \overline V(\mu_1),\frac{K'}{N'} W_2(\mu_0,\mu_1)^2\Big)\;.
  \end{align*}
  Taking the exponential again then yields the claim. By the lower
  boundedness of $V$ we have
  $\cP_2(X,d,\e^{-V}m)\subset\cP_2(X,d,m)$. Now the assertion of the
  proposition is a consequence of the observation
  \begin{align*}
    \ent(\mu\vert\e^{-V}m)~=~\ent(\mu\vert m) + \overline V(
\mu)
  \end{align*}
  and Lemma~\ref{lem:convex-sum}. The latter assertion is obvious from
  the proof.
\end{proof}

We will now derive some first geometric consequences of the entropic
curvature-dimension condition.

\begin{proposition}[Generalized Brunn--Minkowski inequality]\label{prop:BMI}
  Assume that $(X,d,m)$ satisfies the condition $\ecdkn$ with
  $N\geq1$. Then for all measurable sets $A_0,A_1\subset X$ with
  $m(A_0),m(A_1)>0$ and all $t\in[0,1]$ we have
  \begin{align}\label{eq:BMI}
    \bar{m}(A_t)^{1/N}~\geq~\sigkn{1-t}{\Theta}\cdot m(A_0)^{1/N} + \sigkn{t}{\Theta}\cdot m(A_1)^{1/N}\;,
  \end{align}
  where $\bar{m}$ is the completion of $m$, $A_t$ denotes the set of
  $t$-midpoints and $\Theta$ the minimal/maximal distance between
  points in $A_0$ and $A_1$, i.e.
  \begin{align*}
    A_t~&=~\{\gamma_t : \gamma:[0,1]\to X \text{ geodesic s.t. }\gamma_0\in A_0, \gamma_1\in A_1\}\;,\\
   \Theta~&=~
    \begin{cases}
      \inf_{x_0\in A_0,x_1\in A_1}d(x_0,x_1)\;, & K\geq 0\;,\\
      \sup_{x_0\in A_0,x_1\in A_1}d(x_0,x_1)\;, & K< 0\;.
    \end{cases}
  \end{align*}
\end{proposition}

\begin{proof}
  We first prove the assertion under the assumption that
  $m(A_0),m(A_1)<\infty$, the general case then follows by
  approximating the sets $A_0,A_1$ by sets of finite volume. Applying
  the condition $\cd^e(K,N)$ to $\mu_i=m(A_i)^{-1}\one_{A_i}m$ for $i=0,1$
  yields
  \begin{align}\label{eq:BMI1}
    U_{N}(\Gamma_t)~\geq~\sigkn{1-t}{W_2(\mu_0,\mu_1)}\cdot m(A_0)^{1/N} + \sigkn{t}{W_2(\mu_0,\mu_1)}\cdot m(A_1)^{1/N}\;,
  \end{align}
  where $\mu_t=\rho_tm$ is the $t$-midpoint of a geodesic connecting
  $\mu_0$ and $\mu_1$. Since $\mu_t$ is concentrated on $A_t$, which
  is a Souslin set, a double application of Jensen's inequality gives
  that
  \begin{align*}
    U_{N}(\mu_t)~&=~\exp\Big(-\frac{1}{N}\int\log\rho_t\dd\mu_t\Big)~\leq~\int\rho_t^{-1/N}\dd\mu_t\\
                ~&=~\int\limits_{A_t}\rho_t^{1-1/N}\dd \bar{m}~\leq~\bar{m}(A_t)^{1/N}\;.
  \end{align*}
  Hence \eqref{eq:BMI} follows by noting that
  $\theta\mapsto\sigkn{t}{\theta}$ is increasing if $K\geq0$ and
  decreasing if $K<0$ and that $W_2(\mu_0,\mu_1)\geq\Theta$
  (resp. $\leq\Theta$).
\end{proof}

The Brunn--Minkowski inequality entails further geometric consequences
like a Bishop--Gromov type volume growth estimate and a generalized
Bonnet--Myers theorem. The following results can be deduced from
Proposition~\ref{prop:BMI} using similar arguments as in \cite{S06}
and replacing the coefficients $\tau^{(t)}_{K/N}(\cdot)$ by
$\sigkn{t}{\cdot}$.

\begin{remark}\label{rem:nonsharp}
  The estimates presented below are not sharp, yet they provide
  necessary local compactness results for example. We will see below
  that under the assumption that $(X,d,m)$ is non-branching the
  $\ecdkn$ condition is equivalent to the $\cdskn$ condition. It has
  been proven by Cavaletti \& Sturm \cite{CS12} that under the same
  assumption $\cdskn$ implies the measure contraction property
  $\text{MCP}(K,N)$ from which a sharp Bishop--Gromov and Lichnerowicz
  inequality can be derived, see \cite{S06}.
\end{remark}

To state the volume growth estimate we introduce the following
notation. Given a metric measure space $(X,d,m)$ and a point
$x_0\in\supp[m]$ we denote by
\begin{align*}
  v(r)~:=~m(\overline{B_r(x_0)})
\end{align*}
the volume of the closed ball of radius $r$ around $x_0$. Moreover, we
set
\begin{align*}
  s(r)~:=~\limsup\limits_{\delta\to0}\frac{1}{\delta}m(\overline{B_{r+\delta}(x_0)}\setminus B_{r}(x_0))
\end{align*}
for the volume of the corresponding sphere.

\begin{proposition}[Generalized Bishop--Gromov inequality]\label{prop:BGI}
  Assume that $(X,d,m)$ satisfies the condition $\ecdkn$ with
  $N\geq1$. Then each bounded closed set $M\subset\supp[m]$ is compact
  and has finite volume. More precisely, for each $x_0\in\supp[m]$ and
  $0<r<R\leq\pi\sqrt{N/(K\vee0)}$,
  \begin{align}\label{eq:BGI}
    \frac{s(r)}{s(R)}~\geq~\left(\frac{\skn{r}}{\skn{R}}\right)^{N}
 \quad \text{and} \quad
   \frac{v(r)}{v(R)}~\geq~\frac{\int_0^r\skn{t}^{N}\dd t}{\int_0^R\skn{t}^{N}\dd t}\;.
  \end{align}
\end{proposition}

\begin{corollary}[Generalized Bonnet--Myers theorem]\label{cor:BMT}
  If $(X,d,m)$ satisfies the condition $\ecdkn$ with $K>0$ and $N\geq
  1$, then the support of $m$ is compact and its diameter $L$ can be
  bounded as $L\leq\pi\sqrt{N/K}$.
\end{corollary}

\begin{remark}\label{rem:exp-int}
  $\ecdkn$ or $\cd (K, \infty)$ yields that $\cP (X,d)$ is a length
  space and hence so is $( \supp m , d )$ \cite[Rem.~I.4.6(iii),
  Prop.~2.11(iii)]{S06}.  Thus, by the local compactness ensured in
  Proposition~\ref{prop:BGI}, if $(X,d,m)$ is a $\ecdkn$ space then
  $(\supp m , d)$ and hence $\cP_2 ( \supp m, d )$ is a geodesic space
  (see e.g.~\cite[Thm.~2.5.23]{BBI01}). In addition, the volume
  growth estimate \eqref{eq:BGI} implies in particular that for any
  $x_0\in X$ and $c>0$:
  \begin{align}\label{eq:exp-int}
    \int_X\e^{-cd(x_0,x)^2}\dd m(x)~<~\infty\;.
  \end{align}
  It is well known that the latter implies that $\ent$ does not take
  the value $-\infty$ on $\cP_2(X,d)$ and is lower semi-continuous
  w.r.t. $W_2$ (see e.g.~\cite[Sec.~7]{AGS11a}).  Thus, when $\supp m
  = X$, Definition~\ref{def:ecdkn} fits well into the setting of
  Section~\ref{sec:metric}, where we assumed these additional
  regularity properties.
\end{remark}

It turns out that under mild assumptions the modified
curvature-dimension condition $\ecdkn$ is equivalent to the reduced
curvature-dimension condition $\cdskn$ introduced in \cite{BS10}. We
recall here the definition. Denote by $\cP_\infty(X,d,m)$ the set of
measures in $\cP_2(X,d,m)$ with bounded support.

\begin{definition}\label{def:reduced-cdkn}
  We say that a metric measure space $(X,d,m)$ satisfies the
  \emph{reduced curvature-dimension condition} $\cdskn$ if and only if
  for each pair $\mu_0=\rho_0 m,\mu_1=\rho_1 m\in\cP_\infty(X,d,m)$
  there exist an optimal coupling $q$ of them and a geodesic
  $(\mu_t)_{t\in[0,1]}$ in $\cP_\infty(X,d,m)$ connecting them such
  that for all $t\in[0,1]$ and $N'\geq N$:
  \begin{align}\label{eq:cdskn}
    \int\rho_t^{-\frac{1}{N'}}\dd\mu_t~\geq~\int\limits_{X\times X}\Big[&\sigknp{1-t}{d(x_0,x_1)}\rho_0(x_0)^{-\frac{1}{N'}}\\\nonumber
    &+ \sigknp{t}{d(x_0,x_1)}\rho_1(x_1)^{-\frac{1}{N'}}\Big]\dd q(x_0,x_1) \;.
  \end{align}
  If \eqref{eq:cdskn} holds for any geodesic $(\mu_t)_{t\in[0,1]}$ in
  $\cP_\infty(X,d,m)$ we say that $(X,d,m)$ is a \emph{strong}
  $\cdskn$ space.
\end{definition}
The assumption we need to prove equivalence of the different
curvature-dimension conditions is the following weak form of
non-branching.
\begin{definition}\label{def:ess-non-branch}
  We say that a metric measure space $(X,d,m)$ is \emph{essentially
    non-branching} if any dynamic optimal coupling
  $\pi\in\cP(\geo(X))$ between two absolutely continuous measures is
  supported in a set of non-branching geodesics, i.e. there exists
  $A\subset\geo(X)$ such that $\pi(A)=1$ and for all $\gamma,\tilde
  \gamma\in A$:
  \begin{align*}
    \gamma_t=\tilde\gamma_t\quad \forall t\in[0,\eps] \text{ for some } \eps>0\ \Rightarrow\ \gamma=\tilde\gamma\;.
  \end{align*}
\end{definition}

This condition has been introduced in \cite{RS12} and it has been
shown that strong $\cd(K,\infty)$ spaces are essentially
non-branching. It has also been noted there that the essential
non-branching condition is equivalent to the following apparently
stronger condition: Every dynamic optimal coupling $\pi$ between
absolutely continuous measures is concentrated on a set of geodesics
that do not meet at intermediate times, i.e. there is $A'\subset
\geo(X)$ such that $\pi(A')=1$ and for all $\gamma,\tilde\gamma\in
A'$:
\begin{align*}
  \gamma_t=\tilde\gamma_t\quad \text{ for some } t\in(0,1)\
  \Rightarrow\ \gamma=\tilde\gamma\;.
\end{align*}
Indeed, assuming the existence of a dynamic optimal coupling where
such crossings happen with positive probability, one can reshuffle the
geodesics before and after the crossing to produce a dynamic optimal
coupling of the same marginals where branching happens with positive
probability, contradicting the essentially non-branching assumption.

An immediate consequence of this observation is the following adaption
of \cite[Lem.~2.8]{BS10}.

\begin{lemma}\label{lem:non-branch}
  Let $(X,d,m)$ be an essentially non-branching metric measure space
  and let $\pi$ be a dynamic optimal coupling. Assume that
  $\pi=\sum_{k=1}^n\a_k\pi^k$ for suitable $\a_k>0$ and dynamic
  optimal couplings $\pi^k$. For given $t\in(0,1)$ and $i\in\{0,t\}$
  we set $\mu_i^k=(e_i)_\#\pi^k$. If the family $\{\mu^k_0\}_k$ is
  mutually singular, then also the family $\{\mu_t^k\}_k$ is mutually
  singular.
\end{lemma}

\begin{theorem}\label{thm:equiv_as}
  Let $(X,d,m)$ be an essentially non-branching metric measure
  space. Then the following assertions are equivalent:
  \begin{itemize}
  \item[(i)] $(X,d,m)$ satisfies $\cdskn$,
  \item[(ii)] For each pair $\mu_0,\mu_1\in\cP_\infty(X,d,m)$
    % and each
    there is a
    dynamic optimal coupling $\pi$ of them such that we have
    $( e_t )_\# \pi \ll m$ and
    \begin{align}\label{eq:cdkn-as}
      \rho_t(\gamma_t)^{-\frac{1}{N}}~\geq~&\sigkn{1-t}{d(\gamma_0,\gamma_1)}\rho_0(\gamma_0)^{-\frac{1}{N}}
      % \\\nonumber
      % + &
      + \sigkn{t}{d(\gamma_0,\gamma_1)}\rho_1(\gamma_1)^{-\frac{1}{N}}\;,
    \end{align}
    for $\pi$-a.e. $\gamma\in\geo(X)$, where $\rho_t$ denotes the density of
    % the push forward of $\pi$ under the map $\gamma\mapsto\gamma_t$,
    $(e_t)_\# \pi$ w.r.t.~$m$.
  \item[(iii)] $(X,d,m)$ satisfies $\ecdkn$.
  \end{itemize}
\end{theorem}

\begin{proof}
  The equivalence of (i) and (ii) has already been proven in
  \cite[Prop.~2.8]{BS10} under the assumption that $X$ is
  non-branching. Note that the statement (ii) is slightly different
  there but equivalent, since under the non-branching assumption
  $m^2$-a.e. pair of points is connected by a unique geodesic. Under
  the weaker essential non-branching condition the equivalence of (i)
  and (ii) follows by repeating almost verbatim the proof of
  \cite[Prop.~2.8]{BS10} substituting \cite[Lem.~2.6]{BS10}
  with Lemma \ref{lem:non-branch}. For details on the necessary
  modifications see also the implication (iii)$\Rightarrow$(ii) below
  which follows a similar argument.

  (ii)$\Rightarrow$(iii): First note that by an approximation argument
  as in \cite[Lem.~2.11]{BS10} one can show that \eqref{eq:cdkn-as}
  also holds for $\mu_0,\mu_1\in\cP_2(X,d,m)$ not necessarily with
  bounded support. Now fix $\mu_0,\mu_1\in\cP_2(X,d,m)$ and a dynamic optimal
  coupling $\pi$ of them satisfying \eqref{eq:cdkn-as}. Taking logarithms on both sides of \eqref{eq:cdkn-as} we obtain
  \begin{align}\label{eq:equiv-as1}
%     -\frac{1}{N}\log\rho_t(\gamma_t)~\geq~G\Big(-\frac{1}{N}\log\rho_0(\gamma_0),-\frac{1}{N}\log\rho_1(\gamma_1),d(\gamma_0,\gamma_1)^2\Big)\;,
     -\frac{1}{N}\log\rho_t(\gamma_t)~\geq~G_t\Big(-\frac{1}{N}\log\rho_0(\gamma_0),-\frac{1}{N}\log\rho_1(\gamma_1),\frac{K}{N} d(\gamma_0,\gamma_1)^2\Big)\;,
  \end{align}
  where the function $G_t$ is given by
  \eqref{eq:convexG}.
%    \begin{align}\label{eq:G}
%     G(x,y,z)~:=~\log\left[\sigkn{1-t}{\sqrt{z}}e^x + \sigkn{t}{\sqrt{z}}e^y\right]\;.
%   \end{align}
%   From Lemma~\ref{lem:convexG} and the fact that
%   $\sigkn{1-t}{\sqrt{z}}=\sigma^{(1-t)}_{zK/N}(1)$ we deduce that $G$
%   is convex.
  Integrating \eqref{eq:equiv-as1} w.r.t. $\pi$ and using
  Jensen's inequality with the aid of Lemma~\ref{lem:convexG} we obtain
  \begin{align*}
    -\frac{1}{N}\ent\big(\mu_t\big)~\geq~G_t\Big(-\frac{1}{N}\ent\big(\mu_0\big),-\frac{1}{N}\ent\big(\mu_1\big), \frac{K}{N} W_2(\mu_0,\mu_1)^2\Big)\;.%,
  \end{align*}
  Hence \eqref{eq:ecdkn} follows by taking the exponential on both sides.

  (iii)$\Rightarrow$(ii): Here we follow closely the arguments in the
  proof of \cite[Prop.~II.4.2]{S06}. Fix $\mu_0,\mu_1\in\cP_\infty(X,d,m)$ and
  a dynamic optimal coupling $\pi$ of them. Let $\{M_n\}_{n\in\N}$ be a
  $\cap$-stable generator of the Borel $\sigma$-field of $X$ with
  $m(\partial M_n)=0$ for all $n$. For each $n$ consider the disjoint
  covering of $X$ given by the $2^n$ sets $L_1=M_1\cap\dots\cap M_n$,
  $L_2=M_1\cap\dots\cap M_n^c$, $\dots$, $L_{2^n}=M_1^c\cap\dots\cap
  M_n^c$. For fixed $n$ and $i,j=1,\dots,2^n$ we define sets
  $A_{i,j}=\{\gamma\in\geo(X)\ :\ (\gamma_0,\gamma_1)\in L_i\times L_j\}$ and probability measures
  $\mu_0^{i,j},\mu_1^{i,j}$ by
  \begin{align*}
    \mu_0^{i,j}(B)=\a_{i,j}^{-1}\pi\big(\{\gamma_0\in B\cap L_i, \gamma_1\in L_j\}\big)\;,\quad \mu_1^{i,j}(B)=\a_{i,j}^{-1}\pi\big(\{\gamma_0\in L_i,\gamma_1\in B\cap L_j\}\big)\;,
  \end{align*}
  provided that $\a_{i,j}=\pi(A_{i,j})>0$. By (iii) we can choose
  dynamic optimal couplings $\pi^{i,j}$ of them such that
  \begin{align}\label{eq:equiv-as2}
   U_N(\mu^{i,j}_t)~\geq~&\sigkn{1-t}{W_2(\mu^{i,j}_0,\mu^{i,j}_1)}\cdot U_N(\mu^{i,j}_0)
   % \\\nonumber
   % + &
   + \sigkn{t}{W_2(\mu^{i,j}_0,\mu^{i,j}_1)}\cdot U_N(\mu^{i,j}_1)\;,
% \exp\Big(-\frac{1}{N}\int\log\rho^{i,j}_t(\gamma_t)\dd q^{i,j}~\geq~G\Big(-\frac{1}{N}\int\log\rho^{i,j}_0\dd q^{i,j},-\frac{1}{N}\int\log\rho^{i,j}_1\dd q^{i,j},\int d^2\dd q^{i,j}\Big)\;,
  \end{align}
  where $\mu_t^{i,j}=(e_t)_\# \pi^{i,j}$. Define
  \begin{align*}
    \pi^{(n)}:=\sum\limits_{i,j=1}^{2^n}\a_{i,j}\pi^{i,j}\;,\qquad \mu_t^{(n)}=(e_t)_\# \pi^{(n)}\;.
  \end{align*}
  Then $\pi^{(n)}$ is a dynamic optimal coupling of the measures
  $\mu_0,\mu_1$ and $(\mu^{(n)}_t)_{t\in[0,1]}$ is a geodesic between
  them. Since the measures $\mu_0^{i,j}\otimes\mu_1^{i,j}$ are
  mutually singular and $X$ is essentially non-branching, also the
  measures $\mu_t^{i,j}$ are mutually singular for each fixed $t$ by
  Lemma~\ref{lem:non-branch}. We conclude that
  $\rho_t^{(n)}(\gamma_t)=\a_{i,j}\rho_t^{i,j}(\gamma_t)$ on the
  set $A_{i,j}$. Plugging this into \eqref{eq:equiv-as2} and taking
  logarithms on both sides we find
  \begin{align}\label{eq:equiv-as3}
        &-\frac{\a_{i,j}^{-1}}{N}\int\limits_{A_{i,j}}\log\rho^{(n)}_t(\gamma_t)\dd \pi^{(n)}\\\nonumber
&\geq~G_t \Big(-\frac{\a_{i,j}^{-1}}{N}\int\limits_{A_{i,j}}\log\rho_0(\gamma_0)\dd \pi^{(n)},-\frac{\a_{i,j}^{-1}}{N}\int\limits_{A_{i,j}}\log\rho_1(\gamma_1)\dd \pi^{(n)},\a_{i,j}^{-1}\frac{K}{N} \int\limits_{A_{i,j}} d^2(\gamma_0,\gamma_1)\dd \pi^{(n)}\Big)\;.%,
  \end{align}
  Since $\mu_0,\mu_1$ have bounded support, all geodesic in the
  support of the measures $\pi^{(n)}$ stay within a single closed
  bounded set $B$. By Proposition~\ref{prop:BGI} $B$ is compact and
  has finite mass. Hence also the measures $\pi^{(n)}$ are supported
  in a single compact set and thus converge weakly, up to extraction
  of a subsequence, to a dynamic optimal coupling $\tilde \pi$ of
  $\mu_0$ and $\mu_1$. Since $m(\partial M_i)=0$ for all $i$ we deduce
  that
  \begin{equation*}
      \pi\big(\{\gamma_{0}\in M_{i}, \gamma_{1}\in M_{j}\}\big)~=~\lim\limits_{n\to\infty}\pi^{(n)}\big(\{\gamma_{0}\in M_{i}, \gamma_{1}\in M_{j}\}\big)~=~\tilde \pi\big(\{\gamma_{0}\in M_{i}, \gamma_{1}\in M_{j}\}\big)
  \end{equation*}
  for each $i,j$ and hence $(e_0 , e_1 )_\# \pi = ( e_0 , e_1 )_\#
  \tilde{\pi}$.  In particular $\tilde{\pi}$ is a dynamic optimal
  coupling of $\mu_0$ and $\mu_1$. By weak lower semi-continuity of
  the entropy we can pass to the limit as $n\to\infty$ in the left
  hand side of \eqref{eq:equiv-as3}. Invoking furthermore the
  convexity of $G_t$ given by Lemma~\ref{lem:convexG} and Jensen's
  inequality we see that
\begin{align}\label{eq:equiv-as4}
        &-\frac{\a^{-1}}{N}\int\limits_{A}\log\rho_t(\gamma_t)\dd \tilde{\pi}\\\nonumber
&\geq~G_t\Big(-\frac{\a^{-1}}{N}\int\limits_{A}\log\rho_0(\gamma_0)\dd \tilde{\pi},-\frac{\a^{-1}}{N}\int\limits_{A}\log\rho_1(\gamma_1)\dd \tilde{\pi},\a^{-1}\frac{K}{N}\int\limits_{A} d^2(\gamma_0,\gamma_1)\dd \tilde{\pi}\Big)\;,
  \end{align}
  for any set $A$ which is a union of a finite number of the sets
  $A_{i,j}$ and $\a=\tilde{\pi}(A)$. This implies the
  $\tilde{\pi}$-a.s. inequality \eqref{eq:cdkn-as}.
\end{proof}

\begin{corollary}\label{cor:strong-cde-cds}
  For a metric measure space $(X,d,m)$ the following assertions are equivalent:
  \begin{itemize}
  \item[(i)] $(X,d,m)$ is a strong $\cdskn$ space,
  \item[(ii)] For each pair $\mu_0,\mu_1\in\cP_\infty(X,d,m)$, and
    each dynamic optimal coupling $\pi$ of it \eqref{eq:cdkn-as}
    holds,
  \item[(iii)] $(X,d,m)$ is a strong $\ecdkn$ space.
  \end{itemize}
\end{corollary}

\begin{proof}
  Note that both (i) and (iii) imply that $(X,d,m)$ satisfies the
  strong $\cd(K,\infty)$ condition. \cite[Thm.~1.1]{RS12} gives that
  every strong $\cd(K,\infty)$ space is essentially non-branching. In
  addition, \cite[Cor.~1.4]{RS12} also states that on strong $\cd
  (K,\infty)$ spaces the dynamic optimal coupling of $\mu_0$ and
  $\mu_1$ is unique for each $\mu_0 , \mu_1 \in \cP_2 (X,d,m)$. Hence
  the assertion follows from the same arguments as
  Theorem~\ref{thm:equiv_as}. Indeed, the dynamic optimal coupling
  $\tilde{\pi}$ obtained in the proof of Theorem~\ref{thm:equiv_as}
  (iii)$\Rightarrow$(ii) coincides with $\pi$. Note that the
  essentially non-branching assumption is not used in the implications
  (ii)$\Rightarrow$(i),(iii).
\end{proof}

We conclude this section with a globalization property of the strong
entropic curvature-dimension condition. We say that a metric measure
space $(X,d,m)$ satisfies the \emph{local} entropic
curvature-dimension condition $\cd^e_{\text{loc}}(K,N)$ if and only if
every point $x\in\supp m$ has a neighborhood $M$ such that for each
pair $\mu_0,\mu_1\in\cP^*_2(X,d,m)$ supported in $M$ there exists a
geodesic $(\mu_t)_{t\in[0,1]}$ in $\cP^*_2(X,d,m)$ satisfying
\eqref{eq:ecdkn}. Similarly, we say that $(X,d,m)$ is a \emph{strong}
$\cd^e_{\text{loc}}(K,N)$ space if in addition \eqref{eq:ecdkn} holds
along \emph{every} constant speed geodesic $(\mu_t)_{t\in[0,1]}$ in
$\cP^*_2(X,d,m)$ with $\mu_0,\mu_1$ supported in $M$. Note that
$(X,d,m)$ is essentially non-branching if it is $\cd^e_{\text{loc}}
(K,N)$ space. Indeed, we first localize the problem in the argument
in \cite{RS12} and hence the local condition is sufficient.

\begin{theorem}[Local-global]\label{thm:cde-locglob}
  Let $(X,d,m)$ be a geodesic metric measure space. Then it satisfies
  the strong $\ecdkn$ condition if and only if it satisfies the strong
  $\cd^e_{\text{loc}}(K,N)$ condition.
\end{theorem}

\begin{proof}
  The only if part is obvious. For the if part, assume that $(X,d,m)$
  is a strong $\cd^e_{\text{loc}}(K,N)$ space. First note that this
  implies that $X$ is locally compact. Indeed, this can be seen by
  estimating the volume growth of balls in a small neighborhood around
  any point similarly as in Proposition~\ref{prop:BGI}. $(X,d)$ being
  a length space, local compactness implies that bounded closed sets in $X$
  are compact, see \cite[Prop.~2.5.22]{BBI01}.

  Now we first verify the $\cd^e(K,N)$ inequality \eqref{eq:ecdkn} for
  a geodesic $(\mu_t)_{t\in[0,1]}$ in $\cP_2^*(X,d,m)$ where the
  measures $\mu_t$ are jointly supported in a compact set $K$. By
  compactness and the strong $\cd^e_{\text{loc}}(K,N)$ condition we
  can find $\epsilon>0$ and a disjoint partition $(Y_i)_i$ of $K$ such
  that the $\eps$-neighborhoods $U_i$ of $Y_i$ have the following
  property: any geodesic $(\mu_t)_{t\in[0,1]}$ in $\cP_2^*(X,d,m)$
  with $\mu_0,\mu_1$ supported in $U_i$ satisfies
  \eqref{eq:ecdkn}. Write $\mu_t=(e_t)_\#\pi$, where
  $\pi\in\cP(\geo(X))$ is the associated dynamic optimal
  coupling. Then there exists $L>0$ such $d(\gamma_0,\gamma_1)\leq L$
  for all $\gamma$ in the support of $\pi$. We claim that for any
  $0\leq r\leq t\leq s\leq 1$ with $\abs{s-r}<\eps/L$:
  \begin{align}\label{eq:ecdkn-loc-time1}
    U_{N}(\mu_t)~\geq~\sigkn{\frac{s-t}{s-r}}{W_2(\mu_r,\mu_s)} U_{N}(\mu_r) +
    \sigkn{\frac{t-r}{s-r}}{W_2(\mu_r,\mu_s)} U_{N}(\mu_s)\;,
  \end{align}
  which suffices to show \eqref{eq:ecdkn} by virtue of
  Lemma~\ref{lem:knconvexint2}. Indeed, let us define the sets
  $A_{i}=\{\gamma\in\geo(X)\ :\ \gamma_{t}\in Y_i\}$ and define the
  measures
  \begin{align*}
    \pi_{i}~:=~\alpha_{i}^{-1}\pi|_{A_{i}}\;,
  \end{align*}
  provided that $\alpha_{i}:=\pi(A_{i})>0$. Then for
  $\pi_i$-a.e. geodesic $\gamma$ and $\tau\in[r,s]$ one has $\gamma_\tau\in
  U_i$. Setting $\mu^i_\tau=(e_\tau)_\#\pi_i$ we infer that the geodesic
  $(\mu^i_\tau)_{\tau\in[r,s]}$ is supported in $U_i$. From the construction
  of $U_i$ we obtain for $\tau\in[r,s]$:
  \begin{align}\label{eq:ecdkn-loc-time2}
    U_{N}(\mu^i_\tau)~\geq~\sigkn{\frac{s-\tau}{s-r}}{W_2(\mu^i_r,\mu^i_s)} U_{N}(\mu^i_r) +
    \sigkn{\frac{\tau-r}{s-r}}{W_2(\mu^i_r,\mu^i_s)} U_{N}(\mu^i_s)\;.
  \end{align}
  Note that $\mu_\tau=\sum_i\a_i\mu^i_\tau$.
  Hence we have that (see e.g. \cite[Rem.~I.4.2]{S06})
  \begin{align}\label{eq:ent-subadd}
    \ent(\mu_\tau)\geq\sum_i\a_i\ent(\mu^i_\tau) +
    \sum_i\a_i\log\a_i\;.
  \end{align}
  For $\tau=t$ we have equality in \eqref{eq:ent-subadd} since the
  family $(\mu_t^i)_i$ is mutually singular by construction. Taking
  logarithms in \eqref{eq:ecdkn-loc-time2} and summing over $i$ we
  obtain
 \begin{eqnarray*}
\lefteqn{ -\frac1N \ent(\mu_t)\ 
= \ -\frac1N \sum_i\a_i\big[\ent(\mu_t^i)+\log\a_i\Big]}\\
                        &\geq&\sum_i\a_i G_{\frac{t-r}{s-r}}\left(-\frac1N \Big[\ent(\mu_r^i)+\log\a_i\Big],-\frac1N \Big[\ent(\mu_s^i)+\log\a_i\Big],\frac{K}{N}W_2^2(\mu^i_r,\mu^i_s)\right)\\
                         &\geq&G_{\frac{t-r}{s-r}}\left(-\frac1N \sum_i\a_i\Big[\ent(\mu_r^i)+\log\a_i\Big],-\frac1N \sum_i\a_i\Big[\ent(\mu_s^i)+\log\a_i\Big],\frac{K}{N}\sum_i\a_iW_2^2(\mu^i_r,\mu^i_s)\right)\\
                         &\geq&G_{\frac{t-r}{s-r}}\left(-\frac1N \ent(\mu_r),-\frac1N \ent(\mu_s),\frac{K}{N}W_2^2(\mu_r,\mu_s)\right)\;,
  \end{eqnarray*}
  where we have used \eqref{eq:ent-subadd} as well as the convexity of
  $G_{\frac{t-r}{s-r}} (x,y, \kappa)$ given by Lemma~\ref{lem:convexG}
  and its monotonicity in $x,y$. Taking the exponential yields
  \eqref{eq:ecdkn-loc-time1}.

  Finally, we establish the $\ecdkn$ inequality \eqref{eq:ecdkn} for
  an arbitrary, not necessarily compactly supported geodesic
  $(\mu_t)_{t\in[0,1]}$ in $\cP_2^*(X,d,m)$. Partition $X$ in a
  disjoint collection of precompact sets $K_i$ and let $\pi_{i,j}$ be
  dynamic optimal couplings obtained by conditioning the coupling
  $\pi$ associated to $(\mu_t)_t$ to have starting point in $K_i$ and
  endpoint in $K_j$. By the previous argument any compactly supported
  geodesic satisfies \eqref{eq:ecdkn}. Since $\cd^e_{\text{loc}}(K,N)$
  implies that $(X,d,m)$ is essentially non-branching, the measures
  $(e_t)_\#\pi_{i,j}$ are mutually singular using
  Lemma~\ref{lem:non-branch}. Thus arguing as before the inequality
  \eqref{eq:ecdkn} for $(\mu_t)_t$ can be obtained by summing the
  corresponding inequalities valid along the geodesics
  $(\mu_t^{i,j})_t$ associated to $\pi_{i,j}$.
\end{proof}

\subsection{Calculus  and heat flow on metric measure spaces}
\label{sec:recap}

Here we recapitulate briefly some of the results obtained by Ambrosio,
Gigli and Savar\'e in a series of recent works, see
\cite{AGS11a,AGS11b,AGS12,G12}. In particular, we introduce notation and
concepts that we use in the sequel about the powerful machinery of
calculus on metric measure spaces developed by these authors. We refer
to \cite{AGS11a,AGS11b} for more details on the definitions and
results.

Let $(X,d,m)$ be a metric measure space. The basic object of study,
introduced in \cite{AGS11a} is the Cheeger energy. For a measurable
function $f:X\to\R$ it can be defined by
\begin{align*}
  \ch(f)=\frac12\int\wug{f}^2\dd m\;,
\end{align*}
where $\wug{f}:X\to[0,\infty]$ denotes the so called minimal weak
upper gradient of $f$. An important approximation result
\cite[Thm.~6.2]{AGS11a} states that for $f\in L^2(X,m)$ the Cheeger
energy can also be obtained by a relaxation procedure:
\begin{align*}
  \ch(f)=\inf\left\{\liminf\limits_{n\to\infty}\frac12\int\abs{\nabla f_n}^2\dd m\right\}\;,%\;:\; f_n\in \Lip(X),\ \norm{f_n- f}_{L^2(m)}\to0\right\}\;.
\end{align*}
where the infimum is taken over all sequences of Lipschitz functions
$(f_n)$ converging to $f$ in $L^2(X,m)$ and where $\abs{\nabla f_n}$
denotes the local Lipschitz constant. In particular, Lipschitz
functions are dense in in the domain of $\ch$ in $L^2(X,m)$ denoted by
$D(\ch)=W^{1,2}(X,d,m)$ in the following sense: For each $f \in D
(\ch)$ there exist a sequence $( f_n )_{n \in \N}$ of Lipschitz
functions such that $f_n \to f$ in $L^2$ and $| \nabla f_n | \to
\wug{f}$ in $L^2$ \cite[Lem.~4.3(c)]{AGS11a}. For a Lipschitz function
$f$ the slope, or local Lipschitz constant, is an upper gradient. Thus
\begin{align}\label{eq:lip-wug}
  \wug{f}~\leq~\abs{\nabla f}\qquad \text{a.e.}
\end{align}

It turns out that $\ch$ is a convex and lower
semi-continuous functional on $L^2(X,m)$. It allows to define the
Laplacian $-\Delta f\in L^2(X,m)$ of a function $f\in W^{1,2}(X,d,m)$
as the element of minimal $L^2$-norm in the subdifferential
$\partial^-\ch(f)$ provided the latter is non-empty. In this
generality, $\ch$ is not necessarily a quadratic form and
consequently $\Delta$ need not be a linear operator.

The classical theory of gradient flows of convex functionals in
Hilbert-spaces allows to study the gradient flow of $\ch$ in
$L^2(X,m)$: For any $f\in L^2(X,m)$ there exists a unique continuous
curve $(f_t)_{t\in[0,\infty)}$ in $L^2(X,m)$, locally absolutely
continuous in $(0,\infty)$ with $f_0=f$ such that $\ddt
f_t~\in~\partial^-\ch(f_t)$ for a.e. $t>0$. In fact, we have $f_t\in
D(\Delta)$ and
\begin{align*}
  \ddtr f_t~=~\Delta f_t
\end{align*}
for all $t>0$. This gives rise to a semigroup $(\bH_t)_{t\geq0}$ on
$L^2(X,m)$ defined by $\bH_tf=f_t$, where $f_t$ is the unique
$L^2$-gradient flow of $\ch$.

On the other hand, one can study the metric gradient flow of the
relative entropy $\ent$ in $\cP_2(X,d)$. Under the assumption that
$(X,d,m)$ satisfies $\cd(K,\infty)$ it has been proven in \cite{Gi10}
and more generally in \cite[Thm.~9.3(ii)]{AGS11a} that for any $\mu\in
D(\ent)$ there exist a unique gradient flow of $\ent$ starting from
$\mu$ in the sense of Definition~\ref{def:metric-gf}. This gives rise
to a semigroup $(\cH_t)_{t\geq0}$ on $\cP_2(X,d)$ defined by
$\cH_t\mu=\mu_t$ where $\mu_t$ is the unique gradient flow of $\ent$
starting from $\mu$.

One of the main result of \cite{AGS11a} is the identification of the
two gradient flows, which allows to consistently define the heat flow
on $\cd(K,\infty)$ spaces.

\begin{theorem}[{\cite[Thm.~9.3]{AGS11a}}]\label{thm:gf-identification}
  Let $(X,d,m)$ be a $\cd(K,\infty)$ space and let $f\in L^2(X,d,m)$
  such that $\mu=f m\in \cP_2(X,d)$. % Let $(f_t)$ be the gradient flow
  % of $\ch$ in $L^2(X,m)$ starting in $f$ and let $(\mu_t)$ be the
  % gradient flow of $U$ starting in $\mu$.
  Then we have
  \begin{align*}
    \cH_t\mu=(\bH_tf) m \quad\forall t\geq0\;.
  \end{align*}
\end{theorem}

A byproduct of this result is a representation of the slope of the
entropy.
\begin{align}\label{eq:slope-FI}
  \abs{\nabla^-\ent}(\rho m)~=~4\int\wug{\sqrt{\rho}}^2\dd m
\end{align}
for all probability densities $\rho$ with $\sqrt{\rho}\in
D(\ch)$. Note that the minimal weak upper gradient satisfies a chain
rule, \cite[Prop.~5.16]{AGS11a}: for $\phi:I\to\R$ non-decreasing and
locally Lipschitz we have
\begin{align}\label{eq:chainrule-wug}
  \wug{\phi(f)}~=~\phi'(f)\wug{f}\;.
\end{align}
A basic property of the heat flow is the maximum principle, see
\cite[Thm.~4.16]{AGS11a}: If $f\in L^2(X,m)$ satisfies $f\leq C$
$m$-a.e. then also $\bH_tf\leq C$ $m$-a.e. for all $t\geq0$.

If $\ch$ is assumed to be a quadratic form, and without any curvature
assumption, the notion of weak upper gradient gives rise to a powerful
calculus, in which not only the norm of the gradient, but also scalar
products between gradients are defined. For details we refer to
\cite[Sec. 4.3]{AGS11b} and \cite[Sec. 4.3]{G12}, where this calculus
has been developed in larger generality. We note briefly that given
$f,g\in D(\ch)$, the limit
\begin{align}\label{eq:scp}
  \langle\nabla f,\nabla g\rangle~:=~\lim\limits_{\eps\searrow0}\frac{1}{2\eps}\left(\wug{(f+\eps g)}^2-\wug{f}^2\right)
\end{align}
can be shown to exists in $L^1(X,m)$. Moreover, the map
$D(\ch)^2\ni(f,g)\mapsto\langle\nabla f,\nabla g\rangle \in L^1(X,m)$ is
bilinear, symmetric and satisfies
\begin{align*}
  \abs{\langle\nabla f,\nabla g\rangle}~\leq~\wug{f}\wug{g}\;.
\end{align*}
For all $f,g,h\in D(\ch)\cap L^\infty(X,m)$ we have the Leibniz rule:
\begin{align}\label{eq:Leibniz}
  \int \langle\nabla f,\nabla(g h)\rangle\dd m~=~\int h \langle\nabla f,\nabla g\rangle\dd
  m +\int g \langle\nabla f,\nabla h\rangle\dd m\;.
\end{align}

A quadratic Cheeger energy gives rise to a strongly local Dirichlet
form $(\cE,D(\cE))$ on $L^2(X,m)$ by setting $\cE(f,f)=\ch(f)$ and
$D(\cE)=W^{1,2}(X,d,m)$. In particular, $W^{1,2}(X,d,m)$ is a Hilbert
space and $L^2$-Lipschitz functions are dense in the usual sense
\cite[Prop.~4.10]{AGS11b}. In this case $\bH_t$ is a semigroup of
self-adjoined linear operators on $L^2(X,m)$ with the Laplacian
$\Delta$ as its generator. The previous result implies that for
$f,g\in W^{1,2}(X,d,m)$
\begin{align*}
  \cE(f,g)~=~\int \langle\nabla f,\nabla g\rangle\dd m\;,
\end{align*}
i.e. the energy measure of $\cE$ has a density given by
\eqref{eq:scp}. Moreover, for $f\in W^{1,2}$ and $g\in D(\Delta)$ we
have the integration by parts formula
\begin{align}\label{eq:int-by-parts}
  \int \langle\nabla f,\nabla g\rangle\dd m~=~-\int f\Delta g\dd m\;.
\end{align}

\subsection{The Riemannian curvature-dimension condition}
\label{sec:riem-cdkn}

In this section we introduce the notion of Riemannian
curvature-dimension bounds. This notion can be seen as a
generalization of the Riemannian Ricci curvature bounds for metric
measure spaces introduced in \cite{AGS11b} for mms with finite
reference measure and later generalized in \cite{AGMR12} to
$\sigma$-finite reference measures. We will rely on the powerful
machinery of calculus on metric measure spaces already developed by
Ambrosio, Gigli, Savar\'e and co-authors in a series of recent
works. Following their nomenclature, we make the following

\begin{definition}\label{def:riemcdkn}
  We say that a metric measure space $(X,d,m)$ is
  \emph{infinitesimally Hilbertian} if the associated Cheeger energy
  is quadratic. Moreover, we say that it satisfies the
  \emph{Riemannian curvature-dimension condition} $\rcdkn$ if it
  satisfies any of the equivalent properties of
  Theorem~\ref{thm:RCDKN-equiv} below.
\end{definition}

\begin{theorem}\label{thm:RCDKN-equiv}
  Let $(X,d,m)$ be a metric measure space with $\supp m = X$. The
  following properties are equivalent:
  \begin{itemize}
  \item[(i)] $(X,d,m)$ is infinitesimally Hilbertian and satisfies the $\cdskn$ condition.
  \item[(ii)] $(X,d,m)$ is infinitesimally Hilbertian and satisfies the $\ecdkn$ condition.
  \item[(iii)] $(X,d,m)$ is a length space satisfying the exponential
    integrability condition \eqref{eq:exp-int} and any
    $\mu\in\cP_2(X,d)$ is the starting point of an $\evikn$ gradient
    flow of $\ent$.
  \end{itemize}
\end{theorem}

\begin{remark}\label{rem:strongcde}
  Note that according to Theorem~\ref{thm:eviconvex}, (iii) even
  implies that $(X,d,m)$ is a strong $\ecdkn$ space and a geodesic
  space.
\end{remark}

\begin{remark}\label{rem:additive-linear}
  Since both $\cdskn$ and $\ecdkn$ imply the $\cd(K,\infty)$
  condition, \cite[Thm.~5.1]{AGS11b}, resp. \cite[Thm.~6.1]{AGMR12}
  show that the requirement that the Cheeger energy $\ch$ is quadratic
  can equivalently be replaced in (i) and (ii) by additivity of the
  semigroup $\cH_t$, in the sense that
  $\cH_t\big(\lambda\mu+(1-\lambda)\nu\big)=\lambda\cH_t\mu
  +(1-\lambda)\cH_t\nu$ for any $\mu,\nu\in\cP_2(X,d)$ and
  $\lambda\in[0,1]$.
\end{remark}

\begin{proof}
  (i)$\Leftrightarrow$(ii): Both $\cdskn$ and $\ecdkn$ imply the
  $\cd(K,\infty)$ condition. Thus \cite[Thm.~6.1]{AGMR12} yields that
  under either (i) or (ii) the $\evi_K$ gradient flow of $\ent$ exists
  for every starting point. This implies that $(X,d,m)$ is a strong
  $\cd(K,\infty)$ space and hence essentially non-branching by
  \cite[Thm.~1.1]{RS12}. In this setting, Theorem~\ref{thm:equiv_as}
  yields equivalence of $\cdskn$ and $\ecdkn$.

  (ii)$\Rightarrow$(iii): By Remark~\ref{rem:exp-int}, $(X,d)$ is a
  geodesic space and satisfies \eqref{eq:exp-int}. Taking
  Theorem~\ref{thm:contraction} into account it is sufficient to show
  that $\cH_t(\mu)$ is an $\evikn$-gradient flow of $\ent$ for every
  $\mu\in\cP_2(X,d,m)$ of the form $\mu=f m$ with $f$ bounded and
  $\ch(\sqrt{f})<\infty$. Set $\mu_t:=\cH_t(\mu)=f_t m$ and note that
  $f_t$ is still bounded with $\ch(\sqrt{f_t})<\infty$ for all
  $t>0$. By Proposition~\ref{prop:evi-equiv} it is sufficient to take
  reference measures in \eqref{eq:knflow} of the form $\sigma=g m$
  where $g$ is bounded and has bounded support. Taking into account
  \eqref{eq:evialt1} we have to show that for a.e. $t>0$:
  \begin{align}\label{eq:equiv1}
    \frac{U_N(\sigma)}{U_N(\mu_t)}~\leq~\ckn{W_2(\mu_t,\sigma)} - \frac{\skn{W_2(\mu_t,\sigma)}}{N\cdot W_2(\mu_t,\sigma)}\ddt\frac12
    W_2(\mu_t,\sigma)^2 \;.
  \end{align}
  This will follow from essentially the same arguments as in the proof
  of \cite[Thm.~6.1]{AGMR12}. Let us briefly sketch these arguments,
  indicating the modifications that are necessary.

  First, \cite[Thm.~6.3]{AGMR12} yields that for a.e. $t>0$:
  \begin{align}\label{eq:equiv2}
    \ddt\frac12 W_2(\mu_t,\sigma)^2~=~-\cE_{\mu_t}(\phi_t,\log f_t)\;,
  \end{align}
  where $\phi_t$ is a suitable Kantorovich potential for the optimal
  transport from $\mu_t$ to $\sigma$ and $\cE_{\mu_t}(\cdot,\cdot)$ is
  the bilinear form associated to the weighted Cheeger energy
  $\ch_{\mu_t}(f)=\frac12\int\abs{\nabla f}_{w,\mu_t}\dd\mu_t$ (see
  \cite[Sec. 3]{AGMR12}). We claim that also
  \begin{align}\label{eq:equiv3}
    \cE_{\mu_t}(\phi_t,\log f_t)~\geq~\frac{N\cdot W_2(\mu_t,\sigma)}{\skn{W_2(\mu_t,\sigma)}}\Big[-\ckn{W_2(\mu_t,\sigma)} + \frac{U_N(\sigma)}{U_N(\mu_t)}\Big]\;.
  \end{align}
  Combining then \eqref{eq:equiv2} and \eqref{eq:equiv3} yields the
  desired inequality \eqref{eq:equiv1}.

  To prove \eqref{eq:equiv3} one argues similar as in
  \cite[Thm.~6.5]{AGMR12}. First $f_t$ is approximated by suitable
  truncated probability densities $f_t^\delta$. Then, by successively
  minimizing the entropy of midpoints, a particularly nice geodesic
  $(\Gamma_s^{\delta,t})_{s\in[0,1]}$ connecting
  $\mu_t^\delta=f_t^\delta m$ to $\sigma$ is constructed which
  satisfies the $\cd(K,\infty)$ condition and has density bounds. From
  the construction it is immediate that in our setting this geodesic
  also satisfies the $\ecdkn$ condition. Thus on one hand, we have by
  Lemma~\ref{lem:convex-below-tangent} below the inequality
  \begin{align}\label{eq:equiv4}
    \liminf\limits_{s\searrow0}\frac{U_N(\Gamma^{\delta,t}_s)-U_N(\mu^\delta_t)}{s}~\geq~\frac{W_2(\mu^\delta_t,\sigma)}{\skn{W_2(\mu^\delta_t,\sigma)}}\Big[-U_N(\mu^\delta_t)\cdot\ckn{W_2(\mu^\delta_t,\sigma)}
    + U_N(\sigma)\Big]\;.
  \end{align}
  On the other hand, \cite[Prop.~6.6]{AGMR12} yields that
  \begin{align}\label{eq:equiv5}
    -\cE_{\mu^\delta_t}(\phi^\delta_t,\log f^\delta_t)~\leq~\liminf\limits_{s\searrow0}\frac{\ent(\Gamma^{\delta,t}_s)-\ent(\mu^\delta_t)}{s}\;,
  \end{align}
  where $\phi^\delta_t$ is a Kantorovich potential relative to
  $\mu^\delta_t$ and $\sigma$. By $K$-convexity of $\ent$ along the
  geodesic $\Gamma^{\delta,t}$ we have
  \begin{align*}
    \limsup\limits_{s\searrow0}\frac{\ent(\Gamma^{\delta,t}_s)-\ent(\mu^\delta_t)}{s}~\leq~\ent(\sigma)
    - \ent(\mu^\delta_t) - \frac{K}{2} W_2(\mu^\delta_t,\sigma)^2
  \end{align*}
  and thus
  $\big(\ent(\Gamma^{\delta,t}_s)-\ent(\mu^\delta_t)\big)^2=o(s)$ as
  $s\to0$. Now \eqref{eq:equiv4} and \eqref{eq:equiv5} together with a
  Taylor expansion of $x\mapsto e^{-x/N}$ yield
  \begin{align}\label{eq:equiv6}
    \cE_{\mu^\delta_t}(\phi^\delta_t,\log f^\delta_t)~\geq~\frac{N\cdot W_2(\mu^\delta_t,\sigma)}{\skn{W_2(\mu^\delta_t,\sigma)}}\Big[-\ckn{W_2(\mu^\delta_t,\sigma)} + \frac{U_N(\sigma)}{U_N(\mu^\delta_t)}\Big]\;.
  \end{align}
  Finally \eqref{eq:equiv3} is obtained by lifting the truncation and
  passing to the limit $\delta\to0$ in \eqref{eq:equiv6}. Passage to
  the limit in the RHS is obvious, for the LHS a delicate argument is
  needed which is given in the proof of \cite[Thm.~6.5]{AGMR12}.

  % Note that in the case treated in \cite{AGS11b} where $m$ is a
  % probility measure a simpler argument is sufficient avoiding the use
  % of weighted Cheeger energies. Here \cite[Thms.~4.1,4.8]{AGS11b}
  % yield
  % \begin{align*}
  %   \ddt\frac12 W_2(\mu_t,\sigma)^2~&\leq~\frac{\ch(f_t-\eps\phi_t)-\ch(f_t)}{\eps}\;,\\
  % \liminf\limits_{s\searrow0}\frac{\ent(\Gamma^t_s)-\ent(\mu_t)}{s}~&\geq~\frac{\ch(\phi_t)-\ch(\phi_t+\eps f_t)}{\eps}\;,
  % \end{align*}
  % for any $\eps>0$ and any $\ecdkn$ geodesic $\Gamma^t$ connecting
  % $\mu_t$ to $\sigma$. Plugging this into Lemma
  % \ref{lem:convex-below-tangent} and using that the RHSs coincide up
  % to $O(\eps)$ as $\eps\to0$ since $\ch$ is quadratic then immediately
  % yields \eqref{eq:equiv1}.

  (iii)$\Rightarrow$(ii). Since by Lemma~\ref{lem:consistentKN} an
  $\evikn$ flow is in particular an $\evi_K$ flow,
  \cite[Thm.~5.1]{AGS11b} or \cite[Thm.~6.1]{AGMR12} already gives
  that $(X,d,m)$ is infinitesimally Hilbertian. Let us now show that
  $(X,d,m)$ is a strong $\ecdkn$ space. The same argument as in the
  proof of \cite[Lem.~5.2]{AGS11b} yields for any pair $\mu_0,\mu_1\in
  D(\ent)\subset\cP_2(X,d,m)$ the existence of a geodesic
  $\Gamma:[0,1]\to D(\ent)$ connecting $\mu_0$ to $\mu_1$. Hence
  $D(\ent)$ is a geodesic space and Theorem~\ref{thm:eviconvex} shows
  that \eqref{eq:ecdkn} holds along any geodesic in $D(\ent)$.
\end{proof}

\begin{lemma}\label{lem:convex-below-tangent}
  Let $(X,d,m)$ satisfy the $\ecdkn$ condition and let
  $\mu_0,\mu_1\in\cP_2(X,d,m)$. Then there exists a geodesic
  $(\mu_t)_{t\in[0,1]}$ in $\cP_2(X,d,m)$ connecting $\mu_0$ and $\mu_1$ such
  that, with $\theta=W_2(\mu_0,\mu_1)$,
  \begin{align}\label{eq:convex-below-tangent}
    U_N(\mu_1)~\leq~\ckn{\theta}\cdot U_N(\mu_0) + \frac{\skn{\theta}}{\theta}\cdot\liminf\limits_{t\searrow0}\frac{U_N(\mu_t)-U_N(\mu_0)}{t} \;.
  \end{align}
\end{lemma}

\begin{proof}
  Let $(\mu_t)_{t\in[0,1]}$ be the geodesic connecting $\mu_0$ and $\mu_1$ given by
  the $\ecdkn$ condition. We immediately obtain that for
  every $t\in[0,1]$:
  \begin{align*}
    U_N(\mu_t)-U_N(\mu_0)~\geq~\left[\sigkn{1-t}{\theta}-1\right]\cdot U_N(\mu_0) + \sigkn{t}{\theta}\cdot U_N(\mu_1)\;.
  \end{align*}
  Dividing by $t$ on both sides and passing to the limit $t\searrow0$
  the assertion follows from the fact that
  \begin{align*}
    \ddt\sigkn{t}{\theta} = +\frac{\theta\cdot\ckn{t\theta}}{\skn{\theta}}\;,\quad \sigkn{0}{\theta}=0\;,\quad \sigkn{1}{\theta}=1\;.
  \end{align*}
\end{proof}

\begin{proposition}[Weighted spaces]\label{prop:weighted-rcdkn}
  Let $(X,d,m)$ be a $\rcdkn$ space and let $V:X\to\R$ be continuous,
  bounded below and strongly $(K',N')$-convex function in the sense of
  Definition~\ref{def:knconvex} with $\int\exp(-V)\dd m< \infty$. Then
  $(X,d,\e^{-V}m)$ is a $\rcd^*(K+K',N+N')$ space.
\end{proposition}

\begin{proof}
  By Proposition~\ref{prop:weighted-cde}, $(X,d,\e^{-V}m)$ is a
  $\cd^e(K+K',N+N')$ space. Invariance of the weak upper gradient
  under multiplicative changes of the reference measure by
  \cite[Lem.~4.11]{AGS11a} together with the Leibniz rule
  \eqref{eq:Leibniz} give that the Cheeger energy associated to
  $\e^{-V}m$ is again quadratic. See also
  \cite[Prop.~6.19]{AGS11b}. Thus the assertion follows from
  Theorem~\ref{thm:RCDKN-equiv} (ii).
\end{proof}

The Riemannian curvature-dimension condition has a number of natural
properties that we collect here. The first one is the stability under
convergence of metric measure spaces in the transportation distance
$\D$. We refer to \cite[Sec.~I.3]{S06} for the definition and properties
of the transportation distance.

\begin{theorem}[Stability]\label{thm:rcd-stable}
  Let $( (X_n,d_n,m_n) )_{n \in \N}$ be a sequence of $\rcdkn$ spaces
  with $m_n \in \cP_2 (X_n,d_n)$. If
  $\D\big((X_n,d_n,m_n),(X,d,m)\big)\to 0$ for some metric measure
  space $(X,d,m)$ then $(X,d,m)$ is also a $\rcdkn$ space.
\end{theorem}

Note that this in particular implies stability of the $\rcdkn$-condition under {\it measured Gromov-Hausdorff convergence} (mGH-convergence for short). Indeed, for compact mms -- and only for such spaces the concept of mGH-convergence is well-established -- mGH-convergence implies $\D$-convergence
\cite[Lemma 3.18]{S06}.

\begin{proof}
  We follow essentially the arguments of Ambrosio, Gigli and Savar\'e
  in \cite[Thm.~6.10]{AGS11b} where stability of the $\rcd(K,\infty)$
  condition has been established.

  We show stability of characterization (iii) in Theorem~\ref{thm:RCDKN-equiv}.
  By Proposition~\ref{prop:evi-equiv} and
  Corollary~\ref{cor:contraction} it is sufficient to show that for
  any $\mu=f m\in\cP_2(X,d,m)$ with $f\in L^\infty(X,m)$ there
  exists a continuous curve $(\mu_t)_{t\in[0,\infty)}$ in $\cP_2(X,d)$,
  locally absolutely continuous in $(0,\infty)$ and starting in $\mu$
  such that for any $\nu=\sigma m\in\cP_2(X,d)$ with $\sigma\in
  L^\infty(X,d,m)$ and any $s\leq t$:
  \begin{align}\label{eq:stable1}
    e_K(t-s)\frac{N}{2}\left(1-\frac{U_N(\nu)}{U_N(\mu_{t})}\right)~\geq~\e^{K(t-s)}&\skn{\frac12
        W_2(\mu_{t},\nu)}^2%\\\nonumber
 -&\skn{\frac12 W_2(\mu_{s},\nu)}^2\;.
  \end{align}
  Choose optimal couplings $(\hat d_n,q_n)$ of $(X_n,d_n,m_n)$ and
  $(X,d,m)$. Given $\mu=f m\in\cP_2(X,d,m)$ we set
  \begin{align*}
   Q_n\mu(\dd x)~=~\int f(y)q_n(\dd x,\dd y)~\in~\cP_2(X_n,d_n,m_n)\;.
  \end{align*}
  Similarly we obtain an operator
  $Q_n':\cP_2(X_n,d_n,m_n)\to\cP_2(X,d,m)$, see \cite[Lem.~I.4.19]{S06} and also \cite[Prop.~2.2,2.3]{AGS11b}.

  Now set $\mu^n=Q_n\mu$. By assumption there exists a curve
  $(\mu^n_t)_{t\in[0,\infty)}$ in $\cP_2(X_n,d_n)$ starting from
  $\mu^n$ such that for all $s\leq t$:
  \begin{align}\label{eq:stable2}
    e_K(t-s)\frac{N}{2}\left(1-\frac{U^n_N(\nu^n)}{U^n_N(\mu^n_{t})}\right)~\geq~\e^{K(t-s)}&\skn{\frac12
        W_2(\mu^n_{t},\nu^n)}^2%\\\nonumber
 -&\skn{\frac12 W_2(\mu^n_{s},\nu^n)}^2\;,
  \end{align}
  where $\nu^n=Q_n\nu$ and $U^n_N$ corresponds to the relative entropy
  functional in $(X_n,d_n,m_n)$. By the maximum principle we have
  $\mu^n_t\leq Cm_n$ with $C = \| \rho \|_{L^\infty ( X, m )}$.
  For each $t\geq0$ set
  $\tilde\mu^n_t:=Q'_n\mu^n_t\in\cP_2(X,d)$. We claim that, after
  extraction of a subsequence, we have that $\tilde\mu^n_t\to\mu_t$ in
  $\cP_2(X,d)$ as $n\to\infty$ for a curve $(\mu_t)$ in $\cP_2(X,d)$.

  Indeed, note that $\tilde\mu^n_t\leq C m$ for all $n$ and $t$. From
  the Energy Dissipation Equality \eqref{eq:EDE} we conclude that
  \begin{align*}
    \int_s^t\abs{\dot\mu^n_r}^2\dd r~\leq~\ent(\mu^n|m^n)~\leq~C\log C
  \end{align*}
  and hence the curves $(\mu^n_t)$ are equi-absolutely continuous.
  Since $m\in\cP_2(X,d)$, the set of measures
  $\{\mu\in\cP_2(X,d,m)):\mu\leq Cm\}$ is relatively compact w.r.t
  $W_2$-convergence. Hence, by a diagonal argument, we conclude that
  up to extraction of a subsequence $\tilde\mu^n_t\to\mu_t$ for all
  $t\in\Q_+$ and some $\mu_t\in\cP_2(X,d)$. Using the equi-absolute
  continuity of the curves $(\mu^n_t)$ and the equi-continuity of the
  map $Q_n'$ we obtain convergence for all times
  $t\in[0,\infty )$ for
  the same subsequence and a curve $(\mu_t)$ in $\cP_2(X,d)$ which is
  again absolutely continuous.

  Finally, we observe that since the operators $Q_n,Q'_n$ do not
  increase the entropy we have $U_N^n(\nu^n)\geq U_N(\nu)$ and by
  lower semi-continuity of the entropy also
  $\ent(\mu_t)\leq\liminf_n\ent(\tilde\mu^n_t)\leq\liminf_n\ent(\mu^n_t|m^n)$. Moreover,
  we have $W_2(\mu^n_t,\nu^n)\to W_2(\mu_t,\nu)$. This allows to pass
  to the limit in \eqref{eq:stable2} to obtain \eqref{eq:stable1}.
\end{proof}

\begin{theorem}[Tensorization]\label{thm:rcd-tensor}
  For $i=1,2$ let $(X_i,d_i,m_i)$ be $\rcd^*(K,N_i)$ spaces. Then the
  product space $(X_1\times X_2,d,m_1\otimes m_2)$, defined by
  \begin{align*}
    d\big((x,y),(x',y')\big)^2=d_1(x,x')^2 + d_2(y,y')^2\;,
  \end{align*}
  also satisfies $\rcd^*(K,N_1+N_2)$.
\end{theorem}

\begin{proof}
  The result will follow indirectly: According to
  Theorem~\ref{thm:grad-est} below, the $\rcd^*(K,N_i)$-conditions
  will imply the Bakry--Ledoux conditions $\bl(K,N_i)$ on the first
  and second factor. According to \cite[Thm.~5.2]{AGS12}, this implies
  that the product space satisfies $\bl(K,N_1+N_2)$.  Now
  Theorems~\ref{thm:BEW2CDE} and \ref{thm:RCDKN-equiv} imply that the
  $\rcd^*(K,N_1+N_2)$ condition holds on the product space.
\end{proof}

\begin{remark}
  Let us also briefly sketch an alternative more direct argument using
  characterization (i) of Theorem~\ref{thm:RCDKN-equiv}:
  First, \cite[Thm.~6.17]{AGS11b} yields that the Cheeger energy on
  the product space is again quadratic. Since $(X_i,d_i,m_i)$ are in
  particular strong $\cd(K,\infty)$ spaces, they are essentially
  non-branching according to Definition~\ref{def:ess-non-branch} by
  \cite[Thm.~1.1]{RS12}. This implies that also the product space is
  essentially non-branching. The latter can be seen using the fact
  that if $\gamma=(\gamma_1,\gamma_2)$ is a geodesic in $X_1\times
  X_2$, then $\gamma_i$ are geodesics in $X_i$. Finally, the reduced
  curvature-dimension condition tensorizes under the essentially
  non-branching assumption. This follows from the same arguments as in
  \cite[Thm.~4.1]{BS10}, where tensorization has been proven under the
  slightly stronger assumption that the full space is non-branching.
 \end{remark}

We conclude with a globalization property of the $\rcdkn$ condition.

\begin{theorem}[Local-to-global]\label{thm:rcd-locglob}
  Let $(X,d,m)$ be a strong $\cd^e_{\text{loc}}(K,N)$ space with
  $m\in\cP_2(X,d)$ and assume that it is locally infinitesimally
  Hilbertian in the following sense: there exists a countable covering
  $\{Y_i\}_{i\in I}$ by closed sets with $m(Y_i)>0$ such that the
  spaces $(Y_i,d,m_i)$ are infinitesimally Hilbertian, where
  $m_i=m(Y_i)^{-1}m\vert_{Y_i}$. Then $(X,d,m)$ satisfies the $\rcdkn$
  condition.
\end{theorem}

\begin{proof}
  Using characterization (ii) in Theorem~\ref{thm:RCDKN-equiv}, the
  assertion is a direct consequence of the fact that both
  infinitesimal Hilbertianity and the \emph{strong} $\ecdkn$ condition
  by themselves have the local-to-global property. Indeed, by
  \cite[Thm.~6.20]{AGS11b} the mms $(X,d,m)$ is again infinitesimally
  Hilbertian, i.e. the associated Cheeger energy is quadratic. By
  Theorem~\ref{thm:cde-locglob} it also satisfies the strong $\ecdkn$
  condition.
\end{proof}

\begin{remark}
  It is also possible to establish local--to--global property by
  passing through the corresponding result for $\cdskn$ with the aid
  of Theorem~\ref{thm:RCDKN-equiv}. This requires to check that the
  (quite complicated) proof of globalization for $\cdskn$ in
  \cite[Thm.~5.1]{BS10} also works under the slightly weaker
  ess. non-branching assumption. Thus, we prefer to give an
  independent and, to our knowledge, novel argument in the preceding
  proof.
\end{remark}

\subsection{Dimension dependent functional inequalities}
\label{sec:FI}

Here we present dimensional versions of classical transport
inequalities. Namely, we show that the new entropic
curvature-dimension condition entails improvements of the HWI
inequality, the logarithmic Sobolev inequality and the Talagrand
inequality taking into account the dimension bound. These results can
be seen as finite dimensional analogues of the famous results by
Bakry--\'Emery \cite{BE85} and Otto--Villani \cite{OV00}.

Given a probability measure $\mu\in\cP_2(X,d)$ we define the
\emph{Fisher information} by
\begin{align*}
  I(\mu)~=~4\int\wug{\sqrt{f}}^2\dd m\;,
\end{align*}
provided that $\mu=f m$ is absolutely continuous with a density
$f$ such that $\sqrt{f}\in D(\ch)$. Otherwise we set
$I(\mu)=+\infty$. With this notation,
the equality \eqref{eq:slope-FI}, which is valid on $\rcd (K,\infty)$ spaces,
means $\abs{ \nabla^-\ent}(f m)~=~I(f m)$.

\begin{theorem}[$N$-HWI inequality]\label{thm:N-HWI}
  Assume that the mms $(X,d,m)$ satisfies the $\ecdkn$ condition. Then
  for all $\mu_0,\mu_1\in\cP_2(X,d,m)$,
  \begin{align}\label{eq:UWI}
  %\tag{{N-HWI}($K$)}
    \frac{U_N(\mu_1)}{U_N(\mu_0)}~\leq~\ckn{W_2(\mu_0,\mu_1)} + \frac1N\skn{W_2(\mu_0,\mu_1)}\sqrt{I(\mu_0)}\;.
  \end{align}
\end{theorem}

\begin{proof}
  We can assume that $I(\mu_0)=\abs{\nabla^-\ent}(\mu_0)$ is finite,
  as otherwise there is nothing to prove. Let $(\mu_t)_{t\in[0,1]}$ be
  the constant speed geodesic connecting $\mu_0$ to $\mu_1$ given by
  the $\ecdkn$ condition. Since $(K,N)$-convexity of $\ent$ along the
  geodesic $(\mu_t)$ implies usual $K$-convexity along the same
  geodesic we have
  \begin{align*}
    \limsup\limits_{t\searrow0}\frac{\ent(\mu_t)-\ent(\mu_0)}{t}~\leq~\ent(\mu_1)
    - \ent(\mu_0) - \frac{K}{2} W_2(\mu_0,\mu_1)^2\;.
  \end{align*}
  On the other hand, we have
  \begin{align}\nonumber
    \liminf\limits_{t\searrow0}\frac{\ent(\mu_t)-\ent(\mu_0)}{t}~&\geq~-\limsup\limits_{t\searrow0}\frac{\max\{\ent(\mu_0)-\ent(\mu_t),0\}}{t}\\\label{eq:nhwi1}
    ~&\geq~-\abs{\nabla^-\ent}(\mu_0)\cdot W_2(\mu_0,\mu_1)\;.
  \end{align}
  Thus $\big(\ent(\mu_t)-\ent(\mu_0)\big)^2=o(t)$ as $t\to0$. By
  Lemma~\ref{lem:convex-below-tangent} and a Taylor expansion of
  $x\mapsto e^{-x/N}$ we obtain
 \begin{align*}
   \frac{U_N(\mu_1)}{U_N(\mu_0)}~&\leq~\ckn{\theta} + \frac{\skn{\theta}}{\theta\cdot U_N(
   \mu_0)}\cdot\liminf\limits_{t\searrow0}\frac{U_N(\mu_t)-U_N(\mu_0)}{t}\\
   &=~\ckn{\theta} - \frac{\skn{\theta}}{\theta\cdot
     N}\cdot\limsup\limits_{t\searrow0}\frac{\ent(\mu_t)-\ent(\mu_0)}{t}\;,
  \end{align*}
  where we set $\theta=W_2(\mu_0,\mu_1)$. Applying the estimate
  \eqref{eq:nhwi1} again yields the claim.
\end{proof}

\begin{corollary}[$N$-LogSobolev inequality]\label{cor:N-LogSob}
  Assume that $(X,d,m)$ is a $\ecdkn$ space with $K>0$ and that
  $m\in\cP_2(X,d)$. Then for all $\mu\in\cP_2(X,d,m)$,
  \begin{align}
  %\tag{N-LSI($K$)}
  \label{eq:N-LSI}
    KN\left[\exp\left(\frac{2}{N}\ent(\mu)\right) - 1\right]~\leq~I(\mu)\;.
  \end{align}
\end{corollary}
The LHS obviously is bounded from below by $2K\cdot \ent(\mu)$.

\begin{proof}
  We apply the $N$-HWI inequality from Theorem~\ref{thm:N-HWI} to the
  measures $\mu_0=\mu$ and $\mu_1=m$. Noting that $U_N(m)=1$ and
  setting $\theta=W_2(\mu,m)$ we obtain
  \begin{align*}
    \exp\left(\frac{1}{N}\ent(\mu)\right)~\leq~\ckn{\theta} + \frac1N\skn{\theta}\sqrt{I(\mu)}\;.
  \end{align*}
 Taking the square and using Young's inequality $2ab\leq Ka^2+K^{-1}b^2$ we obtain
 \begin{align*}
    \exp\left(\frac{2}{N}\ent(\mu)\right)~&\leq~\ckn{\theta}^2 + \frac2N\skn{\theta}\ckn{\theta}\sqrt{I(\mu)} + \frac1{N^2}\skn{\theta}^2I(\mu)\\
      &\leq~\left(\ckn{\theta}^2+\frac{K}{N}\skn{\theta}^2\right)\left[1 + \frac{1}{KN}I(\mu)\right]\;.
\end{align*}
Since $\ckn{\cdot}^2+\frac{K}{N}\skn{\cdot}^2=1$, this yields the claim.
\end{proof}

\begin{corollary}[$N$-Talagrand inequality]\label{cor:N-Talagrand}
  Assume that $(X,d,m)$ is a $\ecdkn$ space with $K>0$ and that
  $m\in\cP_2(X,d)$. Then $W_2(\mu,m)\le \sqrt{\frac NK}\,\frac\pi2$  for any $\mu\in\cP_2(X,d,m)$ and
 \begin{align}\label{eq:T}
  %\tag{{T}($K,N$)}
   \ent(\mu)~\geq~-N\log\cos\left(\sqrt{\frac{K}{N}}W_2(\mu,m)\right)\;.
  \end{align}
\end{corollary}
Note that under the given upper bound on $W_2(\mu,m)$, the RHS in the above estimate is bounded from below by $\frac K2 \, W_2(\mu,m)^2$.
\begin{proof}
  The claims follow immediately by applying the $N$-HWI inequality
  \eqref{eq:UWI} from Theorem~\ref{thm:N-HWI} to the measures
  $\mu_0=m$ and $\mu_1=\mu$ and noting that $U_N(m)=1$ as well as
  $I(m)=0$.
\end{proof}

It is interesting to note that in the spirit of Otto--Villani a
slightly weaker Talagrand-like inequality can also be derived from the
$N$-LogSobolev inequality.

\begin{proposition}
  Let $(X,d,m)$ be a $\cd(K',\infty)$ space for some $K'\in\R$ such
  that $m\in\cP_2(X,d)$. Assume that the $N$-LogSobolev inequality
  \eqref{eq:N-LSI} holds for some $K>0$. Then for any
  $\mu\in\cP_2(X,d,m)$,
 \begin{align}\label{eq:T-alt}
  %\tag{{T}($K,N$)}
   W_2(\mu,m)~\leq~\sqrt{\frac{N}{K}\left[\exp\left(\frac{2}{N}\ent(\mu)\right)-1\right]}\;.
  \end{align}
\end{proposition}

\begin{proof}
  We fix $\mu\in\cP_2(X,d,m)$ and introduce the function
  $A:[0,\infty)\to\R_+$ defined by
  \begin{align*}
    A(t)~=~W_2(\cH_t\mu,\mu) + \sqrt{\frac{N}{K}\left[\exp\left(\frac{2}{N}\ent(\cH_t\mu)\right)-1\right]}\;.
  \end{align*}
  Obviously, $A(0)$ equals the right hand side of \eqref{eq:T-alt},
  while $A(t)\to W_2(\mu,m)$ as $t\to\infty$. Thus it is sufficient to
  prove that $A$ is non-increasing. First note that under the
  $\cd(K',\infty)$ condition we have the estimate
  \begin{align}\label{eq:ddtrW2}
    \ddtr W_2(\cH_t\mu,\mu)~\leq~\sqrt{I(\cH_t\mu)}\;.
  \end{align}
  Indeed, using triangle inequality we find
  \begin{align*}
    \limsup\limits_{h\searrow0}\frac{1}{h}\Big(W_2(\cH_{t+h}\mu,\mu)-W_2(\cH_t\mu,\mu)\Big)~&\leq~\limsup\limits_{h\searrow0}\frac{1}{h} W_2(\cH_{t+h}\mu,\cH_t\mu)~=~\abs{\dot{(\cH_t\mu)}}\;.
  \end{align*}
  Now \eqref{eq:ddtrW2} follows from the fact that $\cH_t\mu$ is a
  metric gradient flow of $\ent$ by virtue of the Energy Dissipation
  Equality \eqref{eq:EDE} and \eqref{eq:slope-FI}. Moreover, we
  calculate
  \begin{align*}
   \ddtr \sqrt{\frac{N}{K}\left[\exp\left(\frac{2}{N}\ent(\cH_t\mu)\right)-1\right]}~&=~\left(NK\left[\exp\left(\frac{2}{N}\ent(\cH_t\mu)\right)-1\right]\right)^{-\frac12}\ddtr \ent(\cH_t\mu)\\
        &~=~-\left(NK\left[\exp\left(\frac{2}{N}\ent(\cH_t\mu)\right)-1\right]\right)^{-\frac12}I(\cH_t\mu)\\
        &~\leq~-\sqrt{I(\cH_t\mu)}\;,
  \end{align*}
  where we have used \eqref{eq:N-LSI} in the last step. Thus we have
  shown that $\ddtr A(t)\leq 0$ which yields the claim.
\end{proof}

\begin{remark}
  Note that the arguments in the proofs above are of a purely metric
  nature. The preceding results can be formulated and proven verbatim
  in the setting of Section~\ref{sec:metric} by replacing $\ent$ with
  a $(K,N)$-convex function $S$ on a metric space, the Fisher
  information $I$ with the slope $\abs{\nabla^-S}$ and $\cH_t\mu$ with
  the gradient flow of $S$. However, for concreteness we choose to work
  in the Wasserstein framework.
\end{remark}

\section{ Equivalence of $\cd^e(K,N)$ and the Bochner Inequality $\be(K,N)$}\label{sec:cde-bochner}

In this section we will study properties of the gradient flow $H_t f$
of the (quadratic) Cheeger energy $\ch$ in $L^2 (X,m)$. We refer to
Section~\ref{sec:recap} and references therein for notations and basic
properties of them.

\subsection{From $\cd^e(K,N)$ to $\bl(K,N)$ and $\be(K,N)$}
\label{sec:cde2bochner}

In this section we study the analytic consequences of the Riemannian
curvature-dimension condition. In particular, we show that it implies
a pointwise gradient estimate in the spirit of Bakry--Ledoux. This in
turn allows us to establish the full Bochner inequality.

As an immediate consequence of Definition~\ref{def:riemcdkn} and
Theorem~\ref{thm:contraction} we obtain the following Wasserstein
expansion bound. Recall from Proposition \ref{prop:simp-control} that
this bound in turn implies a slightly weaker and simpler bound not
involving the function $\skn{\cdot}$.

\begin{theorem}[$W_2$-expansion bound]\label{thm:W2-contraction}
  Let $(X,d,m)$ be a $\rcdkn$ space. For any $\mu,\nu\in\cP_2(X,d)$
  and $0<s,t$ we have
  \begin{align}\label{eq:W2-contraction}
    \skn{\frac12
      W_2(\cH_t\mu,\cH_s\nu)}^2~\leq~&e^{-K(s+t)}\skn{\frac12
      W_2(\mu,\nu)}^2\\\nonumber &+
    \frac{N}{K}\Big(1-e^{-K(s+t)}\Big)\frac{\big(\sqrt{t}-\sqrt{s}\big)^2}{2(s+t)}\;.
  \end{align}
  In particular, in the limit $s\to t$ and $\nu\to\mu$ we have
  \begin{align}\label{eq:W2-contraction-inf}
    W_2(\cH_t\mu,\cH_s\nu)^2~\leq~&e^{-2Kt}W_2(\mu,\nu)^2 +
    \frac{N}{K}\frac{1-e^{-2Kt}}{4t^2}\cdot\abs{s-t}^2\\\nonumber & +
    o\big(W_2(\mu,\nu)^2+\abs{t-s}^2\big)\;.
  \end{align}
\end{theorem}

Next we will show that \eqref{eq:W2-contraction} implies
Bakry--Ledoux's gradient estimate. To do it with minimal a priori
regularity assumptions, we will introduce another condition, which is
satisfied for each $\rcd(K',\infty)$ space (see Remark~\ref{rem:regularity}
below).

\begin{assumption} \label{ass:Ch-reg} $(X,d,m)$ is a length metric
  measure space satisfying $\supp m = X$ and
  \eqref{eq:exp-int}. In addition, every $f \in D ( \ch )$ with
  $\wug{f} \le 1$ has a 1-Lipschitz representative.
\end{assumption}

\begin{theorem}[Bakry--Ledoux gradient estimate]\label{thm:grad-est}
  Let $(X,d,m)$ be an infinitesimally Hilbertian metric measure space
  satisfying Assumption~\ref{ass:Ch-reg}. Assume that
  \eqref{eq:W2-contraction} with $K\in \R$, $N\in(0,\infty)$ holds for
  the measures $( \bH_t \eta ) m$ and $(\bH_s\sigma)m$ instead of
  $\cH_t \mu$ and $\cH_s \nu$ for each $\mu= \eta m$ and $\nu= \sigma
  m$ in $\cP_2 (X,d,m)$ and $t,s \ge 0$. Then
  \begin{align}\label{eq:grad-est}
    \wug{\bH_t f}^2 + \frac{4Kt^2}{N\big(e^{2Kt}-1\big)}\abs{\Delta \bH_t f}^2~\leq~e^{-2Kt}\bH_t\big(\wug{f}^2\big)\;.
  \end{align}
  $m$-a.e. in $X$ for any $f\in D(\ch)$ and $t>0$.
\end{theorem}

Before giving the proof we note the following result, which gives a
stronger version of the gradient estimate involving the Lipschitz
constant under more restrictions on $f$.

\begin{proposition}\label{prop:grad-est-ref}
  Let $(X,d,m)$ be an infinitesimally Hilbertian metric measure space
  satisfying Assumption~\ref{ass:Ch-reg}. If \eqref{eq:grad-est} holds and $\wug{f}\in
  L^\infty(X,m)$ then $\bH_t f$, $\bH_t ( \wug{f}^2 )$ and $\Delta
  \bH_t f$ have continuous representatives satisfying everywhere in
  $X$:
  \begin{align}\label{eq:grad-est-ref}
    \abs{\nabla \bH_t f}^2 + \frac{4Kt^2}{N\big(e^{2Kt}-1\big)}\abs{\Delta
      \bH_t f}^2~\leq~e^{-2Kt} \bH_t\big(\wug{f}^2\big)\;.
  \end{align}
\end{proposition}

\begin{remark}\label{rem:regularity}
  Under $\rcd (K',\infty)$, Assumption~\ref{ass:Ch-reg} is always
  satisfied (see \cite{AGMR12,AGS11b,AGS12}).  Moreover, with the aid
  of Theorem~\ref{thm:gf-identification}, the other assumption in
  Theorem~\ref{thm:grad-est} easily yields \eqref{eq:W2-contraction}
  in this case. Conversely, the assumptions in
  Theorem~\ref{thm:grad-est} implies $\rcd (K,\infty)$.  Indeed, by
  Proposition~\ref{prop:simp-control}, \eqref{eq:W2-contraction}
  yields the $W_2$-contraction estimate, which corresponds to
  \eqref{eq:WC0}. Under Assumption~\ref{ass:Ch-reg}, such an estimate
  yields Bakry--\'Emery's $L^2$-gradient estimate (see
  \cite[Cor.~3.18]{AGS12}, \cite[Thm.~2.2]{Kuw10}).  Then $\rcd
  (K,\infty)$ follows from \cite[Thm~4.18]{AGS12} under
  Assumption~\ref{ass:Ch-reg} again.

  Note that $\rcd (K',\infty)$ ensures some regularization property of
  $\bH_t$. For instance, $\bH_t f (x) = \int_X f \, \dd \cH_t
  \delta_x$ holds $m$-a.e. for every $f \in L^2 (X,m)$. Moreover,
  this representative of $\bH_t f$ satisfies the strong Feller
  property, that is, $x \mapsto \int_X f \, \dd \cH_t \delta_x$ is
  bounded and continuous for any bounded measurable $f$ (see
  \cite[Thm.~6.1]{AGS11b}, \cite[Thm.~7.1]{AGMR12}).
\end{remark}

\begin{proof}[Proof of Theorem \ref{thm:grad-est}]
  For simplicity of presentation, we give a proof when $(X,d)$ is a
  geodesic space. One can easily extend the argument to the length
  space case. We first consider the case that $f$ is bounded and
  Lipschitz with bounded support. Let us denote $\tilde{\bH}_t f (x) :
  = \int_X f \,\dd \cH_t \delta_x$, which is a representative of
  $\bH_t f$, see Remark \ref{rem:regularity}. For $x , y \in X$, $x
  \neq y$ and $t , s \ge 0$ and any coupling $\pi_{s,t}$ of $\cH_{s} (
  \d_{x} )$ and $\cH_{t} ( \d_{y} )$, we have
  \begin{equation} \label{eq:couple} \tilde{\bH}_{s} f (x) -
    \tilde{\bH}_{t} f (y) \le \int_{X \times X} \abs{ f (z) - f (w) }
    \pi_{s,t} (\mathrm{d} z \mathrm{d} w )\; .
  \end{equation}
  Since $| f (z) - f (w) | \le \Lip (f) d (x,y)$, \eqref{eq:couple}
  and \eqref{eq:W2-contraction} yield
  \begin{align*}
    \skn{ \frac{1}{2 \Lip(f)} ( \tilde{\bH}_s f (x) - \tilde{\bH}_t f
      (y) ) }^2 & \le \skn{ \frac{1}{2} W_1 ( \cH_s ( \d_x ) , \cH_t (
      \d_y ) ) }^2
    \\
    & \le \skn{ \frac{1}{2} W_2 ( \cH_{s} ( \d_x ) , \cH_{t} ( \d_y ) )
    }^2
    \\
    & \hspace{-12em} \le \e^{-K (s+t)} \skn{\frac12 d ( x , y )}^2 +
    \frac{N ( 1 - \e^{-K(s+t)} )}{2 K ( s+t )} ( \sqrt{t} - \sqrt{s}
    )^2 \;.
  \end{align*}
  It implies that the map $(u,z) \mapsto \tilde{\bH}_u f (z)$ is
  locally Lipschitz on $( 0 , 1 )\times X$ and hence $u \mapsto
  \tilde{\bH}_u f (z)$ is differentiable $\cL^1$-a.e.
  for each fixed $z \in X$,
  where $\cL^1$ is the one-dimensional Lebesgue measure.

  The first step is to show the following inequality:
  \begin{align} \label{eq:grad-est0} \abs{\nabla \tilde{\bH}_t f}
    (x)^2 + \frac{4Kt^2}{N\big( e^{2Kt}-1 \big)} \left(
      \frac{\partial}{\partial t} \tilde{\bH}_t f (x) \right)^2 \leq
    \e^{-2Kt} \tilde{\bH}_t ( \abs{ \nabla f }^2 ) (x)
  \end{align}
  for each $x \in X$ and $t > 0$ such that $u \mapsto \bH_u f (x)$ is
  differentiable at $t$.
  Let $y \in X$ and $s \ge 0$.
  let us define $r = r (x,y; s,t) > 0$ and $G_r f : X \to \R$ by
  \begin{align*}
    r & : =
    \begin{cases}
      W_2 ( \cH_s ( \d_x ) , \cH_t ( \d_y ) )^{1/2} & \mbox{if $W_2 (
        \cH_s ( \d_x ) , \cH_t ( \d_y ) ) > 0$},
      \\
      d (x,y) & \mbox{otherwise}.
    \end{cases}
    \\
    G_r f (z) & : = \sup_{z' ; \; d(z, z') \in (0,r)} \frac{ \abs{f(z)
        - f (z')} }{d (z,z')}\;.
  \end{align*}
  Then by taking a coupling $\pi_{s,t}$ as a minimizer of $W_2 ( \cH_s
  ( \d_x ) , \cH_t ( \d_y ) )$ in \eqref{eq:couple},
  \begin{align} \nonumber \int_{X \times X} & \abs{ f (z) - f (w) }
    \pi_{s,t} ( \mathrm{d}z \mathrm{d} w ) \\ \nonumber & = \int_{X
      \times X} \abs{ f (z) - f (w) } 1_{ \{ d(z,w) \le r \} }
    \pi_{s,t} ( \mathrm{d}z \mathrm{d} w ) \\ \nonumber & \quad +
    \int_{X \times X} \abs{ f (z) - f (w) } 1_{ \{ d(z,w) > r \} }
    \pi_{s,t} ( \mathrm{d}z \mathrm{d} w ) \\ \nonumber & \le \int_{X
      \times X} G_r f (z) d( z, w ) \pi_{s,t} ( \mathrm{d}z \mathrm{d}
    w ) + 2 \| f \|_\infty \pi_{s,t} ( d > r ) \\ \nonumber & \le
    \left( \int_{X} ( G_r f )^2 d \cH_s (\d_x) \right)^{1/2} W_2 (
    \cH_{s} (\d_x) , \cH_t (\d_{y}) ) \\ \label{eq:difference} &
    \hspace{6em} + \frac{ 2 \| f \|_\infty }{r^2} W_2 ( \cH_{s}
    (\d_{x}) , \cH_{t} (\d_{y} ) )^2\; .
  \end{align}
  After substituting \eqref{eq:difference} into \eqref{eq:couple}, we
  apply \eqref{eq:W2-contraction} with $\mu = \d_y$ and $\nu
  = \d_x$ to obtain
  \begin{align} \nonumber \tilde{\bH}_{s} & f (x) - \tilde{\bH}_{t} f
    (y) \\ \nonumber & \le \tilde{\bH}_{s} ( ( G_r f )^2 )(x)^{1/2} \\
    \nonumber & \quad \times 2 s_{K/N}^{-1} \Bigg( \sqrt{ \e^{-K
        (s+t)} s_{K/N} \left( \frac12 d(x,y) \right)^2 + \frac{ N ( 1
        - \e^{- K (s+t)} )}{2 K ( s + t ) } ( \sqrt{t} - \sqrt{s} )^2
    } \Bigg)
    \\
    & \qquad \label{eq:diff2} + 2 \| f \|_{\infty} W_2 ( \cH_s (\d_x)
    , \cH_t (\d_y) )
  \end{align}
  by using our choice of $r$.
  Since the inequality \eqref{eq:grad-est0} is quadratic w.r.t. scalar
  multiplication of $f$, we may assume without loss of generality that
  \begin{align*}
    | \nabla \tilde\bH_t f |(x) = \limsup_{y \to x} \frac{ [ \tilde\bH_tf (x) - \tilde\bH_tf (y) ]_+ }{
      d ( x, y ) }\;.
  \end{align*}
  Take a sequence $( y_n )_{n \in \N}$ in $X$ such that $\displaystyle
  \lim_{n \to \infty} \frac{ \tilde\bH_tf (x) - \tilde\bH_tf (y_n) }{ d(x, y_n ) } = |
  \nabla \tilde\bH_tf | (x) $ holds.  Take $\a \in \R \setminus \{ 0 \}$, which
  will be specified later.  For each $n \in \N$, let us take $s_n = t
  + \a d ( x, y_n )$ and $r_n = r ( x, y_n ; s_n , t )$.  Then we have
  \begin{align*}
    \lim_{n \to \infty} \frac{ \tilde{\bH}_{s_n} f (x) - \tilde{\bH}_t
      f (y_n) }{ d ( x, y_n ) } & = \lim_{n \to \infty} \left( \a
      \frac{ \tilde{\bH}_{s_n} f (x) - \tilde{\bH}_t f (x) }{ s_n - t
      } + \frac{ \tilde{\bH}_{t} f (x) - \tilde{\bH}_t f (y_n) }{ d (
        x, y_n ) } \right)
    \\
    & = \a \frac{\partial}{\partial t} \tilde{\bH}_t f (x) + | \nabla
    \tilde{\bH}_t f | (x)\;.
  \end{align*}
  Take $\eps > 0$ arbitrary.  Since $G_r f$ is non-decreasing in $r$,
  by substituting $s = s_n$, $y = y_n$ into \eqref{eq:diff2}, dividing
  both sides by $d ( x , y_n )$ and letting $n \to \infty$, we obtain
  \begin{align} \nonumber \a \frac{\partial}{\partial t} \tilde{\bH}_t
    f (x) + | \nabla \tilde{\bH}_t f | (x) & \le \tilde{\bH}_{t} ( |
    G_\eps f |^2 ) (x)^{1/2} \\ \nonumber & \qquad \times \sqrt{ \e^{-
        2 K t} + \a^2 \frac{ N ( 1 - \e^{- 2 K t} )}{4 K t^2} } \;.
  \end{align}
  Here we used the fact that $\tilde{\bH}_{u} ( | G_\eps f |^2 )$ is
  continuous in $u$ (see Remark~\ref{rem:regularity}).
  Let $v_\a$ be a unit vector in $\R^2$ of the
  form $\lambda ( 1, \a \sqrt{ N ( \e^{2Kt} - 1 ) / (4Kt^2)} )$ with
  $\lambda > 0$.  Then, by rewriting the last inequality after $\eps
  \downarrow 0$, we obtain
  \begin{align*}
    v_\a \cdot \left( | \nabla \tilde{\bH}_t f | (x) , \sqrt{ \frac{4 K
          t}{ N ( \e^{2Kt} - 1 )}} \frac{\partial}{\partial t}
      \tilde{\bH}_t f (x) \right) \le \e^{-Kt} \tilde{\bH}_{t} ( |
    \nabla f |^2 ) (x)^{1/2} \;.
  \end{align*}
  By optimizing this inequality in $\a$, we obtain
  \eqref{eq:grad-est0}.

  The second step is to show the following for any bounded and
  Lipschitz $f \in D ( \ch )$: For each $t > 0$ and $m$-a.e.~$x \in
  X$,
  \begin{align}\label{eq:grad-est1}
    \abs{ \nabla \tilde{\bH}_t f } (x)^2 +
    \frac{4Kt^2}{N\big(\e^{2Kt}-1\big)}\abs{ \Delta \bH_{t} f (x)}^2
    \leq \e^{-2Kt} \tilde{\bH}_t \big( \abs{ \nabla f }^2 \big) (x)
    \;.
  \end{align}
  For each $x \in X$, we already know that $t \mapsto \tilde{\bH}_t f
  (x)$ is differentiable for $\cL^1$-a.e.~$t \in [ 0 , \infty )$.
  Thus the Fubini theorem yields that the set $I \subset ( 0 , \infty
  )$ given by
  \begin{equation*}
    I :=
    \left\{
      t \in ( 0 , \infty)
      \; \left| \;
        \mbox{$t \mapsto \tilde{\bH}_t f (x)$
          is differentiable
          for $m$-a.e.~$x \in X$}
      \right.
    \right\}
  \end{equation*}
  is of full $\cL^1$-measure. Take $t \in I$. Then we have
  $\displaystyle \frac{\partial}{\partial t} \tilde{\bH}_t f (x) =
  \Delta \bH_t f (x) $ $m$-a.e.~and hence \eqref{eq:grad-est0} yields
  \eqref{eq:grad-est1}.  Thus it suffices to show $I = ( 0 , \infty )$
  to prove \eqref{eq:grad-est1}.  Indeed, for any $t \in ( 0, \infty
  )$, there is $s \in I$ with $s < t$.  Since $( u , z ) \mapsto
  \tilde{\bH}_u f (z)$ is locally Lipschitz, the dominated convergence
  theorem implies
  \begin{align*}
    \tilde{\bH}_{t - s} \big( \frac{\partial}{\partial s}
    \tilde{\bH}_{s} f \big) (x) & = \tilde{\bH}_{t-s} \left( \lim_{u
        \to 0} \frac{\tilde{\bH}_{s + u } f - \tilde{\bH}_{s} f }{u}
    \right) (x) = \frac{\partial}{\partial t} \tilde{\bH}_t f (x)
  \end{align*}
  and hence $u \mapsto \tilde{\bH}_u f (x)$ is differentiable at $t$
  for any $x \in X$.

  Finally we prove the assertion for $f \in D ( \ch )$.  Let $f_n \in
  D ( \ch )$ be a sequence of bounded Lipschitz functions on $X$
  converging to $f$ in $W^{1,2}$ strongly and $| \nabla f_n | \to |
  \nabla f |_w$ in $L^2$.  Then $\Delta \bH_t f_n \to \Delta \bH_t f$
  in $L^2$ and hence the conclusion follows
  (cf.~\cite[Thm.~6.2]{AGS11b}).
\end{proof}

\begin{proof}[Proof of Proposition \ref{prop:grad-est-ref}]
  Note first that \eqref{eq:grad-est} implies $\rcd (K,\infty)$ as in
  Remark~\ref{rem:regularity}. Take $\mu_0, \mu_1 \in \cP_2 (X,d,m)$
  with bounded densities and bounded supports and $\pi$ be a dynamic
  optimal coupling satisfying $( e_i )_\# \pi = \mu_i$ for $i=0,1$.
  Note that $(e_t)_\# \pi \ll m$ holds since $\rcd (K,\infty)$ holds.
  Let $f_n \in D (\ch)$ be an approximating sequence of $f$ as
  above. We may assume that $( | \nabla f_n | )_{n \in \N}$ is
  uniformly bounded without loss of generality since $\wug{f} \in
  L^\infty (X,m)$.  Then $( \Delta \bH_t f_n )_{n \in \N}$ is
  uniformly bounded in $L^\infty (X,m)$ by \eqref{eq:grad-est1}.  We
  may assume also that $\bH_t ( | \nabla f_n |^2 )$ and $\Delta \bH_t
  f_n$ converges $m$-a.e.~by taking a subsequence if necessary.  We
  apply \eqref{eq:grad-est1} to $f_n$ to obtain
  \begin{align*}
    & \left|
      \int_X \bH_t f_n  \,\dd \mu_1
        -
      \int_X \bH_t f_n \,\dd \mu_0
    \right|
    \le
    \int_{\geo (X)} \int_0^1
      | \nabla \tilde{\bH}_t f_n | ( \gamma_t ) | \dot{\gamma}_t |
    \,\dd t \,
    \pi ( \dd \gamma )
    \\
    & \hspace{2em} \le
    W_2 ( \mu_0 , \mu_1 )
    \int_0^1 \int_{\geo(X)}
    \left(
      \e^{-2Kt} \bH_t ( | \nabla f_n |^2 ) ( \gamma_t )
        -
      \frac{4Kt^2}{N ( \e^{-2Kt} - 1 )} | \Delta \bH_t f_n ( \gamma_t ) |^2
    \right)
    \, \pi ( \dd \gamma )
    \,\dd t .
  \end{align*}
  Then, as $n \to \infty$, the dominated convergence theorem yields
  \begin{multline*}
    \left|
      \int_X \bH_t f  \,\dd \mu_1
        -
      \int_X \bH_t f \,\dd \mu_0
    \right|
    \\
    \le
    W_2 ( \mu_0 , \mu_1 )
    \int_0^1 \int_{\geo(X)}
    \left(
      \e^{-2Kt} \bH_t ( \wug{f}^2 ) ( \gamma_t )
        -
      \frac{4Kt^2}{N ( \e^{-2Kt} - 1 )} | \Delta \bH_t f ( \gamma_t ) |^2
    \right)
    \, \pi ( \dd \gamma )
    \,\dd t .
  \end{multline*}
%  we can replace $f_n$ in the last inequality with $f$.
  By the strong Feller property, $\bH_t ( \wug{f}^2 )$ has a continuous representative.
  Since $\Delta \bH_{t/2} f \in L^\infty (X,m)$ by \eqref{eq:grad-est} with $t/2$ instead of $t$,
  the strong Feller property again implies that $\Delta \bH_t f = \bH_{t/2} \Delta \bH_{t/2} f$
  has a continuous representative.
  Thus by taking $\mu_0$ and $\mu_1$ as a uniform distribution on $B_r (x_0)$ and $B_r (x_1)$
  respectively and letting $r \to 0$, we obtain
  \begin{multline*}
    \left| \bH_t f (x_0) - \bH_t f (x_1) \right|
    \\
    \le
    d ( x_0 , x_1 )
    \sup_{z \in B_{2 d(x_0 , x_1 )} (x_0)}
    \left[
      \e^{-2Kt} \bH_t ( \wug{f}^2 ) (z)
        -
      \frac{4Kt^2}{N ( \e^{-2Kt} - 1 )} | \Delta \bH_t f (z) |^2
    \right]
  \end{multline*}
  for $m$-a.e.~$x_0, x_1$.
  Thus $\bH_t f$ has a Lipschitz representative and \eqref{eq:grad-est-ref} holds.
\end{proof}

\begin{definition}\label{def:BL}
  We say that $(X,d,m)$ satisfies the \emph{Bakry--Ledoux gradient
    estimate} $\bl(K,N)$ with $K\in\R$, $N\in(0,\infty)$ if for any $f\in D(\ch)$
  and $t > 0$
%   with $\wug{f}\in L^\infty(X,m)$ and $t>0$ we have that $\bH_tf$ has a
%   bounded Lipschitz representative and
  \begin{align}\label{eq:bl}
     \wug{\bH_t f}^2
% \abs{\nabla \bH_t f}^2
   + \frac{2t}{N}C(t)\abs{\Delta \bH_t f}^2
   ~\leq~
   e^{-2Kt}\bH_t\big(\wug{f}^2\big)\quad m\text{-a.e. in }X\;,
  \end{align}
  where $C>0$ is a function satisfying $C(t)=1+O(t)$ as $t\to0$.
\end{definition}

Now Theorem~\ref{thm:grad-est} can be reformulated as follows: For an
infinitesimally Hilbertian metric measure space, the $W_2$-expansion
bound \eqref{eq:W2-contraction} implies the $\bl(K,N)$ condition under
Assumption~\ref{ass:Ch-reg}.  Indeed, \eqref{eq:grad-est} states that
\eqref{eq:bl} holds with $C(t)=2K t/(\e^{2Kt}-1)$. The Bakry--Ledoux
gradient estimate $\bl (K,N)$ will allow us to establish the full
Bochner inequality including the dimension term in $\rcdkn$
spaces. This extends the result in \cite{AGS11b}, where a Bochner
inequality without dimension term has been established on
$\rcd(K,\infty)$ spaces. Let us also make precise what we mean by
Bochner's inequality, or the Bakry--\'Emery condition.

\begin{definition}\label{def:BE}
  We say that an infinitesimally Hilbertian metric measure space
  $(X,d,m)$ satisfies the \emph{Bakry--\'Emery condition} $\be(K,N)$,
  or \emph{Bochner inequality}, with $K\in\R$, $N\in(0,\infty)$ if for
  all $f\in D(\Delta)$ with $\Delta f\in W^{1,2}(X,d,m)$ and all $g\in
  D(\Delta) \cap L^\infty (X,m)$ with $g \ge 0$ and $\Delta g\in
  L^\infty(X,m)$ we have
  \begin{align}\label{eq:Bochner}
    &\frac12\int\Delta g \wug{f}^2 \dd m - \int g\langle\nabla(\Delta f),\nabla f\rangle \dd m%\\\nonumber
 ~\geq~K\int g \wug{f}^2\dd m + \frac{1}{N}\int g\big(\Delta f\big)^2\dd m\;.
  \end{align}
\end{definition}

To investigate the relation between Bochner's inequality and the
Bakry-Ledoux gradient estimate, we introduce a mollification of the
semigroup $h^\eps$ given by
\begin{align}\label{eq:sg-moll}
  h^\eps
  f~=~\int_0^\infty\frac{1}{\eps}\eta\left(\frac{t}{\eps}\right)\bH_tf\,\dd
  t\;,
\end{align}
with a non-negative kernel $\eta\in C^\infty_c(0,\infty)$ satisfying
$\int_0^\infty\eta(t)\dd t=1$
for $f \in L^p (X,m)$, $1 \le p \le \infty$.
Note that $h^\eps f \in D (\Delta)$ and
\begin{equation} \label{eq:sg-moll2}
\Delta h^\eps f = - \int_0^1 \frac{1}{\eps} \eta' \left( \frac{t}{\eps} \right)\bH_t f\,\dd t
\end{equation}
for any $f \in L^p (X,m)$, $1 \le p < \infty$.

\begin{theorem}[Bochner inequality]\label{thm:Bochner}
  Let $(X,d,m)$ be an infinitesimally Hilbertian metric measure space
  satisfying $\bl(K,N)$. Then the Bochner inequality $\be(K,N)$ holds.
\end{theorem}

\begin{proof}
  In the language of Dirichlet forms, this is proven in
  \cite[Cor.~2.3, (vi)$\Rightarrow$(i)]{AGS12}. We sketch here an
  argument following basically the ideas developed in \cite{GKO10} in
  the setting of Alexandrov spaces.

  We will first prove \eqref{eq:Bochner} for $f\in D(\Delta)\cap
  L^\infty (X, m)$ with $\Delta f\in D(\Delta)\cap L^\infty(X,m)$ and
  for $g$ satisfying $\Delta g \in D (\ch)$ additionally.  From
  \eqref{eq:grad-est} we obtain immediately
  \begin{align}\label{eq:Bochner1}
    &\int g\wug{\bH_t f}^2\dd m +
   \frac{2t}{N}C(t)\int g \abs{\Delta \bH_t f}^2\dd
    m
    % \\
%     \nonumber
%     ~&
    \leq~e^{-2Kt}\int g\bH_t\big(\wug{f}^2\big)\dd m\;.
  \end{align}
  This will yield \eqref{eq:Bochner} by subtracting $\int
  g\wug{f}^2\dd m$ on both sides, dividing by $t$ and taking the limit
  $t\searrow0$. Indeed, for the left hand side of \eqref{eq:Bochner1},
  we can argue exactly as in the proof of
  \cite[Thm.~4.6]{GKO10}, using the Leibniz rule \ref{eq:Leibniz},
  and note in addition that
  \begin{align*} %\label{eq:Bochner2}
    \lim\limits_{t\to0} \frac{2}{N}C(t)\int g \abs{\Delta \bH_t f}^2\dd
    m~=~\frac{2}{N}\int g\big(\Delta f\big)^2\dd m\;.
  \end{align*}
  For the right hand side of \eqref{eq:Bochner1},
  by a similar calculation, we obtain
  \begin{multline} \label{eq:Bochner3}
    \frac{1}{t}
    \left(
      \int g \bH_t \big( \wug{f}^2 \big) \dd m
      -
      \int g \wug{f}^2 \, \dd m
    \right)
    \\
    =
    - \frac{1}{t} \left(
      \int \bH_t g f \Delta f \,\dd m
      -
      \int g f \Delta f \,\dd m
    \right)
    + \frac{1}{2t} \left(
    \int \Delta \bH_t g \cdot f^2 \,\dd m
    -
    \int \Delta g \cdot f^2 \,\dd m
    \right)\;.
  \end{multline}
  Since $\Delta g, f^2 , f \Delta f \in D ( \ch )$, it converges to
  $\int \Delta g \wug{f}^2 \,\dd m$ as $t \to 0$ and thus we obtain
  \eqref{eq:Bochner}.  To obtain the estimate \eqref{eq:Bochner} for
  general $f$, we approximate $f$ by $h^\eps (f\wedge R)$ and $g$ by
  $T_{\eps'} g$. By \eqref{eq:sg-moll2}, these functions have the
  expected regularity. First we take $\eps' \to 0$.  Since $\wug{f},
  \wug{\Delta f} \in L^1 (X,m) \cap L^\infty (X,m)$ by virtue of
  \eqref{eq:bl} and \eqref{eq:sg-moll2}, it goes well.  Next we take
  $R \to \infty$.  Since $\lim_{R \to \infty} \ch ( f\wedge R - f ) =
  0$ and $\ch ( f \wedge R ) \le \ch (f)$, we can show $\wug{ h^\eps (
    f \wedge R ) }^2 \to \wug{ h^\eps f}^2$ weakly in $L^1 (X,m)$
  similarly as in the proof of \cite[Thm.~4.6]{GKO10}.  The same
  argument also works for $\langle \nabla \Delta h^\eps ( f \wedge R
  ), \nabla h^\eps (f \wedge R ) \rangle$ with the aid of
  \eqref{eq:sg-moll2}.  Again \eqref{eq:sg-moll2} helps the
  convergence of the term involving $N$.  Finally we take $\eps \to
  0$.  we can employ the approximation argument in
  \cite[Thm.~4.6]{GKO10} again when arguing this limit to conclude the
  convergence of the same kind.  The additional dimension term posing
  no difficulty at this moment.
\end{proof}

Also the converse implication holds.
Originally, this was proven by
Bakry and Ledoux in \cite{BL06} in the setting of Gamma calculus. See
also the work of Wang \cite{Wa11}, where the equivalence of gradient
estimates and Bochner's inequality has been rediscovered in the
setting of smooth Riemannian manifolds. Note that the function $C$ in
the next proposition gives a stronger estimate than
\eqref{eq:grad-est} for large $t$.

\begin{proposition}\label{prop:Bochner2BEW}
  Let $(X,d,m)$ be an infinitesimally Hilbertian mms satisfying the
  Bakry--\'Emery condition $\be(K,N)$.
%   Assume that for any $f\in
%   D(\ch)$ with $\wug{f}\in L^\infty(X,m)$ and $t>0$ we have that
%   $H_tf$ has a bounded Lipschitz representative.
  Then the $\bl(K,N)$ condition holds with $C(t)=(1-\e^{-2Kt})/2Kt$.
\end{proposition}

\begin{proof}
  In the language of Dirichlet forms, this is basically proven in
  \cite[Cor.~2.3, (i)$\Rightarrow$(vi)]{AGS12}. Let us sketch the
  argument.

  As in the proof of Theorem~\ref{thm:Bochner}, we first assume $f \in
  D ( \Delta ) \cap L^\infty (X,m)$ with $\Delta f \in D ( \Delta )
  \cap L^\infty (X,m)$. Fix $g\geq0$ with $g\in D(\Delta)\cap
  L^\infty(X,m)$ and $\Delta g\in L^\infty(X,m) \cap D (\ch)$ and
  consider the function
  \begin{align*}
    h(s)~:=~\e^{-2Ks}\int H_{s}g\wug{H_{t-s}f}^2\,\dd m\;.
  \end{align*}
  One estimates the derivative of $h$ as:
  \begin{align*}
    h'(s)~&=~-2K\e^{-2Ks}\int H_{s}g\wug{H_{t-s}f}^2\,\dd m\\
          &\quad +\e^{-2Ks}\int \Delta H_{s}g\wug{H_{t-s}f}^2\,\dd m\\
          &\quad -2\e^{-2Ks}\int H_{s}g \langle\nabla H_{t-s}f,\nabla \Delta H_{t-s}f\rangle\,\dd m\\
          &\geq~\frac{2}{N}\e^{-2Ks}\int H_{s}g\big(\Delta H_{t-s}f\big)^2\,\dd m\\
          &\geq~\frac{2}{N}\e^{-2Ks}\int  g\big(\Delta H_{t}f\big)^2\,\dd m\;,
  \end{align*}
  where we have used \eqref{eq:Bochner} in the first and Jensen's
  inequality in the second inequality. A computation similar to the
  first equality in \eqref{eq:Bochner3}, deduces that $h$ is
  continuous at $0$ and $t$ since $g,f \in L^\infty$.  Thus,
  integrating from $0$ to $t$ we obtain:
  \begin{align*}
\int g\wug{ H_{t}f}^2\,\dd m + \frac{1-\e^{-2Kt}}{NK}\int  g\big(\Delta H_{t}f\big)^2\,\dd m
&\leq \e^{-2Kt}\int H_{t} g\wug{f}^2\,\dd m\;.
  \end{align*}
  For the general case,
  we approximate $f \in D (\ch)$ and $g \in L^2 (X,m) \cap L^\infty (X,m)$
  by $h^\eps ( f \wedge R )$ and $h^{\eps'} g$ respectively.
  As we did in the proof of Theorem~\ref{thm:Bochner},
  We can take $R \to \infty$, $\eps \to 0$ to obtain the last inequality
  for $f$ and $h^{\eps'} g$.
  Since $h^{\eps'} g$ converges to $g$
  with respect to weak${}^*$ topology in $L^\infty (X,m)$ as $\eps' \to 0$,
  the last inequality holds for general $f$ and $g$.
  This is sufficient to complete the proof.
\end{proof}

\subsection{From $\bl(K,N)$ to $\ecdkn$}\label{sec:bochner2cde}

In the following section, we will always assume that $(X,d,m)$ is an
infinitesimally Hilbertian metric measure space and that
Assumption~\ref{ass:Ch-reg} holds. We will show that the Bakry--Ledoux
gradient estimate $\bl(K,N)$ implies the entropic curvature-dimension
condition $\ecdkn$ and thus the $\rcdkn$ condition.

Our approach is strongly inspired by the recent work \cite{AGS12} of
Ambrosio, Gigli and Savar\'e. We follow their presentation and adopt
to a large extent their notation. Under Assumption~\ref{ass:Ch-reg} we
can rely on the results in \cite{AGS12}, since the condition
$\bl(K,N)$ is more restrictive than the classical Bakry--\'Emery
gradient estimate $\bl(K,\infty)$. In particular, we already know that
the Riemannian curvature condition $\rcd(K,\infty)$ holds true,
c.f. Remark~\ref{rem:regularity}, \cite[Cor.~4.18]{AGS12}. Moreover,
we also know that the semigroup $H_t$ coincides with the gradient flow
$\cH_t$ of the entropy in $\cP_2(X,d)$ in the sense of
Theorem~\ref{thm:gf-identification}.

The crucial ingredient in our argument is the action estimate
Proposition~\ref{prop:action-est}. This result calls for an extensive
regularization procedure that was already used in \cite{AGS12}, both
for curves in $\cP_2(X)$ and for the entropy functional, which we will
discuss below. The main difference of our approach compared to
\cite{AGS12} is that our argument now relies on the analysis of the
(nonlinear) gradient flow $(\nu_t)_{t\ge0}$ for the functional $-U_N$
instead of the analysis of the (linear) heat flow which is the
gradient flow $(\mu_t)_{t\ge0}$ for $\ent$. Both flows are related to
each other via time change:
\begin{align*}
  \nu_t=\mu_{\tau_t},\qquad \partial_t \tau_t=\frac1N U_N(\mu_{\tau_t})\;.
\end{align*}
More precisely, the following lemma yields that this time change is
well-defined.

\begin{lemma}\label{lem:welldef-tau}
  Let $\rho\in D(\ent)\subset \cP_2(X,d,m)$. Then there exist
  constants $a,c>0$ depending only on $\abs{\ent(\rho)}$ and the
  second moment of $\rho$ such that a map $\tau:[0,a]\to[0,\infty)$
  can be defined implicitly by
  \begin{align}\label{eq:deftau}
    \int_0^{\tau_{t}}\exp\left(\frac1N\ent(\cH_r\rho)\right)\dd r~=~t
  \end{align}
  and for any $t\in[0,a]$ we have $\tau_t\leq c t$. Moreover, we have
  \begin{align}\label{eq:diff-tau}
    \ddt \tau_t~=~\frac1N U_N(\cH_{\tau_t}\rho)\;.
  \end{align}
\end{lemma}

\begin{proof}
  We first derive a lower bound on $\ent(\cH_r\rho)$. Let us set
  $V(x)=d(x_0,x)$ for some $x_0\in X$. By \eqref{eq:exp-int} we have
  that $z=\int \e^{-V^2}\dd m<\infty$ and $\tilde m=z^{-1}\e^{-V^2}m$
  is a probability measure. Now \cite[Thm.~4.20]{AGS11a} (together
  with a trivial truncation argument) yields that
  \begin{align*}
    \int V^2\dd (\cH_r\rho)~\leq~\e^{4r}\Big(\ent(\rho)+2\int V^2\dd\rho\Big)~=:~\e^{4r}c'\;.
  \end{align*}
  Hence we obtain
  \begin{align*}
    \ent(\rho)~\geq~\ent(\cH_r\rho)~=~\ent(\cH_r\rho|\tilde m) - \int V^2\dd (\cH_r\rho) - \log z~\geq~-\e^{4r}c' - \log z\;.
  \end{align*}
  Now fix some $R>0$ and put $a=z^{-1}\int_0^R\exp(-\e^{4r}c'/N)\dd
  r$, $c=z \exp(\e^{4R}c'/N)$. Then define the function
  $F:[0,R]\to[0,F(R)]$ via
  $F(u)=\int_0^u\exp\big(\ent(\cH_r\rho)/N\big)\dd r$. Since $F$ is
  strictly increasing with $F(0)=0$ and $F(R)\geq a$ by the preceding
  estimate we can define $\tau_t=F^{-1}(t)$ for any
  $t\in[0,a]$. Moreover, we have $F(u)\geq c^{-1}u$ for any $u\leq R$
  which implies $\tau_t\leq ct$. Finally \eqref{eq:diff-tau} follows
  immediately from the differentiability of $F$.
\end{proof}

More generally, given a continuous curve $(\rho_s)_{s\in[0,1]}$ in
$\cP_2(X,d,m)$ such that $\max_s\abs{\ent(\rho_s)}<\infty$ we define a
time change $\tau_{s,t}$ implicitly via
\begin{align}\label{eq:tau-def2}
  \int_0^{\tau_{s,t}}\exp\left(\frac1N\ent(\cH_r\rho_s)\right)\dd r~=~st
\end{align}
for $s\in[0,1]$ and $t\in[0,a]$ satisfying
\begin{align}\label{eq:bd-tau}
  \tau_{s,t}~\leq~c\cdot st
\end{align}
for suitable constants $a,c>0$ depending only on a uniform bound on
the entropy and second moments of $( \rho_s )_{s \in [0,1]}$ and moreover
\begin{align}\label{eq:diff-tau2}
 \partial_t \tau_{s,t}~=~s\cdot U_N(\cH_{\tau_{s,t}}\rho_s)\;.
\end{align}

We will now describe the regularization procedure needed in the
sequel. We will use the notion of \emph{regular curve} as introduced
in \cite[Def.~4.10]{AGS12}. Briefly, a curve $(\rho_s)_{s\in[0,1]}$
with $\rho_s=f_s m$ is called regular if the following are satisfied:
\begin{itemize}
\item $(\rho_s)$ is $2$-absolutely continuous in $\cP_2(X,d)$,
\item $\ent(\rho_s)$ and $I(\bH_tf_s)$ are bounded for $s\in[0,1], t\in[0,T]$,
\item $f\in C^1\big([0,1],L^1(X,m)\big)$ and $\Delta^{(1)}f\in C\big([0,1],L^1(X,m)\big)$,
\item $f_s=h^\eps\tilde f_s$ for some $\tilde f_s\in L^1(X,m)$ and $\eps>0$.
\end{itemize}
Here $I(f)=4\ch(\sqrt{f})$ denotes the Fisher information,
$\Delta^{(1)}$ denotes the generator of the semigroup $H_t$ in
$L^1(X,m)$ and $h^\eps$ is the mollification of the semigroup given in
\eqref{eq:sg-moll}. In the sequel we will denote by $\dot f_s$ the
derivative of $[0,1]\ni s\mapsto f_s\in L^1(X,m)$. We will mostly
denote both the generator in $L^1$ and in $L^2$ by $\Delta$. In the
following we will need an approximation result which is a
reinforcement of \cite[Prop.~4.11]{AGS12}.

\begin{lemma}[Approximation by regular curves]\label{lem:reg-curves}
  Let $(\rho_s)_{s\in[0,1]}$ be an $AC^2$-curve in $\cP_2(X,d,m)$ such
  that $s\mapsto\ent(\rho_s)$ is bounded and continuous. Then there
  exists a sequence of regular curves $(\rho_s^n)$ with the following
  properties. As $n\to\infty$ we have for any $s\in[0,1]$:
\begin{align}\label{eq:reg-curves1}
  W_2(\rho^n_s,\rho_s)~&\to~0\;,\\\label{eq:reg-curves1a}
   \limsup \abs{\dot\rho^n_s}~&\leq~\abs{\dot\rho_s}\quad\text{a.e. in }[0,1]\;,\\\label{eq:reg-curves2}
  \ent(\cH_r\rho^n_s)~&\to~\ent(\cH_r\rho_s)\quad\forall r>0\;,\\\label{eq:reg-curves3}
  \tau^n_{s,t}~&\to~\tau_{s,t}\;,
\end{align}
where $\tau^n$ and $\tau$ denote the time changes defined via the
curves $(\rho^n_s)$ and $(\rho_s)$ respectively on $[0,1]\times[0,a]$
for suitable $a>0$. Moreover, for any $\delta>0$ there are $n_0,r_0>0$
such that for any $n> n_0$ and $r<r_0$ and all $s\in[0,1]$ we have:
\begin{align}\label{eq:reg-curves4}
  \abs{\ent(\rho_s)-\ent(\cH_r\rho^n_s)}~<~\delta\;.
\end{align}
\end{lemma}

\begin{proof}
  Following \cite[Prop.~4.11]{AGS12} we employ a threefold
  regularization procedure. We trivially extend $(\rho_s)_s$ to $\R$
  with value $\rho_0$ in $(-\infty,0)$ and $\rho_1$ in
  $(1,\infty)$. Given $n$, we first define
  $\rho^{n,1}_s=\cH_{1/n}\rho_s=f^{n,1}_sm$. The second step
  consists in a convolution in the time parameter. We set
    \begin{align*}
    \rho^{n,2}_s=f^{n,2}m\;,\qquad f^{n,2}_s~=~\int_\R f^{n,1}_{s-s'}\psi_{n}(s')\dd s'\;,
  \end{align*}
  where $\psi_n(s)=n\cdot\psi(n s)$ for some smooth kernel
  $\psi:\R\to\R_+$ with $\int\psi(s)\dd s=1$. Finally, we set
  \begin{align*}
    \rho^{n}_s~=~f^n_sm\;,\qquad f^n_s~=~h^{1/n} f^{n,2}_s\;,
  \end{align*}
  where $h^\eps$ denotes a mollification of the semigroup given by
  \eqref{eq:sg-moll}. It has been proven in \cite[Prop.~4.11]{AGS12}
  that $(\rho^n_s)_{s\in[0,1]}$ constructed in this way is a regular
  curve and that \eqref{eq:reg-curves1} holds. \eqref{eq:reg-curves1a}
  follows from the convexity properties of $W_2^2$ and the
  $K$-contractivity of the heat flow. Let us now prove
  \eqref{eq:reg-curves2}. Note that on the level of measures the
  semigroup commutes with the regularization, i.e. $H_r\rho^n_s =
  \tilde\rho_s^n$ where $\tilde\rho_s:=\cH_r\rho_s$. Thus it is
  sufficient to prove \eqref{eq:reg-curves2} for $r=0$. By
  \eqref{eq:reg-curves1} and lower semicontinuity of the entropy we
  have $\ent(\rho_s)\leq\liminf_{n\to\infty}\ent(\rho_s^n)$. On the
  other hand, using the convexity properties of the entropy and the
  fact that $\cH_r$ and thus also $h^{1/n}$ decreases the entropy we
  estimate
  \begin{align}\nonumber
    \ent(\rho^n_s)~&\leq~\ent(\rho^{n,2}_s)~\leq~\int\psi_n(s')\ent(\cH_{1/n}\rho_{s-s'})\dd s'~\leq~\int\psi_n(s')\ent(\rho_{s-s'})\dd s'\\\label{eq:contra0}
    &\leq~\ent(\rho_s) + \int\psi_n(s')\abs{\ent(\rho_{s-s'})-\ent(\rho_s)}\dd s'\;.
  \end{align}
  The last term vanishes as $n\to\infty$ since $s\mapsto\ent(\rho_s)$ is
  uniformly continuous by compactness. Thus we obtain
  $\limsup_{n\to\infty}\ent(\rho_s^n)\leq \ent(\rho_s)$ and hence
  \eqref{eq:reg-curves2}. To prove \eqref{eq:reg-curves3} define the
  functions
  \begin{align*}
    F_n(u)~&=~\int_0^u\exp\left(\frac1N\ent(\cH_r\rho^n_s)\right)\dd r\;, &
    F(u)~&=~\int_0^u\exp\left(\frac1N\ent(\cH_r\rho_s)\right)\dd r\;.
  \end{align*}
  Arguing as in Lemma~\ref{lem:welldef-tau} we see that
  $\tau^n_{s,t}=F^{-1}_n(st)$ and $\tau_{s,t}=F^{-1}(st)$ can be
  defined simultaneously on $[0,1]\times[0,a]$ and satisfy
  $\abs{F_n(u)-F_n(v)}\geq c^{-1} \abs{u-v}$ for suitable constants
  $a,c>0$ independent of $n$. Since moreover, by
  \eqref{eq:reg-curves2} and dominated convergence we have $F_n\to F$
  pointwise as $n\to\infty$ we conclude the convergence
  \eqref{eq:reg-curves3}.

  We now prove the last statement of the lemma. To conclude the proof
  we proceed by contradiction. Assume the contrary, i.e. that there
  exists $\delta>0$ and a sequences $n_k\to\infty,r_k\to0$ and
  $(s_k)\subset[0,1]$ such that
  $\abs{\ent(\rho_{s_k})-\ent(\cH_{r_k}\rho^{n_k}_{s_k})}~\geq~\delta$
  for all $k$. Taking into account \eqref{eq:contra0} and the fact
  that $\cH_r$ decreases entropy we must have that for all $k$
  sufficiently large
  \begin{align}\label{eq:contra1}
    \ent(\rho_{s_k})-\ent(\cH_{r_k}\rho^{n_k}_{s_k})~\geq~\delta\;.
  \end{align}
  By compactness we can assume $s_k\to s_0$ as $k\to\infty$ for some
  $s_0\in[0,1]$. We claim that as $k\to\infty$ we have
  $\cH_{r_k}\rho^{n_k}_{s_k}\to \rho_{s_0}$ in $W_2$. Indeed, since
  $\cH_r$ satisfies a Wasserstein contraction and by the convexity
  properties of $W_2$ the regularizing procedure increases distances
  at most an exponential factor (see also
  \cite[Prop.~4.11]{AGS12}). Hence, the triangle inequality yields
  \begin{align*}
    W_2(\rho_{s_0},\cH_{r_k}\rho^{n_k}_{s_k})~&\leq~W_2(\rho_{s_0},\cH_{r_k}\rho_{s_0}) + W_2(\cH_{r_k}\rho_{s_0}, \cH_{r_k}\rho^{n_k}_{s_0}) + W_2(\cH_{r_k}\rho^{n_k}_{s_0},\cH_{r_k}\rho^{n_k}_{s_k})\\
    &\leq~W_2(\rho_{s_0},\cH_{r_k}\rho_{s_0}) + \e^{-Kr_k}W_2(\rho_{s_0},\rho^{n_k}_{s_0}) + \e^{-K r_k}W_2(\rho_{s_0},\rho_{s_k}) + o(1)\;,
  \end{align*}
  and the claim follows from the continuity of $\cH_r$ at $r=0$,
  \eqref{eq:reg-curves1} and the continuity of the curve $(\rho_s)$.
  Letting now $k\to\infty$ in \eqref{eq:contra1}, using continuity of
  $s\mapsto\ent(\rho_s)$ and lower semicontinuity of $\ent$, we obtain
  the following contradiction:
  \begin{align*}
    0~=~\ent(\rho_{s_0})-\ent(\rho_{s_0})~\geq~\limsup\limits_{k\to\infty} \Big(\ent(\rho_{s_k}) - \ent(\cH_{r_k}\rho^{n_k}_{s_k})\Big)~\geq~\delta\;.
  \end{align*}
\end{proof}

The following calculations will be a crucial ingredient in our
argument. For a detailed justification see \cite[Lem.~4.13,
4.15]{AGS12}. The only difference here is the additional time change
in the semigroup. For the following lemmas let $(\rho_s)_{s\in[0,1]}$
be a regular curve and let $\phi:X\to\R$ be Lipschitz with bounded
support. Let $\theta:[0,1]\to[0,\infty)$ be an increasing $C^1$
function with $\theta(0)=0$ and set
$\rho_{s,\theta}=\cH_{\theta_s}\rho_s=f_{s,\theta}m$. Moreover, we set
$\phi_s=Q_s\phi$ for $s\in[0,1]$, where
\begin{align*}
  Q_s\phi(x)~:=~\inf\limits_{y\in X} \left[ f(y) + \frac{d^2(x,y)}{2s} \right]
\end{align*}
denotes the Hopf-Lax semigroup. We refer to \cite[Sec. 3]{AGS11a} for
a detailed discussion. We recall that since $(X,d)$ is a length space,
$Q$ provides a solution to the Hamilton--Jacobi equation, i.e.
\begin{align*}
  \frac{\dd}{\dd s}Q_s\phi~=~-\abs{\nabla Q_s\phi}
\end{align*}
for a.e. $s\in[0,1]$, see \cite[Prop.~3.6]{AGS11a}. Moreover, we have
the a priori Lipschitz bound (\cite[Prop.~3.4]{AGS11a})
\begin{align}\label{eq:Lip-HL}
  \Lip(Q_s\phi)~\leq~2\Lip(\phi)\;.
\end{align}

\begin{lemma}\label{lem:calculus1}
  The map $s\mapsto\int\phi_s\dd\rho_{s,\theta}$ is absolutely
  continuous and we have for a.e. $s\in[0,1]$:
  \begin{align}\label{eq:calculus1}
    \frac{\dd}{\dd s}\int\phi_s\dd\rho_{s,\theta}~=~
    \int\left(-\frac12|\nabla\varphi_s|^2\,f_{s,\theta}+\dot
      f_s \bH_{\theta_s}\varphi_s+\dot\theta_s \Delta f_{s,\theta}\cdot\varphi_s \right)\dd m\;.
  \end{align}
\end{lemma}

We use a regularization $E_\eps$ of the entropy functional where the
singularities of the logarithm a truncated. Let us define
$e_\eps:[0,\infty)\to\R$ by setting
$e_\eps'(r)=\log(\eps+r\wedge\eps^{-1} )+1$ and $e_\eps(0)=0$. Then for
any $\rho=fm\in\cP_2(X,d,m)$ we define
\begin{align*}
  E_\eps(\rho):=\int e_\eps(f)\dd m\;,\qquad
  U^\eps_N(\rho)=\exp\left(-\frac1N E_\eps(\rho)\right)\;.
\end{align*}
Moreover we set $p_\eps(r)=e_\eps'(r^2)-\log\eps - 1$. Note that for any
$\rho\in D(\ent)$ we have $E_\eps(\rho)\to\ent(\rho)$ as $\eps\to0$.

\begin{lemma}\label{lem:calculus2}
  The map $s\mapsto E_\eps(\rho_{s,\theta})$ is absolutely continuous
  and we have for all $s\in[0,1]$:
  \begin{align}\label{eq:calculus2}
    \frac{\dd}{\dd s}E_\eps(\rho_{s,\theta})~=~ \int \left( \dot f_s\bH_{\theta_s}g^\eps_{s,\theta}+\dot\theta_s \Delta f_{s,\theta}\cdot g^\eps_{s,\theta}\right)\dd m\;,
  \end{align}
  where we put $g^\eps_{s,r}=p_\eps(\sqrt{f_{s,r}})$.
\end{lemma}
We also need to introduce the time change related to the regularized
entropy. For fixed $\eps>0$ and let us define $\tau^\eps_{s,t}$
implicitly by
\begin{align}\label{eq:deftau-reg}
  \int_0^{\tau^\eps_{s,t}}\exp\left(\frac1N E_\eps(\cH_r\rho_s)\right)\dd r~=~st\;.
\end{align}

\begin{lemma}\label{lem:calculus3}
  $\tau^\eps$ is well defined on $[0,1]\times[0,a]$ and satisfies
  $\tau^\eps_{s,t}\leq c\cdot st$ for constants $a,c>0$ depending only
  on $\max_s\abs{\ent(\rho_s)}$ and the second moments of $(\rho_s)_{s
    \in [ 0, 1 ]}$.  For fixed $t$ the map $s\mapsto\tau^\eps_{s,t}$
  is $C^1$ on $[0,1]$ and we have:
  \begin{align}\label{eq:calculus3}
    \partial_s\tau^\eps_{s,t}~=~t\cdot
    U^\eps_N(\cH_{\tau^\eps}\rho_s)-\frac1N\int_0^{\tau^\eps}
    \frac{U^\eps_N(\cH_{\tau^\eps}\rho_s)}{U^\eps_N(\cH_r\rho_s)}\int_X\dot
    f_sH_rg^\eps_{s,r}\,\dd m \, \dd r\;.
  \end{align}
  Moreover, as $\eps\to0$ we have $\tau^\eps_{s,t}\to\tau_{s,t}$,
  where $\tau$ is the time change defined by \eqref{eq:deftau-reg}.
\end{lemma}

\begin{proof}
  Define the function
  $F_\eps(s,u)=\int_0^u\exp\left(E_\eps(\cH_r\rho_s)/N\right)\dd r$.
  Note that a uniform bound on $\abs{\ent(\rho_s)}$ implies a uniform
  bound on $\abs{E_\eps(\rho_s)}$ independent of $\eps$. Thus we can
  argue as in Lemma~\ref{lem:welldef-tau} to find $a,c$ such that
  $\tau_{s,t}^\eps$ is well-defined on $[0,1]\times[0,a]$ by
  $F_\eps(s,\tau^\eps_{s,t})=st$ and satisfies $\tau^\eps_{s,t}\leq c
  \cdot st$. Using Lemma~\ref{lem:calculus2} and the fact that
  $s\mapsto \dot f_s$ is continuous in $L^1(X,m)$, since $(\rho_s)_s$
  is a regular curve, we see that $s\mapsto E_\eps(\cH_r\rho_s)$ is
  $C^1$ for fixed $r\geq0$. Moreover, using the boundedness of
  $E_\eps(\cH_r\rho_s)$ we obtain that $F_\eps(\cdot,\cdot)$ is
  $C^1$. Thus the differentiability of $s\mapsto\tau^\eps_{s,t}$
  follows from the implicit function theorem and \eqref{eq:calculus3}
  is obtained by differentiating \eqref{eq:deftau-reg} w.r.t. $s$. The
  last statement about convergence follows as for
  \eqref{eq:reg-curves3} using that $E_\eps(\rho_s)\to\ent(\rho_s)$ as
  $\eps\to0$.
\end{proof}

We need the following integrations by parts and estimates for the
integrals appearing in \eqref{eq:calculus1},
\eqref{eq:calculus2}. Recall that $I(f)=4\int \wug{\sqrt{f}}^2\dd
m$ denotes the Fisher information of a measure $\rho=fm$.

\begin{lemma}\label{lem:alphabeta-est}
  Let $f=h^\eps\tilde f$ for some $\tilde f\in L^1_+(X,m)$ with
  $\tilde f m\in \cP_2(X,m)$. Then for any Lipschitz function $\phi$
  with bounded support we have
  \begin{align}\label{eq:beta-est}
    \int\ip{\nabla\phi,\nabla g^\eps}f\dd m + \int
    q_\eps(f)\ip{\nabla\sqrt{f},\nabla\phi}\dd m~=~-\int \phi\Delta f\dd
    m~\leq~2\Lip(\phi)\cdot\sqrt{I(f)}\;,
  \end{align}
  where $q_\eps(r)=\sqrt{r}\big(2-\sqrt{r}p_\eps'(\sqrt{r})\big)$ and
  $g^\eps=p_\eps(\sqrt{f})$. Moreover we have
  \begin{align}\label{eq:alpha-est}
    \int \wug{g^\eps}^2f\dd m ~\leq~-\int g^\eps\Delta f\dd
    m~\leq~I(f)\;.
  \end{align}
\end{lemma}

\begin{proof}
  We first obtain from \cite[Thm.~4.4]{AGS12}
  \begin{align*}
    -\int \phi\Delta f\dd m~&=~2\int\sqrt{f} \langle\nabla\phi , \nabla \sqrt{f}\rangle\,\dd m\;.%~\leq~L\cdot\left(4\int \wug{\sqrt{f}}^2\dd m\right)^\frac12\;.
  \end{align*}
  Now the first equality in \eqref{eq:beta-est} is immediate from the
  chain rule \eqref{eq:chainrule-wug} for minimal weak upper gradients
  and integration by parts while the second inequality follows readily
  using H\"older's inequality. To prove \eqref{eq:alpha-est} we use
  that by \cite[Lem.~4.9]{AGS12} for any bounded non-decreasing
  Lipschitz function $\omega:[0,\infty)\to\R$ with
  $\sup_rr\omega'(r)<\infty$:
  \begin{align}\label{eq:rough-ibp}
    -\int \omega(f)\Delta^{(1)} f \dd m~\geq~4\int f
    \omega'(f)\wug{\sqrt{f}}^2\dd m\;.
  \end{align}
  Further note that $r\cdot e_\eps''(r)\leq1$ and hence $4r\cdot
  e_\eps''(r)\geq 4r^2\big(e_\eps''(r)\big)^2 =
  r\big(p_\eps'(\sqrt{r})\big)^2$. Hence we get by the chain rule:
  \begin{align}\label{eq:fwugg}
    f\wug{g^\eps}^2~=~f\big(p_\eps'(\sqrt{f})\big)^2\wug{\sqrt{f}}^2~\leq~ 4fe_\eps''(f)\wug{\sqrt{f}}^2\;.
  \end{align}
  Combining this with \eqref{eq:rough-ibp} yields the first inequality
  in \eqref{eq:alpha-est}. For the second inequality note that, since
  we already now that $\rcd(K,\infty)$ holds, $\tilde\bH_\delta
  g^\eps$ is bounded and Lipschitz for all $\delta>0$ by
  \cite[Thm.~6.8]{AGS11b}.
  Hence \cite[Thm.~4.4]{AGS12} and H\"older's
  inequality yield
  \begin{align*}
    -\int\Delta f\bH_\delta g^\eps\dd m~&=~2\int\sqrt{f}\langle\nabla \bH_\delta g^\eps , \nabla \sqrt{f}\rangle\,\dd m~\leq~2\ch(\sqrt{f})^{\frac12}\cdot\left(\int f\wug{\bH_\delta g^\eps}\dd m\right)^{\frac12}\\
       &\leq~4\e^{-K\delta}\ch(\sqrt{f})\;,
  \end{align*}
  where we have used again \eqref{eq:fwugg} and $\bl(K,\infty)$ in the
  last step. Letting $\delta\to0$ yields the second inequality in
  \eqref{eq:alpha-est}.
\end{proof}

We will often use the following estimate (see
\cite[Lem.~4.12]{AGS12}). For any AC$^2$ curve
$(\rho_s)_{s\in[0,1]}$ with $\rho_s=f_s m$ and $f\in
C^1\big((0,1),L^1(X,m)\big)$ and any
Lipschitz function $\phi$
% $\phi\in D(\ch)\cap
% L^\infty(X,m)$ with $\wug{\phi}\in L^\infty(X,m)$
we have
\begin{align}\label{eq:speed-est}
  \left|\int \dot f_s \phi\dd m\right|~\leq~\abs{\dot\rho_s}\cdot\sqrt{\int \abs{\nabla\phi}^2f_s\dd m}\;.
\end{align}

The following result is the crucial ingredient in our argument.

\begin{proposition}[Action estimate]\label{prop:action-est}
  Assume that $(X,d,m)$ satisfies $\bl(K,N_0)$. Let
  $(\rho_s)_{s\in[0,1]}$ be a regular curve and $\phi$ a Lipschitz
  function with bounded support and denote by $\phi_s=Q_s\phi$ the
  Hamilton--Jacobi flow for $s\in[0,1]$. Then for any $N>N_0$ and
  $t\in[0,a]$:
\begin{align}\nonumber
&\int\varphi_1\dd\rho_{1,\tau}-\int\varphi_0\dd\rho_{0}-\frac12\int_0^1|\dot\rho_s|^2\e^{-2K\tau}\dd s+Nt\cdot\left[ U_N(\rho_0)-U_N(\rho_{1,\tau})\right]\\\label{eq:key-est}
\le~&C_1\int_0^1\frac\tau4\left[\left(\frac{U_N(\rho_{s,\tau})}{U_N(\rho_{s})}\right)^2-1-4\left(\frac{N}{N_0}-1\right) + C_2\tau\right]\,\dd s\;,
\end{align}
The constant $C_2$ depends only on $K$ and
$\max_{s\in[0,1]}\abs{\ent(\rho_s)}$, the constant $C_1$ depends in
addition on $\max_{s\in[0,1]}I(\rho_s)$ and $\phi$.
\end{proposition}

\begin{proof}
  For simplicity we assume that $\bl(K,N_0)$ holds with $C\equiv
  1$. We use the abbreviations $\a_r=\a_{s,r}=-\int g^\eps_{s,r}\Delta
  f_{s,r}\,\dd m$ and $\b_r=\b_{s,r}=\int \varphi_s \Delta f_{s,r}\dd
  m$. Moreover, we put $u_r=u_{s,r}=U^\eps_N(\rho_{s,r})$. We will
  also write $\a=\a_{s,\tau}$, $\b=\b_{s,\tau}$, $u=u^\eps_{s,\tau}$.

  Using Lemmas~\ref{lem:calculus1}, \ref{lem:calculus3} and \eqref{eq:lip-wug}, we obtain
  \begin{eqnarray*}
    (A)&:=&\int\varphi_1d\rho_{1,\tau}-\int\varphi_0d\rho_{0}-\frac12\int_0^1|\dot\rho_s|^2\e^{-2K\tau}\dd s\\
    % &=&\int_0^1\left[-\frac12|\dot\rho_s|^2\e^{-2K\tau}+\int \partial_s(\varphi_s\,f_{s,\tau})\dd m\right]\dd s\\
    &=&
    \int_0^1\left[-\frac12|\dot\rho_s|^2\e^{-2K\tau}+\int\left(-\frac12|\nabla\varphi_s|^2\,f_{s,\tau}+\dot f_sH_\tau\varphi_s+\dot\tau \Delta H_\tau f_s\cdot\varphi_s \right)\dd m\right]\dd s\\
    &\leq&
    \int_0^1 \left[-\frac12 |\dot\rho_s|^2\e^{-2K\tau}-\frac12\int \wug{\varphi_s}^2\,  f_{s,\tau} \dd m\right.\\
    &&\left.+\int \dot f_s\cdot H_{\tau}\varphi_s\,\dd m+ \beta t u
      -\beta\frac1N\int_0^\tau\frac{u}{u_r} \int \dot f_s\cdot H_rg^\eps_{s,r}\,\dd m\,\dd r\right]\,\dd s.
  \end{eqnarray*}
  Moreover, by Lemma~\ref{lem:calculus2}, we have
  \begin{eqnarray*}
    (B)&:=& Nt\cdot\left[ U^\eps_N(\rho_0)-U^\eps_N(\rho_{1,\tau})\right]=t\int_0^1U^\eps_N(\rho_{s,\tau})\partial_sE_\eps(\rho_{s,\tau})\dd s\\
    &=&
    t\int_0^1 U^\eps_N(\rho_{s,\tau})\cdot\int g^\eps_{s,\tau}\cdot\left[H_\tau \dot f_s+\dot\tau \Delta H_\tau f_s\right]\dd m\, \dd s\\
    &=& \int_0^1 \left[t u\cdot \int \dot f_s\cdot H_\tau g^\eps_{s,\tau}\,\dd m
      -t^2 u^2\alpha\right.\\
    &&\left. +tu\alpha\frac1N\int_0^{\tau} \frac{u}{u_r} \int \dot f_s\cdot H_rg^\eps_{s,r}\,\dd m\,\dd r\right]\,\dd s.
  \end{eqnarray*}
  Adding up
  \begin{eqnarray*}
    (A)+(B)&\le&
    \int_0^1 \left[-\frac12 |\dot\rho_s|^2\e^{-2K\tau}-\frac12\int  \wug{\varphi_s}^2\,  f_{s,\tau} \dd m
      +tu(\beta-tu\alpha)
    \right.\\
    &&
    \left.+\frac1\tau\int_0^\tau \int \dot f_s\,\e^{-K\tau}\cdot\left[ H_\tau\left(\varphi_s+tu g^\eps_{s,\tau}
        \right)\,
        -\frac\tau N(\beta-tu\alpha)\frac{u}{u_r}  H_rg^\eps_{s,r}\right]\,\dd m\,\e^{K\tau}\,\dd r\right]\dd s\\
    &\le&
    \int_0^1 \left[-\frac12\int  \wug{\varphi_s}^2\,  f_{s,\tau} \dd m
      +tu(\beta-tu\alpha)
    \right.\\
    &&
    \left.+\frac1\tau\int_0^\tau \frac12\int \left|\nabla\left[ H_\tau\left(\varphi_s+tu g^\eps_{s,\tau}
          \right)\
          -\frac\tau N(\beta-tu\alpha)\frac{u}{u_r}  H_r g^\eps_{s,r}\right]\right|^2f_s\,\dd m\,\e^{2K\tau}\,\dd r\right]\dd s\\
    &\le&
    \int_0^1 \left[-\frac12\int  \wug{\varphi_s}^2\,  f_{s,\tau} \dd m
      +tu(\beta-tu\alpha)
    \right.\\
    &&
    \left.+\frac1\tau\int_0^\tau \frac12\int \left|\nabla\left[ H_{\tau-r}\left(\varphi_s+tu g^\eps_{s,\tau}
          \right)\,
          -\frac\tau N(\beta-tu\alpha)\frac{u}{u_r}  g^\eps_{s,r}\right]\right|_{w}^2f_{s,r}\,\dd m\,\e^{2K(\tau-r)}\,\dd r\right]\dd s\\
    &&
    \left.-\frac1\tau\int_0^\tau \frac {r}{N_0}\int \left|\Delta\left[ H_\tau\left(\varphi_s+tu g^\eps_{s,\tau}
          \right)\,
          -\frac{\tau} {N}(\beta-tu\alpha)\frac{u}{u_r}  H_r g^\eps_{s,r}\right]\right|^2f_s\,\dd m\, \e^{2K\tau} \,\dd r\right]\dd s\\
    &=:& (C) + ([D+E]^2) +(F)\;.
\end{eqnarray*}
Here we have used \eqref{eq:speed-est} in the second inequality and in
the last inequality the Bakry--Ledoux gradient estimate $\bl(K,N_0)$
applied to the semigroup $H_r$ in the strong form given by Proposition
\ref{prop:grad-est-ref}. The last term will be estimated as follows
\begin{eqnarray*}
  (F)&\le&\int_0^1\left[-\frac1\tau\int_0^\tau \frac {r}{N_0}\left|\int \Delta\left[ H_\tau\left(\varphi_s+tu g^\eps_{s,\tau}
        \right)-\frac\tau N(\beta-tu\alpha)\frac{u}{u_r}  H_rg^\eps_{s,r}\right]f_s\,\dd m\right|^2\,\e^{2K\tau}
    \,\dd r\right]\dd s\\
  &=&\int_0^1\left[-\frac1\tau\int_0^\tau \frac {r}{N_0}\left|\beta-tu\alpha
      +\frac\tau N(\beta-tu\alpha)\frac{u}{u_r}\alpha_r\right|^2\,\e^{2K\tau}
    \,\dd r\right]\dd s\\
  &=&\int_0^1\left[-\frac1\tau\int_0^\tau \frac {r}{N_0}\left| \beta-tu\alpha\right|^2\cdot\left|1+\frac\tau N\frac{u}{u_r}\alpha_r\right|^2\,\e^{2K\tau}
    \,\dd r\right]\dd s.
\end{eqnarray*}

By virtue of Lemma~\ref{lem:alphabeta-est}, the second last term
$([D+E]^2)$ can be decomposed into
\begin{eqnarray*}
(E^2)&=&
\int_0^1\left[\frac1\tau\int_0^\tau \frac12\frac{\tau^2}{N^2}\left(\frac{u}{u_r}\right)^2 (\beta-tu\alpha)^2\,\e^{2K(\tau-r)}\int\wug{ g^\eps_{s,r}}^2f_{s,r}\,\dd m\,\dd r\right]\dd s\\
&\le&
\int_0^1\left[\frac1\tau\int_0^\tau \frac12\frac{\tau^2}{N^2}\left(\frac{u}{u_r}\right)^2\a_r\cdot (\beta-tu\alpha)^2\,\e^{2K(\tau-r)}\dd r\right]\dd s\;,
\end{eqnarray*}
\begin{eqnarray*}
(2DE)&=&\int_0^1-\frac1\tau\int_0^\tau(\beta-tu\alpha)\frac{u}{u_r}\frac{\tau}{N}\e^{2K(\tau-r)}\int \ip{\nabla H_{\tau-r}\left( \varphi_s + tug^\eps_{s,\tau}\right),\nabla g^\eps_{s,r}}f_{s,r}\,\dd m\,\,\dd r\dd s\\
% &=&\int_0^1-\frac1\tau\int_0^\tau(\beta-tu\alpha)\frac{u}{u_r}\e^{2K(\tau-r)}\frac{\tau}{N}\left[\int \ip{\nabla H_{\tau-r}\left( \varphi_s + tug^\eps_{s,\tau}\right),\nabla f_{s,r}}\,\dd m\right.\\
% &&\qquad\qquad\left.- \int q_\eps(f_{s,r})\ip{\nabla H_{\tau-r}\left( \varphi_s + tug^\eps_{s,\tau}\right),\nabla\sqrt{f_{s,r}}}\,\dd m\right]\,\dd r\dd s\\
&=&   \int_0^1\left[\frac1\tau\int_0^\tau \frac{\tau}{N}\frac{u}{u_r} (\beta-tu\alpha)^2\,\e^{2K(\tau-r)}+ \frac{\tau}{N}\frac{u}{u_r}(\beta-tu\alpha)\gamma^{(1)}\,\e^{2K(\tau-r)} \dd r\right]\dd s\;,
\end{eqnarray*}
where $\gamma^{(1)}=\int q_\eps(f_{s,r})\ip{\nabla H_{\tau-r}\left(
    \varphi_s + tug^\eps_{s,\tau}\right),\nabla\sqrt{f_{s,r}}}\,\dd
m$, and finally
 \begin{eqnarray*}
(D^2)&=&\int_0^1\left[\frac1\tau\int_0^\tau \frac12\int \left|\nabla H_{\tau-r}\left(\varphi_s+tu g^\eps_{s,\tau}
\right)
\right|_w^2f_{s,r}\,\dd m\,\e^{2K(\tau-r)}\,\dd r\right]\dd s\\
&\le&\int_0^1\left[\frac1\tau\int_0^\tau \frac12\int \left|\nabla \left(\varphi_s+tu g^\eps_{s,\tau}
\right)
\right|_w^2f_{s,\tau}\,\dd m\,\dd r\right.\\
&&\left.-
\frac1\tau\int_0^\tau \frac{\tau-r}{N_0}\int \left|\Delta H_{\tau-r}\left(\varphi_s+tu g^\eps_{s,\tau}
\right)
\right|^2f_{s,r}\,\dd m\,\e^{2K(\tau-r)}
\,\dd r\right]\dd s\\
&\le&\int_0^1\left[\frac1\tau\int_0^\tau \frac12\int \left|\nabla \left(\varphi_s+tu g^\eps_{s,\tau}
\right)
\right|_w^2f_{s,\tau}\,\dd m\,\dd r\right.\\
&&\left.-
\frac1\tau\int_0^\tau \frac{\tau-r}{N_0}\left|\int \Delta H_{\tau-r}\left(\varphi_s+tu g^\eps_{s,\tau}
\right)
f_{s,r}\,\dd m\right|^2\,\e^{2K(\tau-r)}
\,\dd r\right]\dd s\\
&\le&
\int_0^1 \left[\frac12\int |\nabla\varphi_s|_w^2\,  f_{s,\tau} \dd m
-tu\b-tu\gamma^{(2)}+ \frac12t^2u^2\a
-\frac1\tau\int_0^\tau \frac{\tau-r}{N_0} (\beta-tu\alpha)^2\,\e^{2K(\tau-r)}
\dd r\right]\dd s
\end{eqnarray*}
where $\gamma^{(2)}=\int q_\eps(f_{s,\tau})\ip{\nabla \varphi_s,\nabla\sqrt{f_{s,\tau}}}\,\dd m$ and
where we applied again the Bakry--Ledoux estimate $\bl(K,N_0)$, now to
the semigroup $H_{\tau-r}$. Summing up everything yields
\begin{eqnarray*}
(A)+(B)&\le&
\int_0^1 \left[-\frac12t^2u^2\alpha +\frac1N(\beta-tu\alpha)^2\cdot (G)+ (H)\right]\dd s
\end{eqnarray*}
where
\begin{eqnarray*}
  (H)&:=&-tu\gamma^{(2)} + \int_0^\tau\frac1N\frac{u}{u_r}(\beta-tu\alpha)\gamma^{(1)}\,\e^{2K(\tau-r)} \dd r\;,
\end{eqnarray*}
and
\begin{eqnarray*}
(G)&:=&\int_0^\tau\left[
-\frac{N}{N_0}\frac r\tau\left(1+\frac\tau N\frac{u}{u_r}\alpha_r\right)^2\,\e^{2K\tau}
+\frac\tau{2N}\left(\frac{u}{u_r}\right)^2\alpha_r\,\e^{2K(\tau-r)}
\right.\\
&&\left.\qquad
+\frac{u}{u_r}\,\e^{2K(\tau-r)}
-\frac{N}{N_0}\frac{\tau-r}\tau\,\e^{2K(\tau-r)}
\right]\,\dd r\\
&\leq&
\int_0^\tau\left[ \frac{N}{N_0}\frac{r}{\tau}\big(\e^{2\abs{K}\tau}-\e^{-2\abs{K}\tau}\big)
-\frac{r}{N}\frac{u}{u_r}\alpha_r\e^{-2\abs{K}\tau}\right.\\
&&\left.\qquad
+\frac\tau{2N}\left(\frac{u}{u_r}\right)^2\alpha_r\e^{2\abs{K}\tau}
+\frac{u}{u_r}\e^{2\abs{K}\tau} -\frac{N}{N_0}\e^{-2\abs{K}\tau}\right]\,\dd r\\
&=&
\frac{\tau N}{2 N_0}\big(\e^{2\abs{K}\tau}-\e^{-2\abs{K}\tau}\big)
+ \frac\tau4\left[\left(\frac{u}{u_0}\right)^2-1\right]\e^{2\abs{K}\tau} + \tau\e^{-2\abs{K}\tau}\left(1-\frac{N}{N_0}\right) \\
&&\qquad
+\left(\e^{2\abs{K}\tau}-\e^{-2\abs{K}\tau}\right)\int_0^\tau\frac{u}{u_r}\dd r\;.
\end{eqnarray*}
Here we used that by Lemma~\ref{lem:alphabeta-est} $\a_r \ge 0$, by
Lemma~\ref{lem:calculus2} $\partial_r
\frac1{u_r}=-\frac1{N\,u_r}\alpha_r$ and thus
$$0>-\int_0^\tau \frac r N\frac{u}{u_r}\alpha_r\,\dd r= \tau- \int_0^\tau\frac{u}{u_r}\,\dd r $$
and
$$\frac1N\int_0^\tau \left(\frac{u}{u_r}\right)^2\alpha_r\,\dd r= \frac12\left[\left(\frac{u_\tau}{u_0}\right)^2-1\right].$$
Since $(\rho_s)$ is regular, $\abs{\ent(\rho_s)}$ and the second
moments of $( \rho_s )_{s \in [0,1]}$ are uniformly bounded. Arguing
as in the proof of Lemma~\ref{lem:welldef-tau} and using that
$\tau_{s,t}\leq c\cdot st$ we find that $\frac{u}{u_r}$ is
bounded. Taylor expansion of the exponentials in the estimate above
thus yields, that for some constant $C_2$, depending only on $K$ and
the $\max_{s\in[0,1]}\abs{\ent(\rho_s)}$,
\begin{eqnarray*}
  (G) &\leq& \frac\tau4\left[\left(\frac{u_\tau}{u_0}\right)^2 - 1 - 4\left(\frac{N}{N_0}-1\right) \right] + C_2\tau^2\;.
\end{eqnarray*}
To control $(H)$ we estimate using Young inequality for any $\delta>0$:
\begin{eqnarray*}
  \gamma^{(2)} &\le& \frac\delta 8 I(\rho_{s,\tau}) + \frac{1}{2\delta}\int q^2_\eps(f_{s,\tau})\wug{\phi_s}^2\dd m\;,\\
  \gamma^{(1)} &\le& \frac\delta 8 I(\rho_{s,r}) + \frac{1}{\delta}\int q^2_\eps(f_{s,r})\Big(\wug{H_{\tau-r}\phi_s}^2 +t^2u^2 \wug{H_{\tau-r}g^\eps_{s,\tau}}^2\Big)\dd m\;.
\end{eqnarray*}
Note that $q_\eps^2(r)\leq 4r$, $q^2_\eps(r)\to0$ as $\eps\to0$. Using
the gradient estimate $\bl(K,\infty)$, \eqref{eq:alpha-est} and
\eqref{eq:Lip-HL} we estimate
\begin{align*}
  \int f_{s,r}\Big(\wug{H_{\tau-r}\phi_s}^2 +t^2u^2 \wug{H_{\tau-r}g^\eps_{s,\tau}}^2\Big)\dd m ~&\leq~ \e^{-2K(\tau -r)}\int f_{s,\tau}\Big(\wug{\phi_s}^2 +t^2u^2 \wug{g^\eps_{s,\tau}}^2\Big)\dd m\\
~&\leq~\e^{-2K(\tau -r)}\Big(4\Lip(\phi)^2 + t^2u^2I(\rho_{s,\tau})\Big)~<~\infty\;.
\end{align*}
Thus, dominated convergence yields that $\gamma^{(1)}\leq (\delta/8)
I(\rho_{s,r}) + O(\eps)$ and $\gamma^{(2)}\leq (\delta/8)
I(\rho_{s,\tau}) + O(\eps)$. It remains to estimate $\a,\b$. By
Lemma~\ref{lem:alphabeta-est} and \eqref{eq:Lip-HL} we have $\a\leq
I(\rho_{s,\tau})$ and $\b\leq
2\Lip(\phi)\sqrt{I(\rho_{s,\tau})}$. Note that combining
\eqref{eq:EDE}, \eqref{eq:slope-FI} and $K$-contractivity of the heat
flow we have $I(\rho_{s,r})\leq \e^{-Kr}I(\rho_s)$ for any $r\geq0$.

Putting everything together we conclude that there exist constants
$C_1,C_3$ depending on $K$, $\max_{s\in[0,1]}\abs{\ent(\rho_s)}$,
$\max_{s\in[0,1]}I(\rho_s)$ and $\phi$ such that
\begin{align*}
&\int\varphi_1\dd\rho_{1,\tau^\eps}-\int\varphi_0\dd\rho_{0}-\frac12\int_0^1|\dot\rho_s|^2\e^{-2K\tau^\eps}\dd s+Nt\cdot\left[ U^\eps_N(\rho_0)-U^\eps_N(\rho_{1,\tau^\eps})\right]\\
\leq~&\int_0^1C_1\frac{\tau^\eps}{4}\left[\left(\frac{U^\eps_N(\rho_{s,\tau^\eps})}{U^\eps_N(\rho_{s})}\right)^2-1-4\left(\frac{N}{N_0}-1\right)+C_2\tau^\eps\right]\,\dd s + C_3\delta + O(\eps)\;,
\end{align*}
where we have made the dependence of $\tau$ and $u$ on $\eps$
explicit. Finally, passing to the limit first as $\eps\to0$ and then
as $\delta \to0$ yields \eqref{eq:key-est}.
\end{proof}

\begin{proposition}\label{prop:green}
  Assume that $(X,d,m)$ satisfies $\bl(K,N)$. Then for each geodesic
  $(\rho_s)_{s\in[0,2]}$ in $\cP_2(X,d,m)$ with $\rho_0,\rho_2\in
  D(\ent)$ and $r\in[0,2]$ we have
\begin{align}\label{eq:key-est2}
  U_N(\rho_r)~\ge~\frac{2-r}2 U_N(\rho_0)+\frac{r}2 U_N(\rho_2)+ \frac KN |\dot\rho|^2\cdot\int_0^2 g(s,r)U_N(\rho_s)\,\dd s
\end{align}
where $g(s,r)=\frac12\min\{s(2-r),r(2-s)\}$ denotes the Green
function on the interval $[0,2]$.
\end{proposition}

\begin{proof}
  We will only prove \eqref{eq:key-est2} for $r=1$ the general
  argument being very similar. Obviously, it is sufficient to prove
  that the inequality \eqref{eq:key-est2} is satisfied with $N$
  replaced by $N'$ for any $N'>N$ and then let $N'\to N$. So let us
  fix $N'>N$ and a geodesic $(\rho_s)_{s\in[0,2]}$ in
  $\cP_2(X,d,m)$. Since we already know that $(X,d,m)$ is a strong
  $\cd(K,\infty)$ space we have that $s\mapsto\ent(\rho_s)$ is
  $K$-convex and thus continuous.

  Using Lemma~\ref{lem:reg-curves} we approximate the geodesic
  $(\rho_s)_{s\in[0,2]}$ by regular curves
  $(\rho_s^n)_{s\in[0,2]}$. Given $t>0$, the estimate
  \eqref{eq:key-est} from Proposition~\ref{prop:action-est}, with
  $N_0,N$ replaced by $N,N'$, holds true for each of the regular
  curves $(\rho^n_s)_{s\in[0,1]}$ and $(\rho^n_{2-s})_{s\in[0,1]}$ and
  any Lipschitz function $\phi$ with bounded support. From the uniform
  convergence \eqref{eq:reg-curves4} in Lemma~\ref{lem:reg-curves} and
  \eqref{eq:bd-tau} we conclude that for all $n$ large enough and $t$
  sufficiently small and all $s\in[0,1]$:
  \begin{align*}
  \left[\left(\frac{U_{N'}(\rho^n_{s,\tau^n})}{U_{N'}(\rho^n_s)}\right)^2-1 + C_2\tau^n\right]~\leq~4\left(\frac{N'}{N}-1\right)\;,
  \end{align*}
  i.e. the right hand side of \eqref{eq:key-est} is
  non-positive. Hence we obtain
  \begin{align*}
    \int\varphi_1\dd\rho^n_{1,\tau^n}-\int\varphi_0\dd\rho^n_{0}-\frac12\int_0^1|\dot{\rho^n_s}|^2\e^{-2K\tau^n}\,\dd s
    ~\leq~N't\cdot\left[U_{N'}(\rho^n_{1,\tau^n})- U_{N'}(\rho^n_0)\right]\;,
  \end{align*}
  for all such $n$ and $t$. Taking the supremum over $\phi$ yields by Kantorovich duality
  \begin{align*}
    \frac12 W_2^2(\rho^n_0,\rho^n_{1,\tau^n})-\frac12\int_0^1|\dot{\rho^n_s}|^2\e^{-2K\tau^n}\,\dd s
    ~\leq~N't\cdot\left[U_{N'}(\rho^n_{1,\tau^n})- U_{N'}(\rho^n_0)\right]\;,
  \end{align*}

  As $n\to\infty$, using the continuity properties
  \eqref{eq:reg-curves1}-\eqref{eq:reg-curves3} we obtain the same
  estimate for the geodesic $(\rho_s)_{s\in[0,1]}$.
  \begin{align*}
    \frac12 W_2^2(\rho_0,\rho_{1,\tau})-\frac12 W_2^2(\rho_0,\rho_1)\cdot\int_0^1 e^{-2K\tau}\dd s ~\le~
    N't\cdot\left[ U_{N'}(\rho_{1,\tau})-U_{N'}(\rho_{0})\right]\dd s\;.
  \end{align*}
  An analogous estimate holds true for the geodesic
  $(\rho_{2-s})_{s\in[0,1]}$
  \begin{align*}
    \frac12 W_2^2(\rho_2,\rho_{1,\tau})-\frac12 W_2^2(\rho_2,\rho_1)\cdot\int_1^2 e^{-2K\tau}\dd s~\le~
    N't\cdot\left[ U_{N'}(\rho_{1,\tau})-U_{N'}(\rho_{2})\right]\dd s\;.
   \end{align*}
   Moreover, since $(\rho_s)_{s\in[0,2]}$ is a geodesic
   \begin{align*}
     \frac12 W_2^2(\rho_0,\rho_{1})+\frac12 W_2^2(\rho_2,\rho_{1})-\frac12
     W_2^2(\rho_0,\rho_{1,\tau})-\frac12
     W_2^2(\rho_2,\rho_{1,\tau})~\le~0\;.
   \end{align*}

   Adding up the last three inequalities (and dividing by $t$) yields
   \begin{align*}
     \frac18 W_2^2(\rho_0,\rho_2)\cdot \frac1t\left[2-\int_0^1
       e^{-2K\tau}\dd s -\int_1^2 e^{-2K\tau}\dd s\right]~\le~
     N'\cdot\Big[2U_{N'}(\rho_{1,\tau})-
     U_{N'}(\rho_0)-U_{N'}(\rho_2)\Big]\dd s\;.
 \end{align*}
 Lower semi-continuity of the entropy implies that in the limit
 $t\to0$ the RHS will be bounded from above by
 \begin{align*}
   N'\cdot\left[2U_{N'}(\rho_{1})- U_{N'}(\rho_0)-U_{N'}(\rho_2)\right]\;.
 \end{align*}

Finally, by the very definition of $\tau$,
\begin{eqnarray*}
\lim_{t\to0}\frac1t\left[2-\int_0^1 e^{-2K\tau}\dd s -\int_1^2 e^{-2K\tau}\dd s\right]&=&-2K\, \int_0^2\partial_t \tau_{s,t}\,\dd s\\
&=&-2K\left[\int_0^1 s U_{N'}(\rho_s)\dd s+\int_1^2 (2-s) U_{N'}(\rho_s)\dd s\right]\\
&=&-4K\int_0^2 g(s,1)\, U_{N'}(\rho_s)\,\dd s.
\end{eqnarray*}
Thus we end up with
\begin{eqnarray*}
-\frac K 2 W_2^2(\rho_0,\rho_2)\cdot \int_0^2 g(s,1)\, U_{N'}(\rho_s)\,\dd s~\le~ N'\cdot\Big[2U_{N'}(\rho_{1})- U_{N'}(\rho_0)-U_{N'}(\rho_2)\Big]\;.
\end{eqnarray*}
Since $|\dot\rho|^2 =W_2^2(\rho_0,\rho_2)/4$,
this proves the claim.
\end{proof}

\begin{remark} A simple rescaling argument yields that for each
  geodesic $(\rho_s)_{s\in[0,1]}$ in $\cP_2(X,d,m)$ with
  $\rho_0,\rho_1\in D(\ent)$ and $r\in[0,1]$:
\begin{align} \label{eq:ecdknG0}
  U_N(\rho_{r})~\ge~(1-r)\cdot U_N(\rho_0)+r\cdot U_N(\rho_1)+ \frac KN |\dot\rho|^2\cdot\int_0^1 g\left(s,r\right)U_N(\rho_s)\,\dd s
\end{align}
where $g(s,r)=\min\{s(1-r),r( 1-s)\}$ now denotes the Green function
on the interval $[0,1]$.
% Moreover, a slight modification of the
% previous argument allows to prove that
% \begin{align} \label{eq:ecdknG}
% U_N(\rho_{r})~\ge~(1-r) U_N(\rho_0)+r U_N(\rho_1)+ \frac KN |\dot\rho|^2\cdot\int_0^1 g(s,r)U_N(\rho_s)\,\dd s
% \end{align}
% for each $r\in[0,1]$.
% \textcolor{red}{
% This extension also follows from \eqref{eq:ecdknG0} directly
% by using the upper semi-continuity of $U_N$,
% which is ensured by \eqref{eq:exp-int}
% (see Remark~\ref{rem:exp-int}).
% }
\end{remark}

\begin{theorem}\label{thm:BEW2CDE}
  Let $(X,d,m)$ be a infinitesimally Hilbertian mms satisfying the
  exponential integrability condition \eqref{eq:exp-int} and
  $\bl(K,N)$. Then the strong $\ecdkn$ condition holds. In particular,
  $(X,d,m)$ is a $\rcdkn$ space and the heat flow satisfies
  $\evi_{K,N}$.
\end{theorem}

\begin{proof}
  By virtue of Lemma~\ref{lem:knconvexint2}, this is merely a
  consequence of Proposition \ref{prop:green} and \eqref{eq:ecdknG0}.
\end{proof}

\begin{remark}
  In the special case $K=0$ it turns out to be possible to derive the
  $\evi_{0,N}$ property directly from the action estimate in
  Proposition~\ref{prop:action-est}. Let us give an alternative
  argument in this case.

  We want to show that for any $\rho,\sigma\in\cP_2(X,d)$ we have for all $t>0$:
  \begin{align}\label{eq:BEW2EVI}
    \ddtr\frac12 W_2^2(H_t\rho,\sigma)~\le~ N\cdot\left[
      1-\frac{U_N(\sigma)}{U_N(H_t\rho)}\right].
  \end{align}
  Obviously, it is sufficient to prove that \eqref{eq:BEW2EVI} is
  satisfied for any $N'>N$ and then let $N'\to N$. Moreover, by the
  semigroup property and Proposition~\ref{prop:evi-equiv} it is
  sufficient to assume that $\rho,\sigma\in D(\ent)$ and show that
  \eqref{eq:BEW2EVI} holds at $t=0$. So let us fix $N'>N$ and a
  geodesic $(\rho_s)_{s\in[0,1]}$ in $\cP_2(X,d,m)$ connecting
  $\rho_0=\sigma$ to $\rho_1=\rho$. Since we already know that $(X,d,m)$ is a strong
  $\cd(0,\infty)$ space we have that $s\mapsto\ent(\rho_s)$ is convex
  and thus continuous. By approximating the geodesic $(\rho_s)$ by
  regular curves one can show as in the proof of Proposition~\ref{prop:green} that
  \begin{align*}
    \frac1t\left[\frac12 W_2^2(\rho_0,\rho_{1,\tau})-\frac12 W_2^2(\rho_0,\rho_{1})\right]~\leq~N'\cdot\left[ U_{N'}(\rho_{1,\tau})-U_{N'}(\rho_{0})\right]\;.
  \end{align*}
  Thus passing to the limit $t\to0$ yields
  \begin{eqnarray*}
    \ddtr\frac12 W_2^2(\rho_0,H_t\rho_1)\Big|_{t=0}\cdot \frac{d}{dt}\tau_{1,t}\Big|_{t=0}~=~
    \ddtr\frac12 W_2^2(\rho_0,H_{\tau_{1,t}}\rho_{1})\Big|_{t=0}~\le~  N'\cdot\left[ U_{N'}(\rho_{1})-U_{N'}(\rho_0)\right]\;.
  \end{eqnarray*}
  Since $\frac{d}{dt}\tau_{1,t}\Big|_{t=0}=U_{N'}(\rho_1)$, this
  finally yields the $\evi_{0,N'}$ inequality:
  \begin{eqnarray*}
    \ddtr\frac12 W_2^2(\rho_0,H_t\rho_1)\Big|_{t=0}~\le~  N'\cdot\left[ 1-\frac{U_{N'}(\rho_0)}{U_{N'}(\rho_1)}\right]\;.
  \end{eqnarray*}
\end{remark}

To finish this section let us consider the classical case of weighted
Riemannian manifolds. More precisely, let $(M,d)$ be a $n$-dimensional
smooth, complete Riemannian manifold and let $V:M\to\R$ be a smooth
function bounded below. Consider the metric measure space
$(M,d,\e^{-V}\text{vol})$. The associated weighted Laplacian is given
by
\begin{align*}
  Lu~=~\Delta u -\nabla V\cdot\nabla u\;.
\end{align*}
It is well known (see e.g. \cite[Thm.~14.8]{Vil09}) that the operator
$L$ satisfies the Bakry--\'Emery condition $\be(K,N)$ if and only if
the generalized Ricci tensor
\begin{align*}
\Ric_{N,V}~:=~\Ric +\Hess V - \frac{1}{N-n}\nabla V\otimes\nabla V
\end{align*}
is bounded below by $K$. As an immediate consequence of our
equivalence result we thus obtain the following
\begin{proposition}\label{prop:rcdkn-mfds}
The mms $(M,d,\e^{-V}\text{vol})$ satisfies the $\ecdkn$-condition if and
only if
  \begin{align*}
    \Ric + \Hess V~\geq~ K + \frac{1}{N-n}\nabla V\otimes\nabla V\;.
  \end{align*}
\end{proposition}

\subsection{The sharp Lichnerowicz inequality (spectral gap)}
\label{sec:Lichnerowicz}

Here we provide a first application of the Bochner formula on
infinitesimally Hilbertian metric measure spaces. Namely we establish
the sharp spectral gap estimate on $\rcdkn$ spaces in the case of
positive curvature $K>0$.

We consider an infinitesimally Hilbertian metric measure space
$(X,d,m)$. Recall that we denote by $\Delta$ the canonical Laplacian
on $(X,d,m)$, i.e. the generator of the heat semigroup in $L^2$ which
is given as the $L^2$-gradient flow of the Cheeger energy $\ch$, see
Section~\ref{sec:recap}.

\begin{theorem}[Spectral gap estimate]\label{thm:lichnerowitz}
  Let $(X,d,m)$ be a mms satisfying the Riemannian curvature dimension
  condition $\rcdkn$ with $K>0$ and $N>1$. Then the spectrum of
  $(-\Delta)$ is discrete and the first non-zero eigenvalue
  $\lambda_1(X,d,m)$ satisfies the following bound:
  \begin{align}\label{eq:lichnerowitz}
    \lambda_1(X,d,m)~\geq~\frac{N}{N-1}K\;.
  \end{align}
\end{theorem}

\begin{proof}
  First
  recall that the $\rcdkn$ condition with $K>0$ implies that $(X,d,m)$
  is doubling by Proposition~\ref{prop:BGI} and compact by Corollary~\ref{cor:BMT}.
  In combination with the result in \cite{Ra12} this
  yields that $(X,d,m)$ supports a global Poincar\'e
  inequality. Moreover, the $\cdskn$ condition implies a global
  Sobolev inequality, by adapting \cite[Thm.~30.23]{Vil09}. These
  ingredients yield the following Rellich--Kondrachov compactness
  property(c.f. \cite[Thm.~8.1]{HK00}): for any sequence of
  functions $(f_n)_n\subset W^{1,2}(X,d,m)$ with
  \begin{align*}
    \sup\limits_n\left( \norm{f_n}_{L^2(X,m)} + \ch(f_n)\right)~<~\infty
  \end{align*}
  we have that up to extraction of a subsequence $f_n\to f$ in
  $L^2(X,m)$ for some $f\in L^2(X,m)$. This compactness theorem is
  sufficient to prove that the spectrum of $(-\Delta)$ is discrete,
  e.g. by following verbatim the proof in \cite{Ber86} of the
  corresponding result for Riemannian manifolds.

  For the eigenvalue estimate we follow the argument in
  \cite{Dav90}. Let $\lambda>0$ be a non-zero eigenvalue of
  $(-\Delta)$ and let $\psi\in D(\Delta)$ be a corresponding
  eigenfunction. We apply the Bochner inequality of Theorem~\ref{thm:Bochner}
  to $f=\psi$ and the test function $g\equiv
  1$. Note that this pair is admissible since $X$ is compact. Thus we
  obtain using the integration by parts formula
  \eqref{eq:int-by-parts}:
  \begin{align*}
    0~&\geq~\int\ip{\nabla(\Delta\psi),\nabla\psi}\dd m + K\int\wug{\psi}^2\dd m + \frac{1}{N}\int(\Delta\psi)^2\dd m\\
      &=~(K-\lambda)\int\wug{\psi}^2\dd m - \frac{\lambda}{N}\int \psi\Delta\psi\dd m\\
      &=~\left(K-\lambda+\frac{\lambda}{N}\right)\int\wug{\psi}^2\dd m\;.
  \end{align*}
  Since $\ch(\psi)>0$ it follows that $\lambda\geq KN/(N-1)$ which
  yields the claim.
\end{proof}

Note that this estimate of the spectral gap is sharp. This can be seen
by considering the model space
\begin{align*}
X=(-\frac{\pi}{2}\sqrt{\frac{N-1}{K}},\frac{\pi}{2}\sqrt{\frac{N-1}{K}})\;,\quad d(x,y)=\abs{x-y}\;,\quad  m(\dd x)=\cos\left(x\sqrt{\frac{K}{N-1}}\right)^{N-1}\dd x\;.
\end{align*}
The corresponding operator is given by
\begin{align*}
  Lf(r)=f''(r) - \sqrt{K(N-1)}\tan\left(r\sqrt{K/(N-1)}\right)f'(r)
\end{align*}
with Neumann boundary conditions. By Proposition~\ref{prop:rcdkn-mfds}
the metric measure space $(X,d,m)$ satisfies $\rcdkn$. It is well
known that the first non-zero eigenvalue of the Neumann problem
associated to $L$ is given by $KN/(N-1)$.

\section{Dirichlet form point of view}
\label{sec:dirichlet}

Up to now we have formulated our results in the setting of metric
measure spaces. Here the Cheeger energy, if assumed to be a quadratic
form, gives rise to a canonical Dirichlet form. In this final section
we take a different point of view and reformulate our results starting
from a Dirichlet form. The relation between the two points of view and
the compatibility of metric measure structures and Energy structures
has been discussed extensively in \cite{AGS12} as well as in \cite{KZ13}.

Let $X$ be a Polish space and let $m$ be a locally finite Borel measure on
$X$. Let $\cE$ be a strongly local Dirichlet form on $L^2(X,m)$ with
domain $D(\cE)$. Denote the associated Markov semigroup in $L^2(X,m)$
by $(P_t)_{t>0}$ and its generator by $\Delta$. Given a function $f\in
D(\cE)$ we denote by $\Gamma(f)$ the associated energy measure
defined by the relation
\begin{align*}
  \int \phi\dd\Gamma(f)~=~\cE(f,f\phi) - \frac12\cE(f^2,\phi)\quad\forall \phi\in D(\cE)\cap L^\infty(X,m)\;.
\end{align*}
If $\Gamma(f)$ is absolutely continuous w.r.t. $m$ we will also denote
its density with $\Gamma(f)$.  The natural notion of a
(pseudo-)distance on $X$ associated to $\cE$ is the intrinsic $d_\cE$
defined by
\begin{align*}
  d_\cE(x,y)~:=~\sup\left\{\abs{f(x)-f(y)}\ :\ f\in D(\cE)\cap C(X), \Gamma(f)\leq m\right\}\;.
\end{align*}
For the sequel, assume that $d_\cE$ is a finite, complete distance on $X$ inducing
  the given topology and assume that $( X, d, m, \cE )$ is upper
  regular energy measure space in the sense of \cite[Def.3.6,
  Def.~3.13]{AGS12}.
  
\begin{corollary}\label{cor:Dirichlet}
  Under the previous assumptions,  the following are equivalent:
  \begin{itemize}
  \item[(i)] Assumption~\ref{ass:Ch-reg} and $\bl(K,N)$ holds,
  i.e. for any $f\in D(\cE)$ with
    $\Gamma(f)\leq m$ and $t>0$, $f$ is 1-Lipschitz and
    \begin{align*}
      \abs{\Gamma P_tf}^2 + \frac{1-\e^{-2Kt}}{NK}\abs{\Delta P_t
        f}^2~\leq~e^{-2Kt}P_t\Gamma(f)\;.
    \end{align*}
  \item[(ii)] $(X,d_\cE,m)$ is an $\rcdkn$ space.
  \end{itemize}
\end{corollary}

\begin{proof}
  Under the assumptions on $d_\cE$ and $\cE$, it is shown in \cite[Thm.~3.14]{AGS12}
   that $\cE$ coincides with the Cheeger energy on
  $(X,d_\cE,m)$. Thus $(X,d_\cE,m)$ is infinitesimally Hilbertian and
  for any $f\in D(\cE)$ we have $\Gamma(f)\ll m$ with density
  $\wug{f}^2$. The equivalence of (i) and (ii) then follows from
  Theorems~\ref{thm:BEW2CDE}, \ref{thm:grad-est}.
\end{proof}

\begin{remark}
  According to \cite[Cor.~2.3]{AGS12} conditions (i) and (ii) of the
  previous result are in turn equivalent to the Bakry--\'Emery
  inequality $\Gamma_2(f)\geq K\Gamma(f)+\frac{1}{N}(\Delta f)^2$ in
  the form of $\be(K,N)$, see Definition~\ref{def:BE}.
\end{remark}

\medskip

\noindent
{\bf Note added in proof.}~~\it{ 
  Since the first version of this article was published on arxiv, several remarkable follow-up papers appeared.
       Garofalo and Mondino have \cite{GM13} have established the Li--Yau
    estimates on metric measure spaces satisfying
    $\rcd^*(K,N)$. Contraction properties of the heat flow reflecting
    dimensional effects have  been exhibited by Bolley, Gentil and
    Guillin \cite{BGG13}, their approach however being very different from ours,
     based on  a new transportation distance instead of the
    $L^2$-Wasserstein distance. The concept of $(K,N)$-convexity has  been adopted by Naber
    \cite{Nab13} in the study of upper and lower Ricci bounds on
    metric measure spaces and the relation with spectral gaps on the
    associated path space 
    
    The authors also would
    like to mention the closely related, independent work in progress of
    Ambrosio, Mondino and Savar\'{e} \cite{AMS13}, where partly similar
    results as in the present article are obtained via a study of the
    porous medium equation in metric measure spaces.
    }

\bibliographystyle{plain}
\bibliography{evikn}

\end{document}